\documentclass[a4paper,11pt]{article}
\usepackage{graphics,color}
\usepackage{amsmath}
\usepackage{amssymb}
\usepackage{mathrsfs}
\usepackage{cite}
\usepackage{verbatim}
\usepackage{float}
\usepackage{graphicx}
\usepackage{amsthm}
\usepackage{textcomp}
\usepackage{subfig}
\usepackage{wrapfig}
\usepackage{hyperref}
\usepackage{enumitem}
\usepackage{yhmath}

\newif\ifdraft
\draftfalse 

\newcommand{\admclassconv}{\mathcal W}
\newcommand{\arc}{\wideparen}
\newcommand{\weakstar}{\stackrel{*}{\rightharpoonup}}
\def\comp{{\mathcal W_{\rm conv}}}
\def\compacts{\mathcal K}
\def\convexcurves{\curves_{\rm conv}}
\def\curves{\Sigma}

\newcommand{\di}{\, {\rm d}}
\newcommand{\Div}{\textrm{div}}
\newcommand{\dist}{\textrm{dist}}
\newcommand{\dOm}{\partial \Om}

\def\eps{\varepsilon}

\newcommand\grad{\nabla}

\def\H{\mathcal H}

\newcommand{\jump}[1]{\text{{\rm \textlbrackdbl}}{#1}\text{{\rm \textrbrackdbl}}}

\newcommand{\largercomp}{\mathcal W}
\newcommand{\len}{\ell}
\newcommand{\Lip}{{\rm Lip}}

\newcommand{\nada}[1]{}

\newcommand{\nullb}{\partial^0}
\newcommand{\Om}{\Omega}
\newcommand{\openannulus}{\Sigma_{\rm ann}}
\newcommand{\openhalfplane}{H^+}

\newcommand{\immclosedann}{\Phi({\overline\Sigma}_{\rm ann})}
\newcommand{\pointinOm}{P'}
\newcommand{\pointinspace}{P}

\newcommand{\sigmamin}{\sigma}
\newcommand{\psimin}{\psi}
\newcommand{\R}{\mathbb{R}}
\newcommand{\res}{\mathop{\hbox{\vrule height 7pt width 0.5pt depth 0pt
\vrule height 0.5pt width 6pt depth 0pt}}\nolimits}

\newenvironment{step}[1]{\medskip\textit{Step #1:}}{}
\numberwithin{equation}{section}
\mathchardef\emptyset="001F

\newtheorem{theorem}{Theorem}[section]
\newtheorem{definition}[theorem]{Definition}

\newtheorem{cor}[theorem]{Corollary}
\newtheorem{lemma}[theorem]{Lemma}

\theoremstyle{definition}
\newtheorem{remark}[theorem]{Remark}
\newtheorem{Remark}[theorem]{Remark}

\title{A non-parametric Plateau problem with partial free boundary}

\author{Giovanni Bellettini
	\footnote{
		Dipartimento di Ingegneria dell'Informazione e Scienze Matematiche, Universit\`a di Siena, 53100 Siena, Italy,
		and International Centre for Theoretical Physics ICTP,
		Mathematics Section, 34151 Trieste, Italy.
		E-mail: bellettini@diism.unisi.it
	}
\and
	Roberta Marziani\footnote{
Angewandte Mathematik, WWU M\"unster, Germany.
		E-mail: roberta.marziani@uni-muenster.de
	}
	\and
	Riccardo Scala\footnote{ 
		Dipartimento di Ingegneria dell'Informazione e Scienze Matematiche, Universit\`a di Siena, 53100 Siena, Italy.
		E-mail: riccardo.scala@unisi.it}
}

\begin{document}

\maketitle

\begin{abstract}
We consider a Plateau problem in codimension $1$ in the non-parametric setting. A Dirichlet boundary datum is given only on part of the boundary $\partial \Omega$ of a bounded convex domain $\Omega\subset\R^2$. Where the Dirichlet datum is not prescribed, 
we allow
a free contact with the horizontal plane. 
We show existence of a solution, 
and prove regularity for the corresponding minimal surface.
Finally we compare these solutions with the classical 
minimal surfaces of Meeks and Yau, and show that they 
are equivalent when  the Dirichlet boundary datum is assigned in at most $2$ disjoint arcs of $\partial \Omega$.
\end{abstract}


\vspace{2mm}

\noindent {\bf AMS (MOS) 2020 Subject Clas\-si\-fi\-ca\-tion:}  49J45, 49Q05, 49Q15, 28A75.

\section{Introduction}\label{sec:introduction}
Let $\Omega\subset \R^2$ 
be a bounded open convex set; 
in this paper we look for an area-minimizing surface 
which can be written as a graph over a subset of $\Omega$,
and spanning
a Jordan curve $\Gamma_\sigma=\gamma\cup \sigma \subset \R^2 \times [0,+\infty)$. 
Here $\gamma$ is fixed (Dirichlet condition) and
is given by a family 
$\{\gamma_i\}_{i=1}^n\subset\partial \Om\times [0,+\infty)$ 
 of $n\in\mathbb{N}$ 
curves 
each joining pairs of points $\{(p_i,q_i)\}_{i=1}^n$ 
of $\partial \Om$. 
Whereas $\sigma$, which represents 
the {\it free boundary}, consists of (the image of)
 $n$ curves $(\sigma_1,\dots,\sigma_n)$
sitting in the  plane containing $\Omega$ (also called free boundary plane), and
joining the  endpoints of $\gamma$ 
in order that $\gamma\cup\sigma$  forms a Jordan curve $\Gamma_\sigma$ 
in $\R^3$. We assume that each $\gamma_i$ 
is Cartesian, i.e., it can be expressed as the graph of a given nonnegative function $\varphi$ 
defined on a corresponding portion of 
$\partial \Omega$. This allows to restrict ourselves 
to the Cartesian setting, and to assume that the competitors 
for the Plateau problem are expressed by graphs of functions $\psi$ 
defined on a suitable subdomain of $\Omega$ depending on $\sigma$.

Our prototypical example is given by the catenoid. Consider a cylinder in $\R^3$ with a circle of radius $r$ as basis, and height $l$. 
Choose a system of Cartesian coordinates in which the $x_1x_2$-plane 
contains the cylinder axis, and restrict attention to the half-space 
$\{x_3\geq0\}$ as
 in Figure \ref{figura1},
where $\Omega = (0,l)\times(-r,r)$ and $n=2$.
Write
$$\partial \Omega=\partial^0_1\Omega\cup\overline{\partial^D_1\Om}\cup \partial^0_2\Omega\cup \overline{\partial^D_2\Om},$$
where  $\partial^0_1\Omega=(0,l)\times\{r\}$, $\partial^0_2\Omega=(0,l)\times\{-r\}$, $\partial^D_1\Om=\{0\}\times(-r,r)$, and $\partial^D_2\Om=\{l\}\times(-r,r)$. 
On the Dirichlet boundary $\partial^D\Om=\partial^D_1\Om\cup\partial^D_2\Om$ we prescribe a continuous function $\varphi$ whose graph 
consists of the two half-circles $\gamma_1$ and $\gamma_2$. The endpoints of $\gamma_1$ and $\gamma_2$ live on the 
free boundary plane (the horizontal plane) and are $p_1=(0,-r)$, $q_1=(0,r)$, and $p_2=(l,r)$, $q_2=(l,-r)$ respectively. The free boundary 
$\sigma$ consists of two curves $\sigma_1$ and $\sigma_2$ 
with endpoints $q_1,p_2$, and $q_2,p_1$, respectively,
constrained to stay in $\overline\Omega$. The concatenation of $\gamma=\gamma_1\cup\gamma_2$ and $\sigma$ 
forms a Jordan curve in $\R^3$
$$\Gamma_\sigma=\gamma_1\cup\sigma_1\cup\gamma_2\cup\sigma_2.$$
Therefore we proceed to look for an area-minimizer 
among all Cartesian surfaces $S$ with boundary $\Gamma_\sigma$ 
{\it keeping $\sigma$ free}, i.e. we minimize the area among all 
pairs $(\sigma,S)$. In this particular case
of the catenoid, a minimizing sequence $((\sigma_k,S_k))$ tends (in a suitable way specified in the sequel) to a minimizer $(\sigma,S)$ which allows for two different possibilities. If $l$ is small, $\sigma_1$ and $\sigma_2$ remain disjoint and the classical  catenoid (half of it, namely the intersection between the catenoid and the half-space $\{x_3\geq0\}$) is the surface $S$, in turn coinciding with the graph of a function $\psi$ defined on the region of $\Omega$ ``enclosed'' by $\sigma$. If instead $l$ is large, the two curves $\sigma_1$ and $\sigma_2$ merge and the region of $\Omega$ enclosed by $\sigma$ tends to become empty (it  reduces to the two segments $\partial_1^D\Omega\cup\partial_2^D\Omega$). This describes the solution given by two (half) discs.

\begin{figure}
	\begin{center}
		\includegraphics[width=0.5\textwidth]{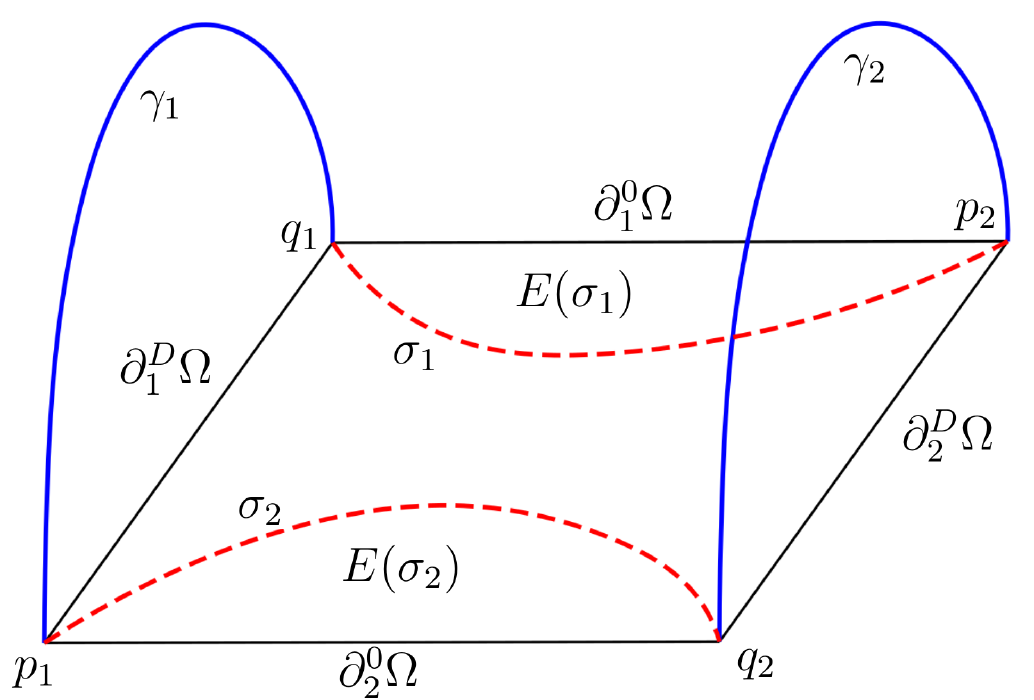}
		\caption{The catenoid: when $l$ is large 
enough the two dotted curves $\sigma_1$ and $\sigma_2$ merge and the (generalized) 
graph of $\psi$ reduces to two vertical half-circles on $\partial^D \Om = 
\partial_1^D\Om\cup \partial_2^D\Omega$.
In this case $\partial^D \Om \subset \partial E(\sigma_1)
\cup \partial E(\sigma_2)$.
}\label{figura1}
	\end{center}
\end{figure}  

	A peculiarity of our problem is the presence of a free boundary. The problem of Plateau with partial free boundary has been exhaustively studied (see for instance \cite{DHS}) but never investigated,
to our best knowledge, with the non-parametric approach.
	
	Referring to Section \ref{sec:preliminaries} 
	for the precise description of the mathematical framework, 
	here we just describe it with few details. 
	 We 
fix some distinct points 
$p_1,q_1,p_2,q_2,\dots,p_n,q_n \in \partial
\Om$ taken in clockwise order. The part of $\partial \Omega$ between the points $p_i$ and $q_i$ is noted by $\partial_i^D\Om$, and the part between $q_i$ and $p_{i+1}$ by $\partial_i^0\Om$. We fix a nonnegative 
continuous function $\varphi\colon\partial\Om\to[0,+\infty)$ which is positive on $\partial^D\Om=\cup_{i=1}^n\partial^D_i\Om$ and vanishes on $\{p_i,q_i\}_{i=1}^n\cup\partial^0\Om$ with $\partial^0\Om=\cup_{i=1}^n\partial^0_i\Om$ , and we consider Lipschitz injective  and 
mutually disjoint curves $\sigma_i$ in $\overline \Om$, $i=1,\dots,n$, 
joining $p_i$ to $q_{i+1}$. 
We suppose the graph of $\varphi$ on $\partial^D\Om$ to be a Lipschitz curve in $\R^3$. We define $E(\sigma):=\cup_{i=1}^nE(\sigma_i)$, with $E(\sigma_i)$ the planar closed region enclosed between $\partial^0_i\Om$ and $\sigma_i$.

We define the two classes
\begin{align}
&\widehat\Sigma:=\{\sigma=(\sigma_1,\sigma_2,\dots,\sigma_n):[0,1]\rightarrow \overline{\Om}^{n}\text{ be curves as above}\},\\
&\mathcal X_\varphi:=\{(\sigma,\psi)\in\widehat\Sigma\times  W^{1,1}(\Omega):\psi=0\text{ a.e.  in }E(\sigma)\text{ and }\psi=\varphi\text{ on }\partial^D\Omega\}.
\end{align}
We want to find a solution to the following minimum problem:
\begin{align}\label{problem1}
\inf_{(\sigma,\psi)\in \mathcal X_\varphi}\mathbb{A}(\psi;\Om\setminus E(\sigma)),
\end{align}
where $\mathbb A$ denotes the classical area integral, i.e., 
\begin{equation}\label{def:area_integral}
	\mathbb{A}(\psi;\Om\setminus E(\sigma)):=
\int_{\Omega\setminus E(\sigma)}\sqrt{1+|\nabla \psi|^2}dx.
\end{equation}

Since, in general, existence of minimizers is not guaranteed 
in the class $\mathcal X_\varphi$,  we need to formulate 
this problem to a more suited space of admissible pairs. Specifically, a standard relaxation procedure leads one to analyse the problem above for pairs $(\sigma,\psi)$ belonging to $\Sigma\times BV(\Omega)$, where $\Sigma$ is a suitable class containing $\widehat\Sigma$ but which also allows for partial overlapping of the curves $\sigma_i$ (a precise definition is given in Section \ref{subsec:setting_of_the_problem}). 
	Therefore we shall be concerned with 
the study of the functional $ \mathcal F_\varphi$ defined as 
\begin{align}\label{def:relaxed_functional_intro}
\mathcal F_\varphi(\sigma,\psi):=\mathcal A(\psi;\Omega)-|E(\sigma)|+\int_{\partial\Omega}|\psi-\varphi|d\mathcal H^1,
\end{align}
where $(\sigma,\psi)\in \mathcal W\subset\Sigma\times BV(\Omega)$,  $\mathcal W$ is the space of pairs $(\sigma,\psi)\in \Sigma\times BV(\Omega)$ such that $\psi=0$ a.e. on $E(\sigma)$, and  $\mathcal A(\psi;\Omega)$ is the relaxed area functional defined as in \eqref{def:relaxed_area}, which accounts for the area of the generalized graph of the map $\psi$ on $\Omega$.  The functional $\mathcal F_\varphi$ extends the area integral $\mathbb A$ to the larger class $\mathcal W$.

\begin{figure}
	\begin{center}
		\includegraphics[width=0.7\textwidth]{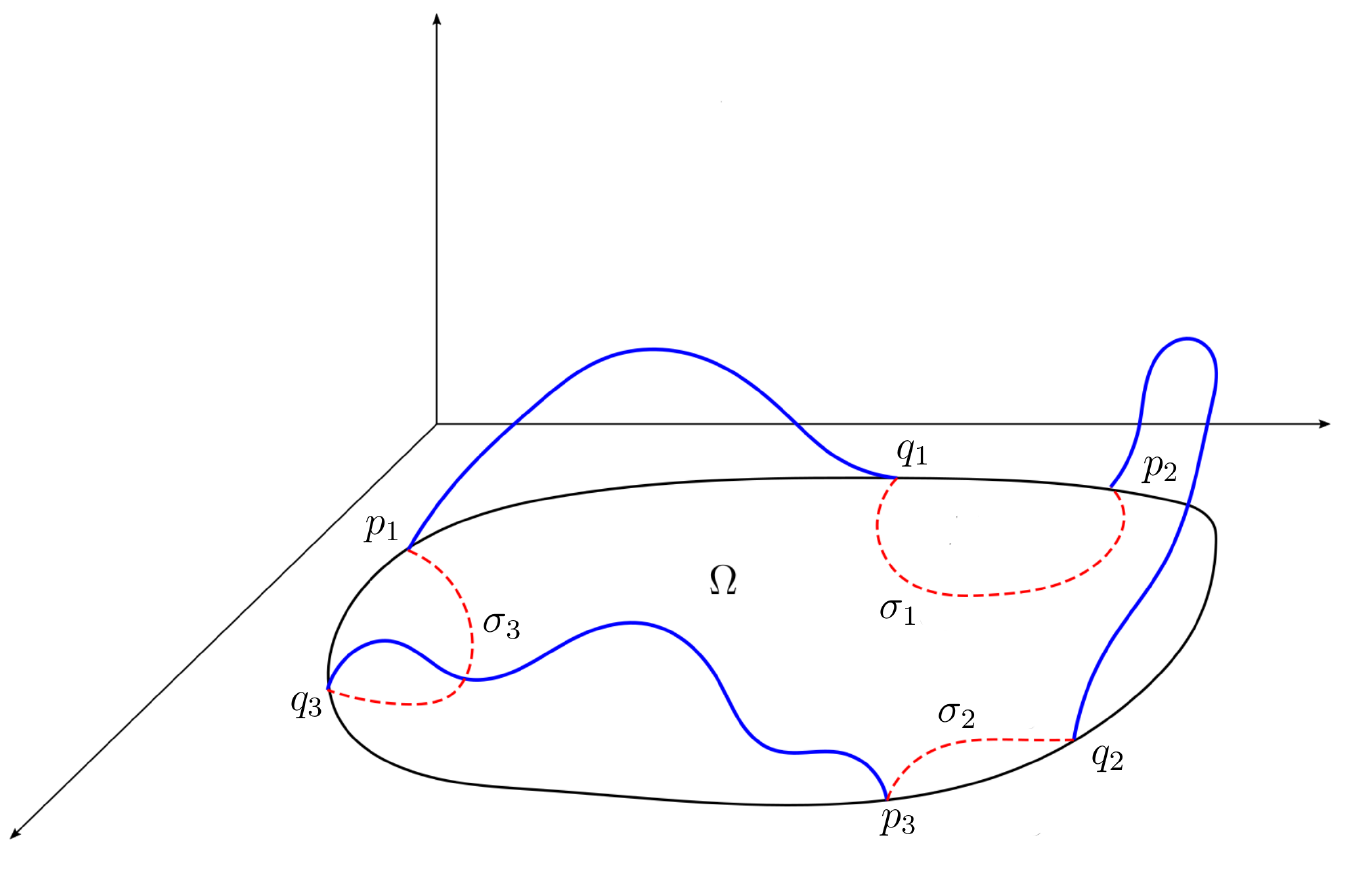}
		\caption{An example of the setting (in 3D), 
when $n=3$. On the boundary of the convex set $\Omega$ we have fixed the 
points $p_i$, $q_i$; the arc of $\partial \Omega$ joining $p_i$ to $q_i$ is $\partial_i^D\Om$, while the arc joining $q_i$ to $p_{i+1}$ 
is $\partial^0_i\Om$ ($p_4:= p_1$). On $\partial^D\Omega$ the Dirichlet boundary datum $\varphi$ is imposed, whose graph has been depicted. The dotted arcs are the free planar curves $\sigma_i$ joining the pairs $(q_i,p_{i+1})$. }\label{figura3}
	\end{center}
\end{figure}

 We then prove the following result, accounting for existence and 
regularity of minimizers of $\mathcal F_\varphi$.
 
 \begin{theorem}\label{teo_main_intro}
There exists a minimizer of $\mathcal F_\varphi$ on $\mathcal W$. Moreover, any minimizer $(\sigma,\psi)\in\mathcal W$ of $\mathcal F_\varphi$ satisfies the following regularity properties:
 \begin{itemize}
 	\item[$(1)$] The region $E(\sigma)$ consists of a family of closed convex sets. The boundary $\partial E(\sigma)$ is given by the union of the arcs $\partial^0\Omega$ 
and a family of disjoint Lipschitz curves in $\overline \Om$ (joining 
the points $p_i$ and $q_j$, in some order). 
Moreover, if $\partial^D_i\Omega$ is not a straight segment, then $\partial^D_i\Omega\cap \partial E(\sigma)=\emptyset$. If instead $\partial^D_i\Omega$ is  a straight segment, then either $\partial^D_i\Omega\cap \partial E(\sigma)=\emptyset$ or $\partial^D_i\Omega\cap \partial E(\sigma)=\partial^D_i\Omega$.
 	\item[$(2)$] The function $\psi$ is real analytic in $\Omega\setminus E(\sigma)$, and is continuous on $\partial^D\Omega\setminus \partial E(\sigma)$ where it attains  the boundary value $\psi=\varphi$.
 	\item[$(3)$] If  $\Om\cap\partial E(\sigma)\ne\emptyset$,
 	there is at least a  minimizer $(\sigma,\psi)$ such that $\psi$ is continuous and null on $\Om\cap\partial E(\sigma)$, and moreover $\Om\cap\partial E(\sigma)$ consists of a family of mutually disjoint smooth curves (joining $p_i$ and $q_j$ in some order).

 \end{itemize}
 \end{theorem}
A comparison with classical solutions of the Plateau problem in parametric form is in order.
Denoting by $\gamma_i$ the graph of the map $\varphi$ on $\overline{\partial_i^D\Om}$, we consider also ${\rm sym}(\gamma_i)$, namely the graph of $-\varphi$ on $\overline{\partial_i^D\Om}$, which is symmetric to $\gamma_i$ with respect to the plane containing $\Omega$. Setting $\Gamma_i:=\gamma_i\cup {\rm sym}(\gamma_i)$, this turns out to be a simple Jordan curve in $\R^3$, for all $i=1,\dots,n$.
Hence we can consider the classical Plateau problem for the curves $\Gamma_i$. In the case $n=1$ it is intuitive that a disc-type minimal surface $S$ spanning $\Gamma=\Gamma_1$ will be symmetric with respect to the plane containing $\Omega$, and that $S^+:=S\cap \{x_3\geq0\}$ will be a minimal disc with partial free boundary on $\Omega$. It is interesting to compare such a minimal disc with the graph of $\psi$, where $(\sigma,\psi)\in \mathcal W$ is a minimizer as in Theorem \ref{teo_main_intro}. Actually, in this simple case $n=1$, it is not difficult to see that $S^+$ is Cartesian, and it is the graph of a function $\psi$ which is positive outside the convex region $E(\sigma)$ enclosed by $\sigma$ and $\partial^0\Om$, and further $(\sigma,\psi)$ is a minimizer as provided by Theorem \ref{teo_main_intro}. Also the converse is true: Any minimizer $(\sigma,\psi)$ that satisfies (1)-(3) of Theorem \ref{teo_main_intro} has as graph of $\psi$ a disc-type surface $S^+$ whose double $S=S^+\cup S^-$ is a classical solution to the Plateau problem for the curve $\Gamma$.

This result is rigorously stated in Theorem \ref{plateau:n1} of Section 
\ref{sec:the_case_n=1}. In Section \ref{sec:6.2} we instead analyse the case
$n=2$. In this case one might look for minimal surfaces obtained as union of two discs spanning $\Gamma_1$ and $\Gamma_2$, or else for a catenoid-type surface spanning $\Gamma=\Gamma_1\cup\Gamma_2$ together. 
Appealing to an existence result due to Meeks and Yau \cite{MY}, we are able to show the counterpart of Theorem \ref{plateau:n1}: Theorem \ref{thm:comparison-with-classical-plateau}, that essentially states that any minimizer  $(\sigma,\psi)\in\mathcal W$ of $\mathcal F_\varphi$ satisfying  properties (1)--(3) of Theorem \ref{teo_main_intro} is (the nonnegative half of) a Meeks-Yau solution, and vice-versa.
In order to prove Theorem  \ref{thm:comparison-with-classical-plateau} we will strongly use the convexity of the domain $\Omega$, which implies that the cylinder $\Omega\times \R$, which contains $\Gamma$ on its boundary, is convex, and so the results of Meeks and Yau are applicable.

Due to the highly nontrivial arguments used to prove this result, we restrict our analysis to the case $n=2$, since a generalization to the case $n>2$ probably requires heavy modifications. 
Indeed, some of the lemmas needed to prove Theorem \ref{thm:comparison-with-classical-plateau} employ crucially the fact that $\partial^0\Om$ consists of only two connected components.
For this reason we leave the case $n>2$ for future investigations.
\medskip

\begin{figure}
	\begin{center}
		\includegraphics[width=0.5\textwidth]{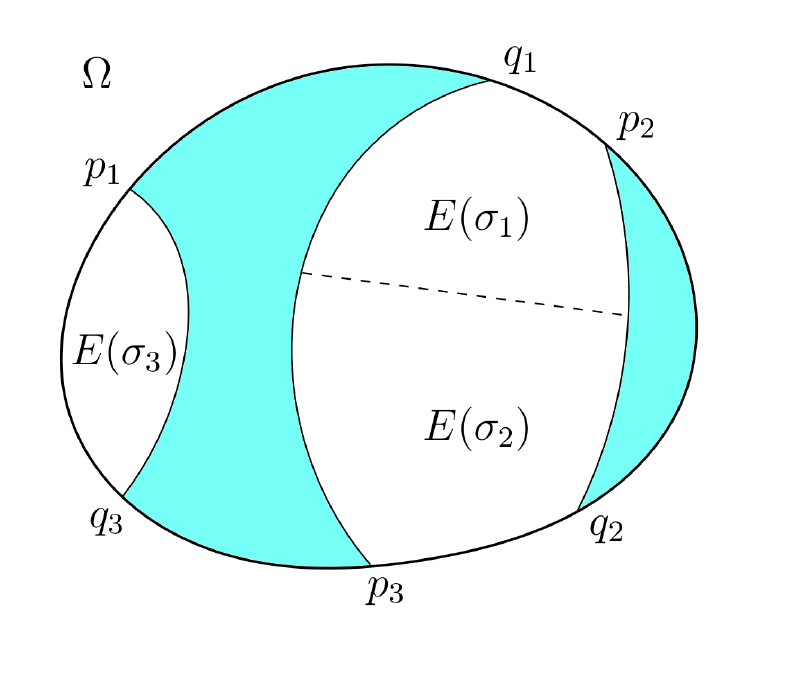}
		\caption{
A possible configuration of the sets $E(\sigma_i)$  (in the relaxed problem). Also in this example  $n=3$. The (clockwise oriented) arcs $\arc{p_1q_1}=\partial^D_1\Omega$, $\arc{p_2q_2}=\partial^D_2\Omega$, and $\arc{p_3q_3}=\partial^D_3\Omega$ are the set where $\varphi$ is prescribed and positive. In the set $\partial^0\Omega=\arc{q_1p_2}\cup\arc{q_2p_3}\cup\arc{q_3p_1}$ and on $E(\sigma)=E(\sigma_1)\cup E(\sigma_2)\cup E(\sigma_3)$ we prescribe $\psi=0$.  The  curves $\sigma_i$ joining $q_i$ to $p_{i+1}$ (with the corresponding set $E(\sigma_i)$) are indicated. On the dotted segment $\sigma_1$ and $\sigma_2$ overlaps with opposite orientations. The emphasized
region $\Omega\setminus E(\sigma)$ is the one where $\psi$ is not necessarily null.}\label{figura2}
	\end{center}
\end{figure}

Let us now come to the reasons for our study.
One motivation is the description 
of a cluster of soap films which are constrained to wet a 
given system of wires $\gamma$ emanating from a given free
boundary plane. The soap 
films are expected to  arrange in such a way to form a free boundary 
on the plane.  Therefore, 
the questions of existence of a minimal configuration 
and its regularity naturally arise.
A second motivation is related to the description of the 
singular part of the $L^1$-relaxation of the Cartesian 2-codimensional area functional
$$
\int_U \sqrt{1 + \vert \grad  u_1\vert^2 + \vert \grad u_2\vert^2 + 
({\rm det}\grad u)^2} ~dx, \qquad u = (u_1, u_2) \in C^1(U; \R^2),
$$
computed on nonsmooth maps. 
The $L^1$-relaxed area functional \cite{AcDa:94,GiMoSu:98_2}, 
denoted by $\mathcal A(\cdot;U)$, 
is mostly unknown, up to 
a few exceptions, see \cite{AcDa:94,BePa:10,BeElPaSc:19,Scala:19,vortex}. One of the remarkable exceptions is the case
 of the vortex map 
$u_V:B_l(0) \setminus \{0\}\subset\R^2\rightarrow\R^2$, defined 
by $u_V(x)=\frac{x}{|x|}$: 
in this case it can be proved that 
\begin{align}\label{value_vortex}
\mathcal A(u_V;B_l(0))
=\int_{B_l(0)}\sqrt{1+|\nabla u_V|^2}dx+\inf \mathcal F_\varphi(\sigma,\psi),
\end{align}
where the infimum is taken over all pairs $(\sigma,\psi)\in \Sigma\times BV(R_{2l})$ with $\psi=0$ a.e. on $E(\sigma)$. Here the setting 
is the following: $n=1$, $R_{2l}=(0,2l)\times(-1,1)$, $\partial^0R_{2l}=(0,2l)\times \{1\}$, $\partial^DR_{2l}=(\{0\}\times (-1,1))\cup ([0,2l]\times \{-1\})\cup (\{2l\}\times (-1,1))$, $p=(0,1)$, $q=(2l,1)$, and $\sigma$ is a unique curve in $\overline R_{2l}$ joining $p$ to $q$. The Dirichlet datum $\varphi :\partial^DR_{2l}\rightarrow [0,\infty)$ is the function $\varphi(w_1,w_2)=\sqrt{1-w_2^2}$. This setting is similar to the catenoid case, with the difference that the Dirichlet boundary is here extended to include the basis $(0,2l)\times \{-1\}$ and the free curve $\sigma$ is just one simple curve (see Figure \ref{figura1mezzo}).

\begin{figure}
	\begin{center}
		\includegraphics[width=0.6\textwidth]{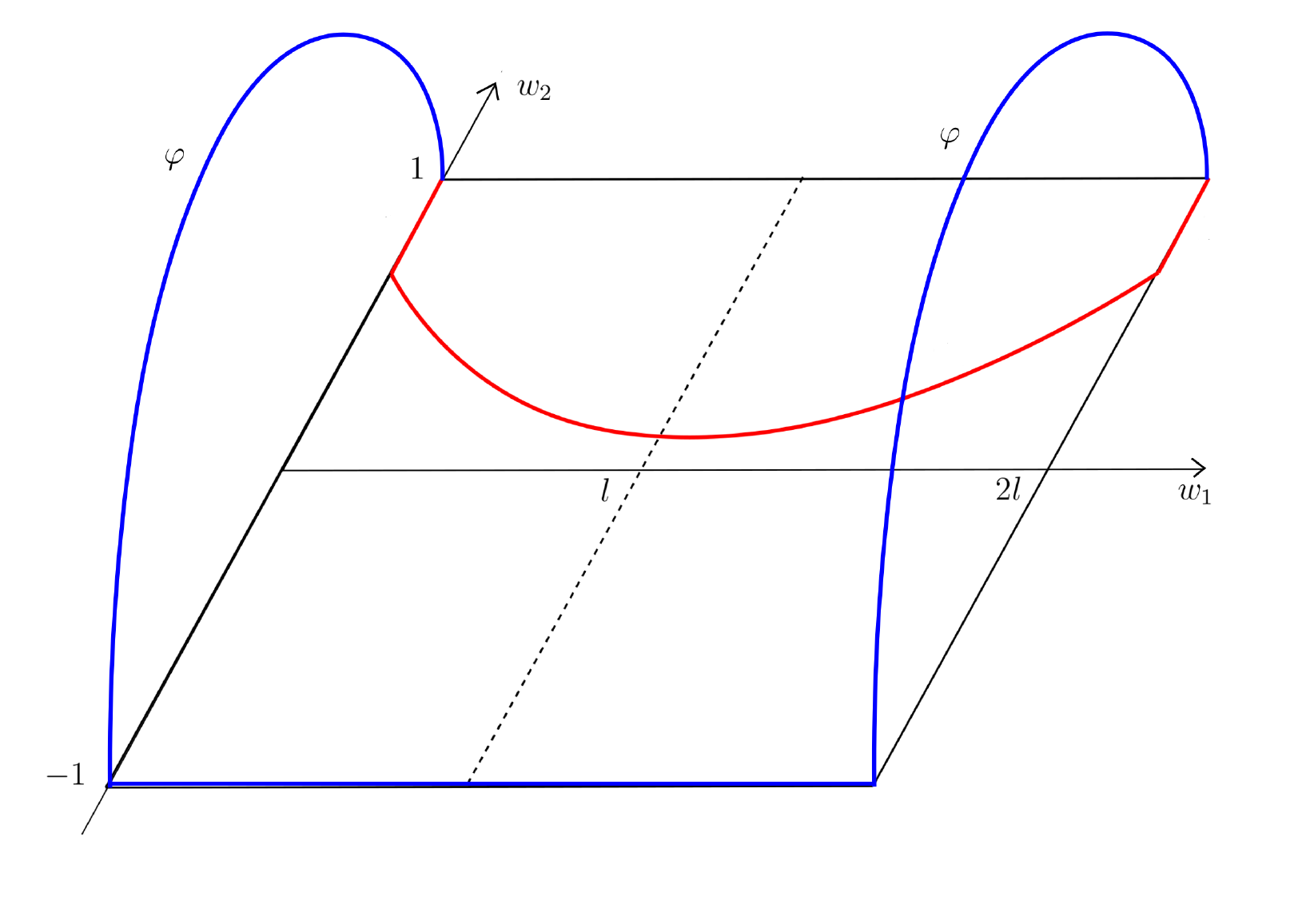}
		\caption{The domain $R_{2l}$ of  the vortex map. The graph of $\varphi$ on $\partial^DR_{2l}$ is emphasized (in particular $\varphi=0$ on the lower horizontal side), together with an admissible curve $\sigma$, which in this specific case partially overlaps the Dirichlet boundary. In this example $n=1$.}\label{figura1mezzo}
	\end{center}
\end{figure}  

In order to construct a recovery sequence for the
relaxed area \eqref{value_vortex} of the vortex map, it is essential to analyse the existence 
and regularity of minimizers of $\mathcal F_\varphi$. In particular, it is necessary to show that there is at least one sufficiently regular\footnote{Conditions provided by Theorem \ref{teo_main_intro} are sufficient.} minimizer $(\sigma,\psi)$. 
The shape of the curve $\sigma$ and the graph of $\psi$ 
are related to the vertical part of a Cartesian current in $B_l(0)\times \R^2$ which arises as limit of (the graphs of) a recovery sequence $(v_k)\subset C^1(B_l(0);\R^2)$ for $\mathcal A(u_V;B_l(0))$. 

According to what happens for the catenoid, also in this case we have a dichotomy for the behaviour of minimizers $(\sigma,\psi)$. When $l$ is small, the solution $(\sigma,\psi)$ consists of a curve $\sigma$ joining $p$ and $q$ whose interior is contained in $R_{2l}$, and its shape is so that $E(\sigma)$ is convex; at the same time the graph of $\psi$ on $R_{2l}\setminus E(\sigma)$ is a sort of half-catenoid, so that if we double it considering also its symmetric with respect to the plane containing $R_{2l}$, it becomes a sort of catenoid spanning two radius one circles, and constrained to contain the segment $(0,2l)\times \{-1\}$. When instead $l$ is larger than a certain threshold, then the solution reduces to two circles spanning the two radius one and parallel circles.

The structure of the paper is as follows.
In Section \ref{sec:preliminaries} we introduce the 
setting of the problem in detail. In order to prove existence of minimizers of $\mathcal F_\varphi$ we first restrict ourselves to prove the result in a smaller
class $\mathcal W_{\textrm{conv}}\subset\mathcal W$ of admissible pairs $(\sigma,\psi)$,
where compactness is easier and allows to make use of the direct method.
 Roughly speaking, the class $\mathcal W_{\textrm{conv}}$
accounts only for specific geometries of the free boundary $\sigma$,
 namely, it considers configurations for which each set $E(\sigma_i)$ is convex. In Section 
\ref{sec:existence_of_a_minimizer_for_the_relaxed_formulation}  we prove the existence of minimizers of $\mathcal F_\varphi$ in $\mathcal W_{\textrm{conv}}$. Next, in Section \ref{sec:minimizers_in_W}, we show the existence of minimizers in the wider class $\mathcal W$ where, essentially, $\sigma$ is not constrained to the previous geometric features; this result is contained in Corollary \ref{cor:mininmathcalW}. To show this we consider a minimizing sequence in $\mathcal W$ and we modify it, by a cut and paste procedure, in order to construct a minimizing sequence in $\mathcal W_{\textrm{conv}}$. In Section 
\ref{sec:existence_of_a_smooth_minimizer} we study the regularity properties of minimizers. Specifically, we state and prove Theorem \ref{thm:regularity}, which rephrases in a more precise way the results contained in Theorem \ref{teo_main_intro}. Theorem \ref{teo_main_intro} 
follows from Theorem \ref{thm:existence}, Corollary \ref{cor:mininmathcalW}, and Theorem \ref{thm:regularity}.
Eventually in Section \ref{sec:comparison_with_the_parametric_Plateau_problem:the_case_n=1,2}
 we compare the solutions we found with the classical minimal surfaces spanning $\Gamma$. Here, as anticipated, we restrict our analysis to the case $n=1,2$, 
the case $n=2$ essentially giving rise to either a catenoid-type minimal surface, or two disc-type surfaces spanning $\Gamma_1$ and $\Gamma_2$. The main theorems here are Theorems \ref{plateau:n1} and \ref{thm:comparison-with-classical-plateau}.
The proof of the former, for the case $n=1$, is quite simple, whereas Theorem \ref{thm:comparison-with-classical-plateau}, for the case $n=2$, requires a series of lemmas. In particular, if $S$ is a Meeks-Yau catenoid-type minimal surface, at one step, we need to employ a Steiner symmetrization of the $3$-dimensional finite perimeter set in $\Omega\times\R$ enclosed by $S$. In turn, using standard results on the condition of equality for the perimeters of a set and its symmetrization, we are able to show that the starting surface $S$ were already symmetric with respect to the plane containing $\Omega$, and already Cartesian, and the conclusion of the proof of Theorem \ref{thm:comparison-with-classical-plateau} is achieved.

\section{Preliminaries}\label{sec:preliminaries}

	\subsection{Area of the graph of a $BV$ function}
	\label{notation}
	Let $U\subset \R^2$ be a bounded open set. For any $\psi\in BV(U)$ we 
	denote by $D\psi$ its distributional gradient, so that
	\begin{equation*}
	D\psi=\nabla \psi\mathcal{L}^2+D^s\psi,
	\end{equation*}
	where $\nabla \psi$ is the approximate gradient of $\psi$ and $D^s\psi$ denotes the singular part of $D\psi$. We recall that the
	$L^1$-relaxed area functional reads as \cite{Giusti:84}
	\begin{align}\label{def:relaxed_area}
	\mathcal A(\psi;U):=\int_U\sqrt{1+|\nabla\psi|^2}dx+|D^s\psi|(U).
	\end{align}
	In what follows we denote by $\partial^*A$ the reduced boundary 
	of  a set of finite perimeter $A\subset\R^3$ (see \cite{AFP}).
	For any $\psi\in BV(U)$ we denote by $R_\psi\subset U$ the set of  regular points of $\psi$, namely the set of points $x\in U$ which are Lebesgue points for $\psi$, $\psi(x)$ 
coincides with the Lebesgue value of $\psi$ at $x$ and $\psi$ is approximately differentiable at $x$.
	We define the subgraph $SG_\psi$ of $\psi$ as
	\begin{equation*}
		SG_\psi:=\{(x,y)\in R_\psi\times\R\colon y<\psi(x)\}.
	\end{equation*}
	This turns out to be a finite perimeter set in $U\times\R$. Its reduced boundary 
	in $U\times \R$ is the generalised graph $\mathcal G_\psi:=\{(x,\psi(x))\colon x\in R_\psi\}$ of $\psi$, 
	which turns out to be a $2$-rectifiable set. If 
	$\jump{SG_\psi}\in\mathcal D_3(\R^3)$ 
	denotes the integral current given by integration over $SG_\psi$ and
	$\partial \jump{SG_\psi}\in\mathcal D_2(\R^3)$ is its boundary in the sense of currents, then 
	$$
	\jump{\mathcal{G}_\psi}=\partial\jump{SG_\psi}\res(U\times\R),
	$$
	with $\jump{\mathcal{G}_\psi}$ denoting the integer multiplicity $2$-current given by integration over $\mathcal{G}_\psi$ (suitably oriented;  see \cite{GiMoSu:98} for more details).
	\subsection{Hausdorff distance}
	
 If $A,B\subset\R^2$ are nonempty, the symbol $d_H(A,B)$ 
 stands for the Hausdorff distance between $A$ and $B$, that is
 \begin{equation*}
 d_H(A,B):=\max\left\{\sup_{a\in A}
 \di_B(a)\,,\,\sup_{b\in B}\di_A(b)\right\},
 \end{equation*}
 where $\di_F(\cdot)$ is the distance from the nonempty set $F\subseteq \R^2$. 
 If we restrict $d_H$ to the class of closed sets, 
 then $d_H$ defines a metric. Moreover:
 \begin{enumerate}[label=(H\arabic*)]
 	\item \label{H1}
 	$\di_A(x)\le \di_B(x)+d_H(A,B)$
 	for every $x\in\R^2$;
 	\item \label{H2} $(\compacts,d_H)$ with
 	$\compacts:=\{K\subset\R^2\, \text{nonempty and compact}\}$ is a complete metric space; 
 	\item \label{H3}
 	If $A, B\subset\R^2$ are bounded, closed, nonempty and convex sets, then $d_H(A,B)=d_H(\partial A,\partial B)$;
 	\item \label{H4}
 	If $A\in \compacts$ is convex, 
 	then there exists a sequence $(A_n)_{n} \subset \compacts$ 
 	of  convex sets with boundary of class $C^\infty$ such that
 	$d_H(A_n,A)\to0$ as $n\to\infty$;
 	\item \label{H5} Let $(A_n)_{n}$ 
be a sequence of closed convex sets in $\R^2$, $A\subset\R^2$ and $d_H(A_n,A)\to0$ as $n\to +\infty$. Then $A$ is convex as well; 
 	\item\label{H6} Let $(A_n)_{n}$ and  $A$ be compact convex subsets of $\R^2$ such that $d_H(A_n,A)\to0$ and let $x\in {\rm int}(A)$; then $x\in A_n$ definitely in $n$;
 	\item \label{H7} Let $A$ and $B$ be closed subsets of $\R^2$ with $d_H(A,B)=\eps$. Then $A\subset B^+_\eps$ and $B\subset A^+_\eps$ 
 	where, for all $E\subset \R^2$, we have set $E^+_\eps:=\{x\in\R^2\colon \di_E(x)\le\eps\}$.
 \end{enumerate}
 \begin{remark}
 	Property \ref{H1} is straightforward, while \ref{H2} is well-known. Also property \ref{H3} is easily obtained (see, e.g. \cite{W}). 
 	Concerning property \ref{H4} we refer to, e.g.,  \cite[Corollary 2]{AF}. 
 	To see \ref{H5}, from  \ref{H1} we have that $\di_{A_n} \to \di_A$ pointwise, 
 	and  therefore since $\di_{A_n}$ is convex, also $\di_A$ is convex, which implies $A$ convex\footnote{Since $A$ is closed, it 
coincides with the sublevel $\{x:d(x,A)\leq 0\}$, which is convex.}.
 	Let us now prove \ref{H6} by contradiction; 
 	assume that there exists a subsequence $(n_k)$ such that $\di_{A_{n_k}}(x)>0$ 
 	for all $k\in \mathbb N$; then $x\in  \R^2\setminus A_{n_k}$, 
  $\di_{A_{n_k}}(x)=\di_{\partial A_{n_k}}(x) $, and using \ref{H1} twice,
 	\begin{align*} 
 	\di_{\partial A}(x)&\leq \di_{\partial A_{n_k}}(x)+
 	d_H(\partial A_{n_k},\partial A)=\di_{A_{n_k}}(x)+d_H( A_{n_k},A)\\
 	&\leq \di_A(x)+ 2d_H(A, A_{n_k})=2d_H(A, A_{n_k})\rightarrow0,
 	\end{align*}
 	the first equality following from \ref{H3}. 
 	This implies $x\in \partial A$, a contradiction.
 \end{remark}

 We begin with the following standard result that will be useful later:
 
 \begin{lemma}\label{lm:charact_convex_sets}
 	Let $K\subset\R^2$ be a convex compact set with nonempty interior. 
 	Then there exists a $1$-periodic curve $\widehat \sigma\in {\rm{Lip}}(\mathbb R;\R^2)$, injective on $[0,1)$, such that $\widehat \sigma([0,1])=\partial K$ and
 	\begin{equation*}
 	\widehat\sigma(t)=\widehat \sigma(0)+\len(\widehat\sigma)\int_0^t\widehat \gamma(s)\,ds,\quad \widehat\gamma(t)=(\cos(\widehat\theta(t))\,,\,\sin(\widehat\theta(t)))\quad \text{for all }\, t\in[0,1],
 	\end{equation*}
 	with $\widehat \theta$ a non-decreasing function satisfying $\widehat \theta(t+1)-\widehat \theta(t)=2\pi$ for all $t\in\R$. 
 \end{lemma}
 Notice that $\widehat \sigma$ is differentiable a.e. in $\mathbb R$ and $\widehat \sigma'(t)=\len(\widehat\sigma)\widehat \gamma(t)$, so that  the speed modulus of the curve $|\hat\sigma'(t)|=\len(\widehat\sigma)$ is constant and coincides with  the length of the curve  $\len(\widehat\sigma)=\int_0^1|\widehat \sigma'(s)|ds$.
 
 \begin{proof}
 	We start by approximating $K$ by convex sets  with $C^\infty$ boundary. By \ref{H4} for all $n\in\mathbb{N}$ there is   a convex compact set $K_n\subset\R^2$ with boundary of class $C^\infty$ and such that 
 	$d_H(K_n,K)\to0$ as $n\to\infty$.
 	For any $n\in\mathbb{N}$ we let $\widehat\sigma_n\in C^\infty(\mathbb R;\R^2)$ be a $1$-periodic function injectively parametrizing $\partial K_n$ on $[0,1)$; therefore $\widehat\sigma_n([0,1])=\partial K_n$, and 
 	\begin{equation*}
 	\widehat \sigma_n(t)=\widehat \sigma_n(0)+\len(\widehat\sigma_n)\int_0^t\widehat\gamma_n(s)\,ds,\quad \widehat\gamma_n(t)=(\cos(\widehat\theta_n(t))\,,\,\sin(\widehat\theta_n(t)))
 	\quad \forall t\in [0,1],
 	\end{equation*}
 	where $\widehat \theta_n\in C^\infty(\R)$ is a non-decreasing function with $\widehat\theta_n(t+1)-\widehat\theta_n(t)=2\pi$, for all $t\in \mathbb R$.
 	In view of (H2), by construction we can find $x_0\in K$, $R>r>0$ such that $B_r(x_0)\subset K_n\subset B_R(x_0)$ for all $n\in \mathbb N$, and  therefore 
 	$\mathcal{H}^1(\partial B_r(x_0))\le \len(\widehat\sigma_n)=
 	\H^1(\mathcal \partial K_n)\le\H^1(\partial B_R(x_0))$; thus, up to subsequence, $\len(\widehat\sigma_n)\to \widehat m\in\R^+$ as $n\to\infty$. Moreover, up to subsequence, we might assume $\widehat\sigma_n(0)\rightarrow p\in \partial K$. 
 	On the other hand observing that
 	\begin{equation*}
 	\int_t^{t+1}|\widehat\theta'_n(s)|ds=\int_t^{t+1}\widehat\theta_n'(s)ds=2\pi, \text{ for all }t\in \mathbb R,
 	\end{equation*}
 	we have that, again up to subsequence, $\widehat\theta_n\weakstar\widehat\theta\in BV_{{\rm loc}}(\R)$ and pointwise (by Helly selection principle), with $\widehat\theta$  a non-decreasing  function with $\widehat \theta(t+1)-\widehat \theta(t)=2\pi$ for all $t\in \mathbb R$. We also have $\widehat \gamma_n\weakstar\widehat \gamma$ in $BV_{{\rm loc}}(\mathbb R;\mathbb R^2)$ where $\widehat \gamma(t)=(\cos(\widehat \theta(t))\,,\,\sin(\widehat \theta(t)))$.\\
 	We let $\widehat \sigma\in {\rm{Lip}}(\R;\R^2)$ be the ($1$-periodic) simple closed curve defined as
 	\begin{equation}\label{sigma_boundary}
 	\widehat \sigma(t):=p+\widehat m\int_0^t\widehat \gamma(s)\,ds\quad \forall t\in\mathbb R.
 	\end{equation}
 	Note that  $\widehat m=\len(\widehat\sigma)$.
 	Then clearly $\widehat \sigma_n\to \widehat \sigma$ in $W^{1,1}([0,1];\R^2)$, since
 	\begin{align*}
 	\|\widehat \sigma'_n-\widehat \sigma'\|_{L^1([0,1];\R^2)}&= \int_0^1|\len(\widehat\sigma_n) \widehat \gamma_n(t)-\len(\widehat\sigma)\widehat \gamma(t)|dt\\
 	&\le |\len(\widehat\sigma_n)-\len(\widehat\sigma)|+\len
 	(\widehat\sigma)\int_0^1|\widehat \gamma_n(t)-\widehat \gamma(t)|dt\to0.
 	\end{align*}
 	By the continuous embedding $W^{1,1}([0,1];\R^2)\subset C^0([0,1];\R^2)$ (and by $1$-periodicity, on $\R$) we also get $\widehat \sigma_n\to\widehat \sigma$ uniformly on $[0,1]$.
 	This, together with property \ref{H3} gives
 	\begin{equation*}
 	d_H(\partial K,\widehat \sigma([0,1]))\le d_H(\partial K,\partial K_n)+d_H(\widehat \sigma_n([0,1]),\widehat \sigma([0,1]))\to0,
 	\end{equation*}
 	which in turn implies $\widehat\sigma([0,1])=\partial K$. The injectivity of $\widehat \sigma$ on $[0,1)$ follows from expression \eqref{sigma_boundary}, the fact that $\widehat m>0$, and the fact that $K$ is convex with nonempty interior. 
 	
 \end{proof}

 \begin{cor}\label{lm:charact_convex_sets2}
 	Let $K\subset\R^2$ be a convex compact set with nonempty interior. Let $q,p$ be two distinct points on $\partial K$, and let  $\arc{pq}$ be the relatively open, connected  curve contained in $\partial K$ with endpoints $q$ and $p$ clockwise ordered.
 	Then there exists an injective curve $\sigma\in {\rm{Lip}}([0,1];\R^2)$ such that $\sigma((0,1))=\arc{pq}$, $\sigma(0)=q$, $\sigma(1)=p$, and
 	\begin{equation*}
 	\sigma(t)=q+\len(\sigma)\int_0^t\gamma(s)\,ds,\quad \gamma(t)=(\cos(\theta(t))\,,\,\sin(\theta(t))\quad \text{for all }\, t\in[0,1],
 	\end{equation*}
 	with $\theta$ a non-decreasing function satisfying $\theta(1)-\theta(0)\le2\pi$.
 \end{cor}

 \begin{proof}
 	Lemma \ref{lm:charact_convex_sets} provides $\widehat \sigma\in \Lip ([0,1];\R^2)$ parametrizing $\partial K$. Then there are two values $t_1,t_2\in [0,1]$, $t_1<t_2$, with $q=\sigma(t_1)$ and $p=\sigma(t_2)$.  Then the existence of $\sigma$ follows by reparametrization of the interval $[t_1,t_2]$, and all the properties follows from the corresponding properties of $\widehat \sigma$.
 \end{proof}
 \subsection{Setting of the problem}\label{subsec:setting_of_the_problem}
 	We fix $\Omega\subset\R^2$ to be an open bounded convex set (strict
convexity is not required) which will be our reference domain.  Given two points $p, q \in \partial \Omega$
in clockwise order, $\arc{pq}$ stands for the relatively
open arc on $\partial \Om$ joining $p$ and $q$. 
 
Let $n\in\mathbb N$, $n \geq 1$, and let $\{p_i\}_{i=1}^n$ be 
distinct points on $\partial\Omega$ chosen in clockwise order; 
we set $p_{n+1} := p_1$. 
For all $i=1,\dots,n$ let $q_i$ be a point in
$\arc{p_ip_{i+1}}
\subset \dOm$. We set
	\begin{align}
	\partial_i^D\Omega:=\arc{p_iq_i},\qquad\nullb_i\Omega:=\arc{q_ip_{i+1}}\qquad \text{ for }i=1,\dots,n,
	\end{align}
	and 
	\begin{equation}
	\partial^D\Om:=\bigcup_{i=1}^n\partial_i^D\Om,\qquad 	\nullb\Om:=\bigcup_{i=1}^n\nullb_i\Om.
	\end{equation}
Since $\partial_i^D\Om$ and $\nullb_i\Omega$
are relatively open in $\partial \Omega$,
so are $\partial^D\Omega$ and $\nullb\Omega$. It follows that $\partial
\Om$ is the disjoint union
$$\partial\Omega=\cup_{i=1}^n\{p_i,q_i\}\cup \partial^D\Om \cup
\nullb\Om.
$$
	We fix a continuous function 
$\varphi:\partial \Omega\rightarrow [0,+\infty)$ such that 
$$
\varphi = 0  \ {\rm on }~ \nullb\Omega
\qquad {\rm and} \qquad
\varphi >0 \ {\rm on }~ 
\partial^D\Omega,
$$
see Figures \ref{figura1}, \ref{figura3}. We will make a further regularity assumption on $\varphi$: we require that the graph $\mathcal G_{\varphi\res\partial^D_i\Om}=\{(x,\varphi(x)):x\in \partial^D_i\Om\}$ of $\varphi$ on $\partial^D_i\Om$ is a Lipschitz curve in $\R^3$, for all $i=1,\dots,n$.
\begin{Remark}\rm
The hypothesis $\varphi>0$ on $\partial^D\Om$ excludes from our analysis 
the example in Figure \ref{figura1mezzo} of
the introduction.
We will further comment on this later on (see Section \ref{subsec_example}); the presence of pieces of $\partial^D\Om$ where $\varphi=0$ will bring to some additional technical difficulties that we prefer to avoid here. However, the setting in Figure \ref{figura1mezzo} can be 
easily achieved by an approximation argument. Namely, one considers a suitable regularization $\varphi_\eps$ of $\varphi$ on $\partial^D\Om$ such that $\varphi_\eps>0$, and then letting $\eps\rightarrow0$ one obtains a solution to the problem with Dirichlet datum $\varphi$.  
\end{Remark}
		We will analyse the functional $\mathcal F = \mathcal F_\varphi$ defined in \eqref{def:relaxed_functional_intro}, namely
	\begin{align}\label{def:relaxed_functional}
	\mathcal F(\sigma,\psi):=\mathcal A(\psi;\Omega)-|E(\sigma)|+\int_{\partial\Omega}|\psi-\varphi|d\mathcal H^1,
	\end{align}
where the pair $(\sigma,\psi)$ belongs to the admissible class 
$\largercomp$, defined as follows:
\begin{equation}\label{def:admissible_class_bis}
\begin{split}
&
\largercomp
:=\Big\{(\sigma,\psi)\in \curves \times BV(\Omega): \psi=0\text{ a.e. in }E(\sigma)\Big\},\\
&\curves:=\Big\{\sigma = (\sigma_1,\dots,\sigma_n)
\in ({\rm Lip}([0,1]; \overline\Om))^n \text{ satisfies (i')-(ii')}\Big\},
\end{split}
\end{equation}
where
\begin{enumerate}[label=(\roman*')]
	\item{\label{i'}} $\sigma=(\sigma_1,\dots,\sigma_n)$ 
with $\sigma_i$ injective, $\sigma_i(0)=q_i$ and $\sigma_i(1)=p_{i+1}$, for  all $i=1,\dots,n$;
	\item{\label{ii'}} For $i=1,\dots,n$,
denoting by 
$E(\sigma_i) \subset \overline \Om$ the closed region enclosed between 
$\nullb_i\Omega$ and $\sigma_i([0,1])$,
we assume 
$\textrm{int}(E(\sigma_i))\cap \textrm{int}(E(\sigma_j))= \emptyset$
for $i \neq j$ where  int denotes the interior part;
we also
set 
\begin{equation}\label{def:E}
	E(\sigma):=\bigcup_{i=1}^n E(\sigma_i).
\end{equation}
\end{enumerate}

\begin{remark}
The injectivity property in \ref{i'} guarantees that the sets $E(\sigma_i)$ 
are simply connected (not necessarily connected). 
The assumption that the interior ${\rm int}(E(\sigma_i))$ 
of the sets $E(\sigma_i)$ are mutually disjoint
is an hypothesis on the curves $\sigma_i$, which essentially translates into the fact that these curves cannot cross
transversally each other, but might overlap. Notice that ${\rm int}(E(\sigma_i))$ might be empty, as the case $\nullb_i\Om=\sigma_i([0,1])$ is not excluded.

\end{remark}

The strategy to show existence and regularity of minimizers 
 of the functional \eqref{def:relaxed_functional} 
(see \eqref{eq:min_F_storto}) 
is to reduce to study the same functional on a restricted class of competitors,  more precisely to reduce our analysis to the case where the sets $E(\sigma_i)$ are convex.
Specifically, we define:
	\begin{equation}\label{def:admissible_class}
	\begin{split}
	&\comp:=\Big\{(\sigma,\psi)\in \convexcurves\times BV(\Omega): 
\ \psi=0\text{ a.e. in }E(\sigma)\Big\},\\
&\convexcurves:=\Big\{\sigma = (\sigma_1,\dots,\sigma_n) 
\in \curves: \sigma \text{ satisfies (i)}\Big\},
	\end{split}
	\end{equation}
and: 
	\begin{enumerate}[label=(\roman*)]
		\item\label{(i)} For 
all $i=1,\dots,n$ the set $E(\sigma_i)$ is convex.
	\end{enumerate}

As we have already said, the sets $\textrm{int}(E(\sigma_i))$ 
might also be empty, since from assumption (i') we cannot exclude that $\sigma_i$ 
overlaps $\nullb_i\Omega$: Recalling that $\Omega$ is convex, this can happen, 
by (ii') and (i), only if $\arc{q_ip_{i+1}}$
is a straight segment\footnote{We shall prove 
that, for a minimizer, $\sigma_i([0,1])$ cannot intersect  
$\partial^D \Omega$ unless $\partial^D \Omega$
is locally a segment, see
Theorem \ref{thm:regularity}.}.
Clearly,
\begin{equation}
\label{eq:W_subset_hatW}
\comp\subset \largercomp.
\end{equation}

\begin{remark}\label{rem:condition_P}
Exploiting the characterization of the boundaries of convex sets 
given in Corollary \ref{lm:charact_convex_sets2}, we see that conditions \ref{i'},\ref{ii'} and \ref{(i)} for the curves in $\convexcurves$ imply the following:

\begin{enumerate}[label=(P)]
	\item\label{(P)}
For all $i=1,\dots, n$ there is a 
nondecreasing function $\theta_i\colon[0,1]\to\R$ 
with $\theta_i(1)-\theta_i(0)\le2\pi$, 
 and such that, setting $\gamma_i(t):=(\cos(\theta_i(t))\,,\,\sin(\theta_i(t)))$ for all $t\in[0,1]$, we have 
	\begin{equation*}
	\sigma_i(t)=q_i+\len(\sigma_i)\int_0^t\gamma_i(s)\,ds\quad \quad\forall t\in[0,1].
	\end{equation*}
	
\end{enumerate}
Here we have denoted the length of $\sigma_i$ by $\len(\sigma_i)$. 
\end{remark}

\section{Existence of minimizers of $\mathcal F$ in $\comp$}
\label{sec:existence_of_a_minimizer_for_the_relaxed_formulation}

The main result of this section reads as follows.

\begin{theorem}[\textbf{Existence of a minimizer of $\mathcal F$ in $\comp$}]\label{thm:existence} 
Let $\mathcal F$ and $\comp$ be as in \eqref{def:relaxed_functional} 
and \eqref{def:admissible_class} respectively. Then there 
is 
a solution 
to 
	\begin{equation}\label{eq:min_F_storto} 
		\min_{(\sigma,\psi)\in\comp}\mathcal F(\sigma,\psi).
	\end{equation}
	
\end{theorem}
We prove Theorem \ref{thm:existence} using the direct method. To this aim 
we need to introduce a notion of convergence in $\comp$.
\begin{definition}[\textbf{Convergence in $\comp$}]
\label{def:conv}
	We say that the sequence $(((\sigma)_k,\psi_k))_k\subset \comp
$, with $(\sigma)_k=((\sigma_1)_k,...,(\sigma_n)_k)$, converges to $(\sigma,\psi)\in \comp$ if:
	\begin{enumerate}[label=$(\alph*)$]
		\item\label{a} 
$(\sigma_{i})_k$ converges to 
 $\sigma_i$
uniformly in $[0,1]$ 
for all $i=1,\dots,n$;
		\item \label{b} $(\psi_k)_k$ converges 
to $\psi$ weakly star in $BV(\Omega)$, i.e., 
$\psi_k \to \psi$ in $L^1(\Omega)$ and 
$D \psi_k \rightharpoonup D\psi$ weakly star in $\Omega$ 
in the sense of measures as 
$k \to +\infty$.    
	\end{enumerate}
\end{definition}

\begin{remark}\label{rem:conv}
For any $i=1,\dots,n$ we have $\lim_{k \to +\infty}
d_H(E((\sigma_{i})_k),E(\sigma_i))= 0$, 
since by property \ref{H3} 
\begin{equation*}
	d_H(E((\sigma_{i})_k),E(\sigma_i))=
d_H(\partial E((\sigma_{i})_k),\partial E(\sigma_i))=
d_H((\sigma_{i})_k([0,1]),\sigma_{i}([0,1]))\to0.
\end{equation*}
\end{remark}

\begin{lemma}[\textbf{Compactness of $\comp$}]\label{lem:compactness}
	Let 
$\big(((\sigma)_{k},\psi_k)\big)_k
\subset\comp$ be a sequence with\\ $\sup_k\mathcal F((\sigma)_k,\psi_k)<+\infty$. Then 
$\big(((\sigma)_{k},\psi_k)\big)_k$ admits 
 a subsequence converging to an element of $\comp$.
\end{lemma}
\begin{proof}We divide the proof in two steps.\\
	
	\step 1 Compactness of $(\sigma)_k$. For simplicity we use the notation $\sigma_{ik} = (\sigma_i)_k$ for every $k\in
\mathbb N$ and $i\in\{1,\dots,n\}$. By condition (P) in Remark \ref{rem:condition_P}, for every $k \in \mathbb N$ and $i\in\{1,\dots,n\}$ 
there exists a non-decreasing 
function $\theta_{ik}\colon[0,1]\to\R$,  $\theta_{ik}(1)-\theta_{ik}(0)\le2\pi$,
such that 
 \begin{equation*}
 \sigma_{ik}(t)=q_i+ \len(\sigma_{ik})\int_0^t\gamma_{ik}(s)ds,\quad \gamma_{ik}(t):=(\cos\theta_{ik}(t)\,,\,\sin\theta_{ik}(t))
 \quad\forall t\in[0,1],
 \end{equation*}
 and with $\sigma_{ik}(1)=p_{i+1}$.
 We observe that 
 $$\len(\sigma_{ik})=\int_0^1|\sigma_{ik}'(t)|dt\le\mathcal{H}^1(\partial\Om),$$ 
since the orthogonal projection $\Pi_{ki}\colon\partial\Om\setminus\partial^0_i\Om\to E(\sigma_{ik})$ is a contraction and $\mathcal{H}^1(\partial\Om\setminus\partial^0_i\Om)\le \mathcal{H}^1(\partial\Om)$. 
 Hence, up to a (not relabelled) subsequence, 
$\len(\sigma_{ik})\to m_i\in \R^+$ as $k \to +\infty$. The number $m_i$ is positive since, for all $k$ and $i$, we have 
$\len(\sigma_{ik})\geq |q_i-p_{i+1}|>0$.
 Moreover 
\begin{equation*}
	\int_0^1|\theta'_{ik}(t)|dt=\int_0^1\theta_{ik}'(t)dt\le2\pi;
\end{equation*}
hence, up to a subsequence, $\theta_{ki}\weakstar\theta_i$ in $BV([0,1])$ and $\theta_i$ is non-decreasing with $\theta_i(1)-\theta_i(0)\le2\pi$. Furthermore $\gamma_{ik}\stackrel{*}{\rightharpoonup}\gamma_i$ in $BV([0,1];\R^2)$  with $\gamma_i(t)=(\cos(\theta_i(t))\,,\,\sin(\theta_i(t)))$. 

As a consequence $\sigma_{ik}\to\sigma_i$ in $W^{1,1}([0,1];\R^2)$, where 
\begin{equation*}
	\sigma_i(t):=q_i+m_i\int_0^t\gamma_i(s)ds=q_i+\len(\sigma_i)\int_0^t\gamma_i(s)ds.
\end{equation*}
Indeed we have
\begin{align}
\|\sigma_{ik}'-\sigma'_i\|_{L^1([0,1];\R^2)}&=\int_0^1|\len(\sigma_{ik})\gamma_{ik}(t)-\len(\sigma_i)\gamma_i(t)|dt\notag\\
	&\le \mathcal{H}^1(\partial\Om)\int_0^1|\gamma_{ik}(t)-\gamma_i(t)|dt+ |\len(\sigma_{ik})-\len(\sigma_i)|.\label{w11}
\end{align}
Now taking the limit as $k\to+\infty$ in \eqref{w11} we conclude.
Thus $\lim_{k \to +\infty}\sigma_{ik}=\sigma_i$ uniformly, hence we also conclude that $\sigma_i$ takes values in $\overline\Om$.\\
It remains to show that $E(\sigma_i)$ 
is convex for any $i \in \{1,\dots, n\}$. 
The uniform convergence of $(\sigma_{ik})$ yields 
$$
\lim_{k\to+\infty}d_H(\partial E(\sigma_{ik}),\partial E(\sigma_{i}))=0.
$$
This, together with property \ref{H3}, gives for $h\ge k$,
\begin{equation*}
\begin{aligned}
	&d_H( E(\sigma_{ik}), E(\sigma_{ih}))=
d_H(\partial E(\sigma_{ik}),\partial E(\sigma_{ih}))
\\
\leq & d_H(\partial E(\sigma_{ik}),\partial E(\sigma_i))+d_H(\partial E(\sigma_{ih}),\partial E(\sigma_i))\to 0\quad\text{as}\quad k\to+\infty,
\end{aligned}
\end{equation*}
and so $(E(\sigma_{ik}))_{k\in \mathbb N}$ is a Cauchy sequence in the space of compact subsets of $\R^2$ endowed with the Hausdorff distance (see \ref{H2}).
We find $ K\subset \R^2$ convex compact such that $d_H( E(\sigma_{ik}),K)\to0$.
Eventually from \ref{H3} we get
\begin{align*}
	&d_H(\partial K, \partial E(\sigma_i))\le d_H(\partial E(\sigma_{ik}),\partial K)+ d_H(\partial E(\sigma_{ik}),\partial E(\sigma_{i}))\\
	=&d_H(  E(\sigma_{ik}),K)+ d_H(\partial E(\sigma_{ik}),\partial E(\sigma_{i}))\to0\quad\text{as}\quad k\to+\infty.
\end{align*}
Therefore we conclude that $\partial K=\partial E(\sigma_{i})$, so $E(\sigma_{ik})\rightarrow E(\sigma_{i})$ in the Hausdorff distance,
and $E(\sigma_i)$ is convex by property \ref{H5}.

\step 2 Compactness of $(\psi_k)$.
Setting $F_k=\cup_{i=1}^nE(\sigma_{ik})$ we have  
\begin{equation*}
	|D\psi_k|(\Om)\le \mathcal{A}(\psi_k,\Om)\le \mathcal{F}((\sigma)_k,\psi_k)+|F_k|\le C<+\infty \qquad \forall k>0,
\end{equation*}
where we used that $|F_k|\le |\Om|$. Therefore, 
up to a subsequence, $\psi_k\weakstar\psi$ in $BV(\Om)$
as $k \to +\infty$. 
To conclude it remains to show that $\psi=0$ a.e.  in $E(\sigma)=\cup_iE(\sigma_{i})$.  By $\lim_{k \to +\infty}d_H(F_k,E(\sigma))=0$,  
property \ref{H6} yields
\begin{equation*}
	\text{if}\quad x\in \text{int}(E(\sigma)) \quad \text{then}\quad x\in F_k\quad\text{definitely in } k,
\end{equation*}
and hence since $\lim_{k\to +\infty} \psi_k= \psi$ a.e. in $\Om$, we infer $\psi=0$ a.e. in $ E(\sigma)$.

\end{proof}

\begin{lemma}[\textbf{Lower semicontinuity of $\mathcal F$ in $\comp$}]
\label{lem:lower_semicontinuity_of_F}
Let $\big(((\sigma)_k,\psi_k)\big)_k\subset\comp$ be a sequence converging 
to $(\sigma,\psi)\in \comp$. Then
\begin{equation*}
	\mathcal{F}(\sigma,\psi)\le\liminf_{k\to+\infty}\mathcal{F}((\sigma)_k,\psi_k).
\end{equation*}
\end{lemma}
\begin{proof}
By a standard argument \cite{Giusti:84}, the functional
	\begin{equation*}
		\psi\in BV(\Om)\mapsto \mathcal{A}(\psi;\Om)+\int_{\partial\Om}|\psi-\varphi|d\mathcal H^1
	\end{equation*}
	is $L^1(\Om)$-lower semicontinuous.
We now show that the map $\sigma \in
\convexcurves
 \mapsto|E(\sigma)|$ 
is continuous. 
Let $(\sigma)_k \subset \convexcurves$, 
$\sigma \in\convexcurves$, and suppose that
$(\sigma_{i})_k$ uniformly converges to $\sigma_i$
 for all $i=1,\dots,n$ as $k \to +\infty$.
Set $F_k:=\cup_{i=1}^n E((\sigma_{i})_k)$ and recall that 
 $E(\sigma)=\cup_{i=1}^n E(\sigma_{i})$. 
	By Remark \ref{rem:conv} 
$\lim_{h \to +\infty} d_H(E((\sigma_{i})_k),E(\sigma_i)) = 0$ 
for all $i=1,\dots,n$ and therefore $d_H(F_k,E(\sigma)) =: \eps_k\to 0^+$. 

By invoking \ref{H7} we have $E(\sigma)\subset (F_k)^+_{\eps_k}$. 
Moreover, since $d_H((F_k)_{\eps_k}^+,E(\sigma))\leq 2\eps_k$, we get $(F_k)^+_{\eps_k}\subseteq (E(\sigma))_{2\eps_k}^+$, and so
\begin{equation*}
\vert E(\sigma)\vert\le	\vert (F_k)^+_{\eps_k}\vert 
\le \vert (E(\sigma))_{2\eps_k}^+\vert.
\end{equation*}
This implies
$$\limsup_{k\rightarrow+\infty} |F_k|\leq \limsup_{k\rightarrow+\infty} |(F_k)^+_{\eps_k}|\le\mathcal |E(\sigma)|.$$
The converse inequality is a consequence of  Fatou's Lemma and \ref{H6}, indeed
$$|E(\sigma)|\leq \int_\Om\liminf_{k\rightarrow+\infty}{\chi_{F_k}(x)}dx\leq\liminf_{k\rightarrow+\infty}\int_\Om{\chi_{F_k}(x)}dx =\liminf_{k\rightarrow+\infty}|F_k|.$$
The assertion of the lemma follows.
\end{proof}
\begin{proof}[Proof of Theorem \ref{thm:existence}] 
By Lemma \ref{lem:compactness} and Lemma \ref{lem:lower_semicontinuity_of_F} we can apply the direct method and conclude.
\end{proof}

\section{Existence of a minimizer of $\mathcal F$ in $\largercomp$}\label{sec:minimizers_in_W}
In this section we 
extend the previous results to the minimization of 
${\mathcal F}$  in the larger class 
$\largercomp$ of competitors.

One issue we find in minimizing the functional 
${\mathcal F}$ on $\largercomp$, 
is that the class $\curves$ in \eqref{def:admissible_class_bis} 
is not closed under uniform convergence, since
a uniform limit of elements 
in $\curves$ needs not be formed by injective curves. 
To overcome this difficulty, in Theorem \ref{teo:convexification} we prove that the infimum of 
${\mathcal F}$ over $\largercomp$ coincides with the infimum of ${\mathcal F}$ over $\comp$. 
Thus in particular, 
by Theorem \ref{thm:existence}, we derive the existence of a minimizer for 
${\mathcal F}$ in $\largercomp$ (Corollary \ref{cor:mininmathcalW}).

\begin{theorem}[\textbf{Existence of a minimizer of $\mathcal F$ in $\largercomp$}]\label{teo:convexification}
 There exists $(\sigmamin,\psimin)\in\comp$ such that
	\begin{equation*}
	{\mathcal{F}}(\sigmamin,\psimin)=	\inf_{(s,\zeta)\in
\largercomp }{\mathcal{F}}(s,\zeta).
	\end{equation*}
	Moreover every connected component of $E(\sigmamin)$ is convex.
\end{theorem}
\begin{remark}
Since the $\sigma_i$'s may overlap, the assumption that every 
$E(\sigma_i)$ is convex does not imply in general that every connected component of $E(\sigma)=\cup_{i=1}^nE(\sigma_i)$ is convex.  
\end{remark}
As a direct consequence of Theorem \ref{teo:convexification} we have:

\begin{cor}\label{cor:mininmathcalW}
	Let $(\sigmamin,\psimin)\in\comp$ be a minimizer as in Theorem \ref{thm:existence}. Then $(\sigmamin,\psimin)$ is also a minimizer of ${\mathcal F}$ in the class $\largercomp$.
\end{cor}

For the reader convenience 
we split the proof of Theorem \ref{teo:convexification} 
into a sequence of intermediate results: Lemmas \ref{lem_prelimiary}, \ref{lem:convexification}, \ref{lem:convexification_bis},
and the conclusion.
First we need to introduce some notation.

Let $(\sigma,\psi)\in \largercomp$.
We fix an extension $\widehat \varphi\in W^{1,1}(B)$ 
of $\varphi$ on an open ball $B\supset\overline\Omega$. 
Extending $\psi$ in  $B \setminus \overline\Omega$ as $\widehat \varphi$
(still denoting by $\psi$ such an extension), we can 
rewrite $\mathcal F(\sigma,\psi)$ as
\begin{equation}\label{eq:equivalent_forumulation_of_F}
{\mathcal F}(\sigma,\psi)=\mathcal A(\psi;B)-|E(\sigma)|-\mathcal A(\psi;B\setminus \overline\Omega).
\end{equation}

\begin{lemma}\label{lem_prelimiary}
	Let  $u\in BV(\R\times (0,+\infty))$ be a nonnegative function with compact support in an open  ball   $B_r$. Then
\begin{equation}\label{eq:trace}
\int_{(\R\times\{0\})\cap B_r}u(s)\;d\mathcal H^1(s)\leq \mathcal A(u;B_r\cap (\R\times(0,+\infty)))-|E_{B_r}|,
\end{equation}
where
$$E_{B_r}:=\{x\in B_{r}\cap (\R\times (0,+\infty)):u(x)=0\}.$$
\end{lemma}

Notice that the 
function $u$ is defined only on the half-plane $\R\times (0,+\infty)$, and 
in \eqref{eq:trace} the symbol $u(s)$ denotes its trace on the line $\R\times \{0\}$.

\begin{proof}

We denote by $x=(x_1,x_2)\in \R^2$ the coordinates in $\R^2$.
Set $\openhalfplane := \R \times (0,+\infty)$, $Z := (B_r\cap \openhalfplane) \times \R$.
			Let 
	\begin{equation*}
	L:=\{(x,y)\in Z \colon x\in R_u,\ y\in(-u(x),u(x))\} \subset \R^3,
	\end{equation*}
where $R_u$ is the set of regular points of $u$.	We have, recalling the notation in Section \ref{notation},
	\begin{equation}\label{eq:claim_step1*}
	\begin{split}
	2\mathcal{A}(u; B_r\cap\openhalfplane)&=\mathcal{A}(u; B_r\cap\openhalfplane)+\mathcal{A}(-u; B_r\cap\openhalfplane)
\\
	&=\mathcal{H}^2(\partial^*(Z \cap SG_{u}))
+\mathcal{H}^2(\partial^*(Z \cap SG_{-u})
)\\
	&=\mathcal{H}^2(Z \cap \partial^*L)+
	2\vert E_{B_r}\vert.
	\end{split}
	\end{equation}
	
Suppose $B_r\cap(\R\times\{0\})=(a,b)\times \{0\}$.  Then, looking at $\mathcal{G}_u$ as an integral current, a slicing argument yields 
	\begin{equation}\label{eq:claim_step2*}
	\begin{split}
	\mathcal{H}^2(Z \cap \partial^*L)&\ge\int_a^b
\mathcal{H}^1(Z \cap\{x_1=t\}\cap \partial^* L)dt\\
	&=\int_a^b\mathcal{H}^1(Z\cap\{x_1=t\}\cap (\textrm{spt}(\mathcal{G}_u-\mathcal{G}_{-u})))dt\\
	&\ge\int_a^b2u(t,0)dt
	= 2\int_{(\R\times\{0\})\cap B_r}u(s)\;d\mathcal H^1(s)\,,
	\end{split}
	\end{equation}
	where the last inequality follows from the following fact:  If we denote by $\jump{\mathcal G_u}_t$ the slice of the current $\jump{\mathcal G_{u}}$ on the line $\{x_1=t\}$, then
	$$\partial \jump{\mathcal G_{u}}_t=\delta_{(t,0,u(t,0))}-\delta_{(t,s_t,0)}
\qquad 
{\rm for~ a.e.~} t\in(a,b), 
$$
	where $s_t\ge0 $ is such that $(t,s_t)=B_r\cap (\{t\}\times \R^+)$, and in 
writing $\delta_{(t,s_t,0)}$ we are using that $u$ has compact support in $B_r$.
	This can be seen, for instance, by approximation of $u$ by smooth maps\footnote{With respect to the strict convergence of $BV(
B_r \cap (\R\times\{0\}))$, which guarantees the approximation 
also of the trace of $u$ on $\partial \big(B_r \cap (\R\times\{0\})\big)$.}.
	Therefore
	\begin{equation*}
	\partial(\jump{\mathcal{G}_u}_t-\jump{\mathcal{G}_{-u}}_t)=\delta_{(t,0,u(t,0))}-\delta_{(t,0,-u(t,0))}
\qquad	{\rm for~ a.e.~} t\in(a,b).
	\end{equation*}
This justifies the last inequality in \eqref{eq:claim_step2*}, and the proof is achieved.
	\end{proof}
 
 We now turn to two 
technical lemmas which are necessary to prove Theorem \ref{teo:convexification}.
 We need to introduce a class of sets whose boundaries are regular enough so that  the trace of a $BV$ function on them is well-defined. 
 Precisely we say that an open subset of $\R^2$ is \textit{piecewise Lipschitz} if 
 it can be written as the union of a finite family of  (not necessarily disjoint) Lipschitz open sets.
 Notice that, by \eqref{def:relaxed_area} if  $V\subset\subset U$ is a piecewise Lipschitz subset  of  an open and bounded $U\subset\R^2$,  then 
 \begin{equation}\label{def:relaxed_area1}
 	\mathcal{A}(\psi,\overline V)=\mathcal{A}(\psi, V)+\int_{\partial V}|\psi^+-\psi^-|d\mathcal{H}^1,
 \end{equation}
 where $\psi^+$ (respectively $\psi^-$) denotes the trace of $\psi\res V$ (respectively $\psi\res(U\setminus\overline V)$) on $\partial V$.

\begin{lemma}[\textbf{Reduction of energy, I}]
\label{lem:convexification}
	For $N\ge1$ let $F_1,\dots,F_N$ be  
	 nonempty connected subsets of $\overline \Om$,
each $F_i$ being the closure of a piecewise Lipschitz 
set, with $F_i \cap F_j= \emptyset$ for $i,j \in \{1,\dots,
N\}$, $i\neq j$. 
	Let $\psi\in BV(B)$ satisfy
	\begin{equation}\label{eq:psi_0_union}
	\psi=0\quad\text{a.e. in}\quad G:=\bigcup_{i=1}^NF_i\quad\text{and}\quad
	\psi=\widehat{\varphi}\quad\text{in}\quad B\setminus{\Om}\,.
	\end{equation}
 Then, for any $i\in\{1,\dots,N\}$,
\begin{align*}
\mathcal A(\psi_i^\star;B)-|G_i^\star|-\mathcal A(\psi_i^\star;B\setminus \overline\Omega)\leq  
\mathcal A(\psi;B)-|G|-\mathcal A(\psi;B\setminus \overline\Omega),
\end{align*}
where 
\begin{equation}\label{def:E_star,Psi_star}
G_i^\star:=\bigcup_{j\neq i}F_j\cup {\rm conv}(F_i)\quad\text{and}\quad
	\psi_i^\star :=\begin{cases}
	0&\text{{\rm in} } {\rm conv}(F_i)\\
	\psi&\text{{\rm otherwise}.}
	\end{cases}
\end{equation}
\end{lemma}

\begin{proof}
Fix $i\in\{1,\dots,N\}$.
	By the convexity of $\Omega$, we have $\psi=\psi_i^\star$ in $B\setminus \overline\Omega$, hence it suffices to show that 
	\begin{equation*}
\mathcal A(\psi_i^\star;B)-|G_i^\star|\leq\mathcal A(\psi;B)-|G|.
	\end{equation*}
We start by observing that we may assume $F_i$ to be simply connected. 
Indeed, if not, we can replace it  with the set obtained by filling the 
holes of $F_i$, and by setting $\psi$ equal to zero in the holes. 
This procedure reduces the energy. 
Indeed, since  $F_i$ is piecewise Lipschitz, any hole $H$ of 
it satisfies $\partial H\subset \cup_{j=1}^n\partial A_j$ 
where $A_j$'s are the Lipschitz sets whose union is $F_i$. 
Hence the trace of $\psi\res H$ on $\partial H$ is well-defined, 
and the external trace $\psi\res (B\setminus H)$ vanishes. 

We have that $(\partial {\rm conv}(F_i))\setminus\partial F_i$ is a countable union of segments. We will next modify $\psi$ by iterating at most countably many operations, setting $\psi=0$ in the region between each of these segments and $\partial F_i$.
	\medskip

		\textit{Step 1: Base case.} Let $l$ be one of such segments, and
	$U$ be the open region enclosed between $\partial F_i$ and $l$.
	We define $\psi'\in BV(\Om)$ as
	\begin{equation*}
	\psi':=\begin{cases}
	0&\text{in}\ \ U\\
	\psi &\text{otherwise}\,.
	\end{cases}
	\end{equation*}
	We claim that
	\begin{equation}\label{eq:claim_bis*}
	\mathcal{A}(\psi';B)-|G'|\le\mathcal{A}(\psi;B)-|G|\,,
	\end{equation}
	where $G':= G \cup
	\overline U$. 
	To prove the claim we introduce the sets 
	$$  
H:={\rm int}(F_i\cup U),\qquad 
	V:=U\cap(\cup_{j\neq i}F_j).
	$$
	Note that $H$ is a piecewise Lipschitz set.
	By construction
	\begin{equation*}
	|G'|=|{H}|+ |\cup_{j\neq i}F_j|-|V|\,,
	\end{equation*}
	and \eqref{eq:claim_bis*} will follow if we show that 
	\begin{equation*}
		\begin{split}
		\mathcal{A}(\psi';B)-|{H}| \le \mathcal{A}(\psi;B)-|\cup_j F_j|+|\cup_{j\neq i}F_j|-|V|=\mathcal{A}(\psi;B)-|F_i\cup V|\,.
		\end{split}
	\end{equation*}
	Since $|H|=|F_i\cup V|+|U\setminus V|$, this can also be written as 
	\begin{equation*}
	\mathcal{A}(\psi';B)
	\le\mathcal{A}(\psi;B)+|U\setminus V|\,.
	\end{equation*}
	In turn $\mathcal{A}(\psi';B)=\mathcal{A}(\psi';\overline U)+\mathcal{A}(\psi';B\setminus \overline U)$ (and similarly for $\psi$), so we have reduced ourselves with proving 
	\begin{equation}\label{eq:claim_tris**}
	\mathcal{A}(\psi';\overline U)
	\le\mathcal{A}(\psi;\overline U)+|U\setminus V|\,.
	\end{equation}
	In view of the definition of $\psi'$ which is zero
 in $U$, we have\footnote{Notice that we use the precise integral formula \eqref{def:relaxed_area1} thanks to the boundary regularity of $U$. More precisely we have $\partial U\setminus l\subset \partial F_i\subset \cup_{j=1}^n\partial A_j$, where $A_j$'s are the Lispchitz sets whose union is $F_i$.}
	$\mathcal{A}(\psi';\overline U)=\int_l|{\psi}^+|d\mathcal H^1+|U|$ ($\psi^+$ denoting the trace of $\psi\res (B\setminus U)$ on the segment $l$) implying that \eqref{eq:claim_tris**} is equivalent to
	\begin{equation*}
\int_l|{\psi}^+|d\mathcal H^1
	\le\mathcal{A}(\psi;\overline U)-|V|\,.
	\end{equation*}
	Finally, if $\psi_U$ denotes the trace of $\psi\res U$ on $l$, we write $\mathcal{A}(\psi;\overline U)=\mathcal{A}(\psi;\overline U\setminus l)+\int_l|\psi^+-\psi_U|d\mathcal H^1$, and the expression above is equivalent to 
		\begin{equation}\label{eq:claim_tris*}
	\int_l|{\psi}^+|d\mathcal H^1
	\le\int_l|\psi^+-\psi_U|d\mathcal H^1+\mathcal{A}(\psi;\overline U\setminus l)-|V|\,.
	\end{equation}
	We now prove \eqref{eq:claim_tris*}.
Fix a Cartesian coordinate system $(x_1,x_2)$ so that $l$ belongs to
 the $x_1$-axis and $U$ belongs to the half-plane $\{x_2>0\}$. 	Let $u$ be an extension of $\psi$ in $\R\times(0,+\infty)$ which vanishes outside $U$.  
	Lemma \ref{lem_prelimiary},
	applied to $u$ with the ball $B_r=B$, implies
	$$	\int_l|{\psi}_U|d\mathcal H^1=\int_{\{x_2=0\}\cap B}u \;d\mathcal H^1\leq \mathcal A(u;B\cap (\R\times (0,+\infty)))-|E_B|\leq \mathcal A(\psi;\overline U\setminus l)-|V|.$$
	Here the last inequality follows by recalling that $\psi$ (and thus $u$) vanishes on $V$.
	From this and the inequality $\int_l|{\psi}^+|d\mathcal H^1\leq \int_l|\psi^+-\psi_U|d\mathcal H^1+\int_l|{\psi}_U|d\mathcal H^1
$ the proof of \eqref{eq:claim_tris*} is achieved, so that \eqref{eq:claim_bis*} follows.

	\textit{Step 2: Iterative case.}  
	We set $\partial({\rm conv}(F_i))
	\setminus\partial F_i=\cup_{j=1}^\infty l_j$ with $l_j$ mutually disjoint segments. 
	For every $h\ge1$ we define the pair  $(\psi_{h},G_{h})$ as follows:
	\begin{itemize}
		\item if $h=1$ 
		\begin{equation*}
		\psi_{1}:=\begin{cases}
		0&\text{in}\ U_1\\
		\psi&\text{otherwise,}
		\end{cases}\qquad \text{and}\quad
		G_{1}:=G\cup \overline U_1\,,
		\end{equation*}
	where $U_1$ is the open region enclosed between $\partial F_i$ and $l_1$. We also define $H_1:={\rm int}(\overline{F_i\cup U_1})$.
		\item if $h\ge2$ 
		\begin{equation*}
		\psi_{h}:=\begin{cases}
		0&\text{in}\ U_h\\
		\psi_{h-1}&\text{otherwise,}
		\end{cases}\quad\text{and}\quad
		G_{h}:=G_{h-1}\cup\overline U_h \,,
		\end{equation*}
		where $U_h$ is the open region enclosed between $\partial H_{h-1}$ and $l_h$ and $H_h:={\rm int}(\overline{H_{h-1}\cup U_h})$.
	\end{itemize}
	By construction each $H_{h}$ is 
 simply connected and piecewise Lipschitz, $H_{h}\subset H_{h+1}$,  $G_{h}\subset G_{h+1}\subset\overline\Om$ for every $h\ge1$, and moreover
	\begin{equation}\label{limit_E*}
	\lim_{h\to+\infty}	|H_{h}|=|{\rm conv}(F_i)|\,,\qquad 
	\lim_{h\to+\infty}	|G_{h}|=|G^\star_i|\,,
	\end{equation}
	where $G^\star_i:=\cup_{h=1}^\infty G_h=\cup_{j\neq i}F_j\cup {\rm conv}(F_i) $. For any $h\geq2$ we apply step 1, 
	 and after $h$ iterations we get 
	\begin{equation}\label{iterative_proc*}
	\mathcal{A}(\psi_{h};B)-|G_{h}|\le \mathcal{A}(\psi_{h-1};B)-|G_{h-1}|\le\cdot\cdot\cdot\le 	\mathcal{A}(\psi_{1};B)-|G_{1}|\le	\mathcal{A}(\psi;B)-|G|\,.
	\end{equation}
	In particular,
	\begin{equation*}
	|D\psi_{h}|(B)\le\mathcal{A}(\psi_{h};B)\le \mathcal{A}(\psi;B)+|G_{h}\setminus G|\le \mathcal{A}(\psi;B)+|\Om\setminus G|\,, 
	\end{equation*}
	for all $h\ge1$, and then we easily see that, up to a subsequence, $\psi_{h}\weakstar\psi^\star_i$ in $BV(B)$, where $\psi_i^*$ is defined as in \eqref{def:E_star,Psi_star}.
	Now the lower semicontinuity of $\mathcal{A}(\cdot;B)$ yields
	\begin{equation}\label{lsc_A*}
	\liminf_{h\to+\infty}\mathcal{A}(\psi_{h},B)\ge
	\mathcal{A}(\psi_i^\star;B)\,.
	\end{equation}
	Finally, gathering together \eqref{limit_E*}-\eqref{lsc_A*} we infer
	\begin{equation*}
	\mathcal{A}(\psi_i^\star;B)-|G^\star_i|\le
	\liminf_{h\to+\infty}\mathcal{A}(\psi_{h};B)-\lim_{h\to+\infty}|G_{h}|\le
	\mathcal{A}(\psi;B)-|G|\,.
	\end{equation*}
	This concludes the proof.	
\end{proof}
\begin{lemma}[\textbf{Reduction of energy, II}]\label{lem:convexification_bis}
Let $N\ge 1$, $F_1,\dots,F_N,G$ and $\psi$ be
as in Lemma \ref{lem:convexification}. 
Then there exist $\tilde n \in \{1, \dots, N\}$ and mutually disjoint closed convex sets $\widetilde F_1,\dots,\widetilde F_{\tilde n}\subset\overline\Om$  
such that 
\begin{equation}\label{A}
G\subset \bigcup_{i=1}^{\tilde n}\widetilde F_i=:G^\star\,,
\end{equation}
and 
		\begin{equation}\label{B}
			\mathcal{A}(\psi^\star;B)-|G^\star|-\mathcal{A}(\psi^\star;B\setminus \overline\Om)\le \mathcal A(\psi; B)-|G|-\mathcal{A}(\psi;B\setminus\overline{\Om})\,,
	\end{equation}
	where
		\begin{equation}\label{C}
	\psi^\star:=\begin{cases}
	0&\text{in}\ G^\star\\
	\psi&\text{otherwise}\,.
	\end{cases}
	\end{equation}
\end{lemma}
\begin{proof} 
	\textit{Base case: $(h=1)$.} We take the sets
	\begin{equation}\label{sets_base_case}
		{\rm conv}(F_1)\,,\ F_2\,,\ \dots\,,\ F_N\quad\text{and}\quad G_1^\star:=\cup_{i=2}^NF_{i}\cup{\rm conv}(F_1)\,,
	\end{equation}
	and let 
	\begin{equation*}
	\psi_1^\star:=\begin{cases}
0&\text{in}\ G_1^\star\\
\psi&\text{otherwise}\,.
	\end{cases}
	\end{equation*}
	Then by Lemma \ref{lem:convexification},
	\begin{align}\label{base_case}
	\mathcal A(\psi_1^\star;B)-|G_1^\star|-\mathcal A(\psi_1^\star;B\setminus \overline\Omega)\leq  
	\mathcal A(\psi;B)-|G|-\mathcal A(\psi;B\setminus \overline\Omega)\,.
	\end{align}
The next step is not necessary if $N=1$. 

	\textit{Iterative step: $(h>1)$}. Suppose $N >1$.
	 Let $1< h\leq m\le N$ be natural numbers, and let $F_{1,h},\dots,F_{m,h}$ be connected closed subsets of $\overline \Om$ 
with nonempty interior that satisfy the following property: There exists $1\le k < h$ such that:
	\begin{enumerate}[label=($\arabic*$)]
		\item \label{1}$F_{1,h},\dots,F_{k,h}$ are convex;
	\item\label{2} $F_{i,h}\cap F_{j,h}=\emptyset$ for all $i, j\ne k$, $i\ne j$. 
	\end{enumerate}
Notice that, if $m>1$, for $h=2$ the sets 
	\begin{equation*}
	F_{1,2}:={\rm conv}(F_1)\,,\ F_{2,2}:=F_2\,,\ \dots\,,\ F_{N,2}:=F_N\,,
\end{equation*}
in the base case satisfy \ref{1}, \ref{2} with $m=N$ and $k=1$.
	We then set $I_k:=\{1\le i\le m,\,i\ne k\colon F_{i,h}\cap F_{k,h}\ne\emptyset\}$ and construct a new family of sets using the following algorithm, distinguishing the two cases (a) and (b):
	\begin{itemize}
		\item[(a)] If $I_k=\emptyset$ we define the sets
		\begin{equation*}
				F_{i,h+1}:=\begin{cases}
	F_{i,h}& \text{for}\ i\neq{k+1}\\
	{\rm conv}(F_{k+1,h})&\text{for}\ i=k+1\,,
				\end{cases}
		\end{equation*}
		and 
		\begin{equation*}
		G_{h+1}^\star:=\cup_{i=1}^mF_{i,h+1}\,;
		\end{equation*}
		\item[(b)] if $I_k\ne\emptyset$, 
		up to relabelling the indices, we may assume that
		\begin{align*}
		I_k=\{k_1,k_1+1,\dots k_2\}\setminus\{k\},
	\end{align*} 
			for some $k_1\ne k_2$ with $1\leq k_1\le k\le k_2\leq m$,
		so that 
		$$\{1,\dots,m\}\setminus\{k\}\setminus I_k=\{1,\dots,k_1-1\}\cup\{k_2+1,\dots,m\}.$$ 
	Then we set
	\begin{equation*}
			F_{i,h+1}:=	\begin{cases}
		F_{i,h}& \text{for}\ i=1,\dots,k_1-1\\
		{\rm conv}(F_{k,h}\cup(\cup_{j\in I_k}F_{j,h}))& \text{for}\ i=k_1\\
		F_{i+k_2-k_1,h}& \text{for}\ i=k_1+1,\dots, m-k_2+k_1\,,
			\end{cases}
		\end{equation*}
		and
		\begin{equation*}
		G_{h+1}^\star:=\cup_{i=1}^{m-k_2+k_1}F_{i,h+1}\,.
		\end{equation*}
	\end{itemize}
In both cases (a) and (b) a direct check shows that the produced sets satisfy properties (1) and (2).

	We define also the function
		\begin{equation*}
	\psi_{h+1}^\star:=\begin{cases}
	0&\text{in}\ G_{h+1}^\star\\
	\psi_h^\star&\text{otherwise}\,.
	\end{cases}
	\end{equation*}
Then, by induction, for all $h$ we use Lemma \ref{lem:convexification}, and 
in view of \eqref{base_case} we infer
			\begin{align*}
		\mathcal A(\psi_{h+1}^\star;B)-|G_{h+1}^\star|-\mathcal A(\psi_{h+1}^\star;B\setminus \overline\Omega)\leq  &
		\mathcal A(\psi_h^\star;B)-|G_h^\star|-\mathcal A(\psi_h^\star;B\setminus \overline\Omega)\\
		&\leq  
		\mathcal A(\psi;B)-|G|-\mathcal A(\psi;B\setminus \overline\Omega)\,.
		\end{align*}
		\textit{Conclusion.} If $N=1$ it is sufficient to apply only the base case. If instead $N>1$ after a finite number $h^\star\le N$ of iterations we obtain a collections of mutually disjoint and closed convex sets $F_1:=F_{1,h^\star},\dots,F_{\tilde n}:=F_{\tilde n,h^\star}$ with $1\le \tilde n\le n$ such that 
		$$G\subset \cup_{i=1}^{\tilde n}F_i=:G^\star\,,$$ 
		and 
		\begin{equation*}\mathcal{A}(\psi^\star;B)-|G^\star|-\mathcal{A}(\psi^\star;B\setminus \overline\Om)\le \mathcal A(\psi; B)-|G|-\mathcal{A}(\psi;B\setminus\overline{\Om})\,,
		\end{equation*}
		with
		\begin{equation*}
		\psi^\star:=\psi_{h^\star}^\star=\begin{cases}
		0&\text{in}\ G^\star\\
		\psi&\text{otherwise}\,.
		\end{cases}
		\end{equation*}
\end{proof}

\begin{proof}[Proof of Theorem \ref{teo:convexification}]
	By Theorem \ref{thm:existence} it is enough to show that
	\begin{equation*}
	\inf_{(\sigma,\psi)
\in\largercomp} {\mathcal{F}}(\sigma,\psi)=
\inf_{(\sigma,\psi)\in \comp}
\mathcal{F}(\sigma,\psi)\,.
	\end{equation*}
Since from 
\eqref{eq:W_subset_hatW}
it follows
\begin{equation*}
\inf_{(\sigma,\psi)\in\largercomp}
{\mathcal{F}}(\sigma,\psi)\le\inf_{(\sigma,\psi)\in\comp }\mathcal{F}(\sigma,\psi),
\end{equation*}
we only need to show the converse inequality.
Take a pair $(\bar\sigma,\bar\psi)\in\largercomp$; we 
suitably modify $(\bar\sigma,\bar\psi)$ into a new pair 
$(\sigma,\psi)\in\comp$ satisfying
\begin{equation*}
	\mathcal{F}(\sigma,\psi)\le{\mathcal F}(\bar\sigma,\bar\psi),
\end{equation*}
and this will conclude the proof.

Let $E(\bar\sigma_1),\dots,E(\bar\sigma_n)$  be the closed sets with mutually disjoint interiors corresponding to $\bar\sigma$ (as in \ref*{ii'}
of Section \ref{subsec:setting_of_the_problem}) and let $G:=\cup_{i=1}^nE(\bar\sigma_i)$.  Consider the (closure of the) connected components $F_1,\dots, F_N$ of $G$, $N\leq n$ . 
Then by Lemma \ref{lem:convexification_bis} there exist $1\le \tilde n\le N$ and $\widetilde F_1,\dots,\widetilde F_{\tilde n}\subset\overline\Om$ mutually disjoint closed and convex satisfying \eqref{A}, \eqref{B} and \eqref{C}.
Therefore, by construction, for every $i=1,\dots,n$,
$q_i$ and $p_{i+1}$ belong to $\widetilde F_j$ for a unique $j\in\{1,\dots,\tilde n\}$.
For every $j=1,\dots,\tilde n$ we denote by
\begin{equation*}
q_{j_1},p_{j_1+1},\dots, q_{j_{n_j}},p_{j_{n_j}+1},
\end{equation*}
the ones that belong to $F_j$. Then we conclude by 
taking $(\sigma,\psi)\in\comp$  with  $\sigma:=(\sigma_1,\dots,\sigma_n)$ and
\begin{equation*}
\sigma_{j_k}([0,1])=\begin{cases}
\overline{q_{j_k}p_{j_k+1}}&\text{for}\quad k=1,\dots,n_j-1\\
\partial F_j\setminus\Bigl(\cup_{h=1}^{n_j}\partial^0_{j_h}\Om\Bigr)\cup
\Bigl(\cup_{h=1}^{n_j-1}\overline{q_{j_h}p_{j_h+1}}\Bigr)&\text{for}\quad k=n_j,
\end{cases}
\end{equation*}
 for every $j=1,\dots,\tilde n$ and $\psi:=\psi^\star$. 
 \end{proof}

%
\section{Regularity of minimizers}\label{sec:existence_of_a_smooth_minimizer}
In this section we investigate regularity properties of minimizers of $\mathcal F$.
The main result  reads as follows.
\begin{theorem}[\textbf{Structure of minimizers}]\label{thm:regularity}
Every minimizer $(\sigma,\psi)\in\comp$ of $\mathcal F$ in 
$\largercomp$, namely 
\begin{equation*}
{\mathcal{F}}(\sigma,\psi)=\min_{(s,\zeta)\in\largercomp} 
{\mathcal{F}}(s,\zeta)\,,
\end{equation*}
satisfies the following properties:
\begin{enumerate}[label=\arabic*]
	\item\label{1.}
Each connected component of $E(\sigma)$ is convex;
	\item\label{2.}  $\psi$ is positive and real analytic 
in $\Om\setminus E(\sigma)$;
	\item\label{3.}  If $\partial^D_i\Om$ is not 
a segment 
for some $i=1,\dots,n$, then $\partial E(\sigma)\cap
\partial_i^D\Om=\emptyset$, $\psi$ is continuous up to $\partial_i^D\Om$, and
	 $\psi=\varphi$ on $\partial_i^D\Om$;
	 \item\label{4.} If $\partial^D_i\Om$ is a segment  for some $i=1,\dots,n$, then either $\partial E(\sigma)\cap
	 \partial_i^D\Om=\emptyset$ or $\partial E(\sigma)\cap
	 \partial_i^D\Om=\partial_i^D\Om$. In the first case $\psi$ 
is continuous up to $\partial_i^D\Om$ and 
	 $\psi=\varphi$ on $\partial_i^D\Om$.
	 \end{enumerate}
 Moreover, there is a minimizer $(\sigma,\psi)\in\comp$ such that 
 \begin{enumerate}[label=5]
 		\item\label{5.}
$\Om\cap \partial E(\sigma)$ consists of a finite number of disjoint curves of class $C^\infty$, and $\psi$ is continuous and null on $\partial E(\sigma)\setminus \partial^D\Omega$.
	\end{enumerate}
\end{theorem}

\begin{remark}
If $\partial_i^D\Omega$ is a straight segment nothing ensures that $\partial E(\sigma)\cap \partial_i^D\Omega=\emptyset$. However, if this intersection is nonempty, then it must be 
$\partial_i^D\Omega\subset\partial E(\sigma)$. The 
prototypical example is given by the classical catenoid, as explained in the introduction (see also Figure \ref{figura1}) where, if the basis of the rectangle $\Omega=R_{2\ell}$ is large enough, a solution $\psi$ is identically zero, and $\partial^D\Omega\subset\partial E(\sigma)$. 

This also explains why in point \ref{5.} of Theorem \ref{thm:regularity} we write $\partial E(\sigma)\setminus \partial^D\Omega$, since $\partial^D\Omega$ might be partially included in $\partial E(\sigma)$ if $\partial_i^D\Omega$ is a segment (for some $i=1,\dots,n$).	
\end{remark}

For the reader convenience we divide the proof in a number of steps. 
\begin{lemma}\label{lem:analicity}
	Every minimizer $(\sigma,\psi)\in\comp$ of $\mathcal F$ in 
	$\largercomp$ satisfies \ref{1.}, \ref{2.} and
$\psi=\varphi$ on $\partial^D\Om\setminus\partial E(\sigma)$.

\end{lemma}
\begin{proof} Item \ref{1.} follows by Theorem \ref{teo:convexification}.
By \cite[Theorem 14.13]{Giusti:84} 
we also have that $\psi$ is real analytic in $\Om\setminus E(\sigma)$. Together with the strong maximum principle \cite[Theorem C.4]{Giusti:84}, this implies that,
in $\Om\setminus E(\sigma)$, 
either $\psi>0$ or $\psi\equiv0$. 
	On the other hand, 
since $\Omega$ is convex we can apply \cite[Theorem 15.9]{Giusti:84} and get that $\psi$ is continuous up to $\partial^D\Om\setminus\partial E(\sigma)$; in particular
	\begin{equation}\label{eq:boundary_value}
		\psi=\varphi>0\quad\text{on}\ \partial^D\Omega\setminus\partial E(\sigma)\,,
	\end{equation}
	which in turn implies $\psi>0$ in $\Omega\setminus E(\sigma)$ .

\end{proof}

\begin{lemma}\label{lem:plateau_solution}
	Let $\Gamma \subset \R^3$ 
be a rectifiable, simple, closed and non-planar 
curve satisfying the following properties:
	\begin{enumerate}[label=$(\arabic*)$]
		\item \label{(1)}$\Gamma\subset\partial(F\times\R)$ for some closed bounded convex set $F\subset\R^2$ with nonempty interior;
		\item \label{(2)}$\Gamma$ is symmetric with respect to the horizontal plane $\R^2\times\{0\}$;
		\item \label{(3)} There are an arc $\arc{pq}\subset\partial F$, with endpoints $p$ and $q$, and
		$f\in C^0( \arc{pq}\cup\{p,q\};[0,+\infty))$  such that $f$ is positive in $\arc{pq}$ and
		\begin{equation}\label{ass:sym-Gamma}
					\Gamma\cap\{x_3\ge0\}=\mathcal{G}_f\cup (\{p\}\times[0,f(p)])\cup (\{q\}\times[0,f(q)]).
		\end{equation}

	\end{enumerate}
	Let $S$  be a solution to the 
classical Plateau problem for $\Gamma$, i.e., a disc-type area-minimizing
surface among all disc-type surfaces spanning $\Gamma$. 
	Then:
	\begin{enumerate}[label=$(\arabic*')$]
		\item \label{1'}$\beta_{p,q}:=S\cap (\R^2\times\{0\})\subset F$ 
is a simple curve of class $C^\infty$ joining $p$ and $q$ such that $\beta_{p,q}\cap\partial F=\{p,q\}$;
		\item\label{2'} $S$ is symmetric with respect to $\R^2\times\{0\}$;
		\item \label{3'}The surface $S^+:=S\cap\{x_3\ge0\}$ is the graph of a 
function $\widetilde \psi\in W^{1,1}(U_{p,q})\cap C^0(\overline{U}_{p,q})$, where $U_{p,q}\subset{\rm int} (F)$ is the open region enclosed between $\arc{pq}$  and $\beta_{p,q}$. Moreover $\widetilde\psi$ is analytic in $U_{p,q}$;
		\item \label{4'}The curve $\beta_{p,q}$ is contained in the closed
convex hull of $\Gamma$, and $F\setminus U_{p,q}$ is convex.
	\end{enumerate}

\end{lemma}
\begin{remark}
	If  the function $f$ in \ref{(3)} is such that $f(p)=f(q)=0$ then \eqref{ass:sym-Gamma} becomes $\Gamma\cap\{x_3\ge0\}=\mathcal{G}_f$. 
For later convenience we prove Lemma \ref{lem:plateau_solution} under the more general assumption \ref{(3)}.

\end{remark}
\begin{proof} 
Even though several arguments are 
standard, we give the proof for completeness.

	\textit{Step 1: $\beta_{p,q}$ is a simple curve  joining $p$ and $q$. } \\
	Let $B_1\subset\R^2$ be the open unit disc centred at the origin and let
	$\Phi=(\Phi_1,\Phi_2,\Phi_3)\colon\overline{B}_1\to S\subset\R^3$ 
be a parametrization of $S$ with $\Phi(\partial B_1)=\Gamma$,
	that is harmonic, conformal, and therefore analytic in $B_1$, continuous up to $\partial B_1$.
	Further by \ref{(1)} it follows that 
	$\Phi$ is an embedding  and hence injective (see \cite{MY} and also \cite[page 343]{DHS}).\\
	By assumption \eqref{ass:sym-Gamma} we have $\{w\in\partial B_1\colon\Phi_3(w)=0\}=\{\Phi^{-1}(p,0),\Phi^{-1}(q,0)\}$, so that $\Phi_3$ changes sign only twice on $\partial B_1$. By applying Rado's lemma (see e.g. \cite[Lemma 2, page 295]{DHS}) to the harmonic function $\Phi_3$ we deduce that $\nabla\Phi_3\ne0$ in $B_1$ and in particular $\{w\in B_1\colon\Phi_3(w)>0\}$ and $\{w\in B_1\colon\Phi_3(w)<0\}$ 
are connected, and $\{w\in B_1\colon\Phi_3(w)=0\}$ is a  simple smooth curve in $B_1$ joining $\Phi^{-1}(p,0)$ and $\Phi^{-1}(q,0)$.
	By the injectivity of $\Phi$ we have that $S\cap(\R^2\times\{0\})=\Phi(\{w\in B_1\colon\Phi_3(w)=0\})$ is a simple smooth curve joining $p$ and $q$.
\smallskip
	
	\textit{Step 2: $S$ is symmetric with respect to the horizontal plane $\R^2\times\{0\}$.}\\
	By step 1 the sets $\{w\in \overline B_1\colon\Phi_3(w)\ge0\}$ and $\{w\in \overline B_1\colon\Phi_3(w)\le0\}$ 	are simply connected and the two surfaces
	$$S^+:=\Phi(\{w\in \overline B_1\colon\Phi_3(w)\ge0\})\,,\quad S^-:=\Phi(\{w\in \overline B_1\colon\Phi_3(w)\le0\})$$ 
have the topology of the disc.	We assume without loss of generality that $\mathcal{H}^2(S^+)\le \mathcal{H}^2(S^-)$.	Let
	$${\rm Sym}(S^+):=\{(x',x_3)\colon (x',-x_3)\in S^+\}\,,\quad \widetilde S:= S^+\cup{\rm Sym}(S^+)\,.$$
	Then $\widetilde S$ is symmetric surface of disc-type with $\partial\widetilde S =\Gamma$ and 
	\begin{equation*}
	\mathcal{H}^2(\widetilde S)=2\mathcal{H}^2(S^+)\le\mathcal{H}^2(S^+)+\mathcal{H}^2(S^-)=\mathcal{H}^2(S)\,.
	\end{equation*}
	In particular $\widetilde S$ is a symmetric solution to the 
Plateau problem for $\Gamma$. Further $S=\widetilde S$ on a relatively 
open subset of $S$; hence, since they are real 
analytic surfaces, they must coincide, $S=\widetilde S$.

	\smallskip
	\textit{Step 3: $S^+$ 
is the graph of a function $\widetilde \psi\in W^{1,1}(U_{p,q})\cap C^0(\overline U_{p,q})$.}\\
	To show this it is enough to 
check the validity of the following
	\smallskip 
	
	\textit{Claim:} Every vertical plane $\Pi$ is tangent to ${\rm int}(S)$ at most at one point.
	\smallskip

	In fact by step 2 this readily implies that ${\rm int}(S^+)$ has no points with vertical tangent plane and hence we can conclude.
	We prove the claim arguing by contradiction as in \cite[page 97]{vortex}, that is we assume there is a vertical plane $\Pi$ tangent to ${\rm int}(S)$ at $x'$ and $x''$ with $x'\ne x''$.
	We define the linear map  $d_\nu(x):=(x-x')\cdot\nu $ with $\nu$ a unit normal to $\Pi$, so that clearly $\Pi=\{x\in\R^3\colon d_\nu(x)=0\}$.
	Since $F$ is convex $\Pi\cap (\partial F\times\{0\})$ contains at most two points. By properties \ref{(1)}-\ref{(3)} each of these points is either the projection on the horizontal plane of one or two points of $\Pi\cap\Gamma$, or the projection on the horizontal plane of one of the vertical segments $\{p\}\times[0,f(p)]$ and $\{q\}\times[0,f(q)]$. Hence $\Pi\cap \Gamma$ contains either:
	\begin{itemize}
		\item 	 at most two points and a segment;
		\item two segments;
		\item four points.
	\end{itemize}
	Without loss of generality we restrict our analysis to the last case (the others are simpler to treat), namely we assume that there are four (clockwise ordered) points $w_1,\dots,w_4\in\partial B_1$ such that 
	$\Pi\cap\Gamma=\{\Phi(w_1),\dots,\Phi(w_4)\}$, that is $d_\nu\circ\Phi(w_i)=0$ for $i=1,\dots,4$. We may also assume  
$d_\nu\circ\Phi>0$ on $\arc{w_1w_2}\cup\arc{w_3w_4}$ and
	$d_\nu\circ\Phi<0$ on $\arc{w_2w_3}\cup\arc{w_4w_1}$.  Here $\arc{w_iw_j}$ denotes the relatively open arc in $\partial B_1$ joining $w_i$ and $w_j$ for $i,j\in\{1,\dots,4\}$.  
	\\
	Notice that the function $d_\nu\circ\Phi\colon \overline B_1\to \R$ is harmonic in $B_1$, continuous up to $\partial B_1$ and vanishes at $w_1,\dots,w_4$; hence, 
by classical arguments \cite[Section 437]{Nitsche:89} we see that 
	the set $\{w\in B_1\colon d_\nu\circ\Phi=0\}$, in a neighbourhood of $w':=\Phi^{-1}(x')$ (respectively $w'':=\Phi^{-1}(x'')$), is the union of a number $m\geq 2$ of analytic curves crossing at $w'$ (respectively $w''$). Thus near $w'$ and $w''$ the set $\{w\in B_1\colon d_\nu\circ\Phi(w)>0\}$ is the union of at least two disjoint open regions $A_{1,1}$, $A_{1,2}$ and $A_{2,1}$, $A_{2,2}$ respectively such that $\overline A_{1,1}\cap\overline A_{1,2}=\{w'\}$, $\overline A_{2,1}\cap\overline A_{2,2}=\{w''\}$. Moreover each $A_{i,j}$ belongs either to the  connected component of $\{w\in B_1\colon d_\nu\circ\Phi(w)>0\}$ containing  $\arc{w_1w_2}$ or to the one containing $\arc{w_3w_4}$. Up to relabelling the indices we have two possibilities.
	\begin{enumerate}[label= Case \arabic*:]
		\item $A_{1,1}$ and $A_{1,2}$ belong to the same connected component containing $\arc{w_1w_2}$. Then we can find two simple curves $\alpha_1$, $\alpha_2$ contained in $A_{1,1}$ and $A_{1,2}$ respectively, that connect $w'$ to a point in $\arc{w_1w_2}$ and such that the region enclosed by the curve $\alpha_1\cup\alpha_2$ intersects $\{w\in B_1\colon d_\nu\circ\Phi(w)<0\}$. Since $d_\nu\circ \Phi>0$ on $\alpha_1\cup\alpha_2$ by the maximum principle we have a contradiction.
		\item $A_{1,1}$ and $A_{2,1}$ belong to the connected component containing $\arc{w_1w_2}$ while $A_{1,2}$ and $A_{2,2}$ belong to the connected component containing $\arc{w_3w_4}$. Then we can find four simple curves $\alpha_{i,j}$ (with $i,j=1,2$) contained respectively in $A_{i,j}$, such that $\alpha_{1,1}$ (respectively $\alpha_{2,1}$) connects $w'$ (respectively $w''$) to a point in $\arc{w_1w_2}$ and $\alpha_{1,2}$ (respectively $\alpha_{2,2}$)  connects $w'$ (respectively $w''$) to $\arc{w_3w_4}$. Then the region enclosed by the curve $\cup_{i,j}\alpha_{i,j}$ intersects $\{w\in B_1\colon d_\nu\circ\Phi(w)<0\}$, while $d_\nu\circ\Phi>0$  on $\cup_{i,j}\alpha_{i,j}$, which again by the maximum principle gives a contradiction.
	\end{enumerate}
	Thus the claim follows.
	
	\smallskip
	\textit{Step 4: The curve $\beta_{p,q}$ is contained in the closed
		convex hull of $\Gamma$, and the set $F\setminus U_{p,q}$ is convex.}\\
	Let $\pi(\Gamma)\subset\partial F$ be the projection of $\Gamma$ onto the plane $\R^2\times\{0\}$. By \cite[Theorem 3, pag. 343]{DHS} the relative interior of $S$ is strictly contained in the convex hull of $\Gamma$, thus in particular
	the curve $\beta_{p,q}$ (respectively $\beta_{p,q}\setminus\{p,q\}$) is contained (respectively strictly contained) in the same half-plane (with respect to the line $\overline{pq}$) that contains $\pi(\Gamma)$. 
	
	Now, assume by contradiction that $F\setminus U_{p,q}$ is not convex. Then there are $p',q'\in\beta_{p,q}$ with the following properties:	
	\begin{itemize}
		\item The open region $U'$ enclosed by $\beta_{p,q}$ and the segment $\overline{p'q'}$ is non-empty and contained in $U_{p,q}$;
		\item the points $p$ and $q$ and the set $U'$ lie on the same side with respect to the line containing $\overline{p'q'}$.
	\end{itemize}
	Let then $d_W\colon\R^3\to\R$ be an affine function that vanishes on the vertical plane containing $\overline{p'q'}$ and is positive on the half-space $W^+$ containing $p,q$ and $U'$. 
	We now observe that $\Gamma\cap W^+$ is the union of two connected subcurves $\Gamma_1$ and $\Gamma_2$, containing $p$ and $q$ respectively. As a consequence $\Phi^{-1}(\Gamma_1)=\arc{w_1w_2}$ and $\Phi^{-1}(\Gamma_2)=\arc{w_3w_4}$ for some $w_1,w_2,w_3,w_4\in\partial B_1$ (clockwise oriented).\\
		On the other hand since $d_W>0$  on $U'$ we can find $t'\in\partial U'\setminus \overline{p'q'}$ such that $d_W\circ\Phi(\Phi^{-1}(t'))=d_W(t')>0$ with $\Phi^{-1}(t')\in B_1$.\\ 
Once again by the harmonicity of $d_W\circ\Phi\colon\overline{B}_1\to\R$ we deduce the existence of a curve $\alpha\subset\{w\in B_1\colon d_W\circ\Phi(w)>0\}$ joining $\Phi^{-1}(t')$ to one of $\arc{w_1w_2}$ and $\arc{w_3w_4}$. Hence $\Phi(\alpha)\subset \Phi( B_1)=\widetilde\psi(U_{p,q})$ is a curve joining $t'$ to one of $\Gamma_1$ and $\Gamma_2$, say $\Gamma_1$. 
This implies that the projection $\pi(\Phi(\alpha))$ of $\Phi(\alpha)$ onto the horizontal plane $\R^2\times\{0\}$ is a curve contained in $U_{p,q}$ that connects $t'$ to $\pi(\Gamma_1)$. So in particular, the curve $\pi(\Phi(\alpha))$ cannot be included in the half-plane $W^+$. But this contradicts the fact that $\alpha\subset\{w\in B_1\colon d_W\circ\Phi(w)>0\}$ (this is because the values of $d_W$ at a point $x$ and $\pi(x)$ are the same).
\end{proof}

We need also the following technical results on the distance function 
$\di_F$ from a convex set $F$.

\begin{lemma}\label{lem:prop_distanza_da_convesso}
	Let $F\subset\R^2$ be bounded, closed and convex.  
	 Then $\Delta\di_F\in L^{\infty}_{\rm loc}(\R^2\setminus F)\cap L^1(B\setminus F)$ for every ball $B$ with $F\subset\subset B$.
\end{lemma}
\begin{proof}
By \cite[Theorem 3.6.7 pag. 75]{Can-Sin} it follows that $\di_F\in C^{1,1}_{\rm loc}(\R^2\setminus F)$, hence $\nabla^2\di_F\in L^{\infty}_{\rm loc}(\R^2\setminus F;\R^{2\times 2})$.
Therefore we only have to check that $\Delta\di_F\in L^1(B\setminus F)$.\\
	Let $\eta>0$ be fixed sufficiently small.
Select $(f_k)_{k\in\mathbb{N}}\subset C^1_c(\R^2; \R^2)$ such that 
$f_k\to \nabla\di_F$ in $W^{1,1}(B
	\setminus F_{\eta/2}^+)$ as $k\to+\infty$.
	By the divergence theorem we have
	\begin{equation}\label{div-formula}
		\int_{B\setminus F^+_\eta}\Div f_k\;dx=\int_{
\partial B\cup\partial (F^+_\eta)}f_k\cdot\nu_\eta \;d\mathcal{H}^1,
	\end{equation}
	with $\nu_\eta$ the outer unit normal to $\partial B\cup\partial (F^+_\eta)$.
By taking the limit as $k\rightarrow \infty$ we get
	\begin{equation}\label{lim_k}
		\lim_{k\to+\infty}	\int_{B\setminus F^+_\eta}\Div f_k\;dx=\int_{B\setminus F^+_\eta}\Delta\di_F \;dx\,,
	\end{equation}
	and
		\begin{equation}\label{lim_k_1}
		\lim_{k\to+\infty}	\int_{\partial B\cup\partial (F^+_\eta)}f_k\cdot\nu_\eta\; d\mathcal{H}^1=	\int_{\partial B\cup\partial (F^+_\eta)}\nabla\di_F\cdot\nu_\eta\; d\mathcal{H}^1\,,
	\end{equation}
	where \eqref{lim_k_1} follows by using that $\partial(F^+_\eta)$ is of class $C^{1,1}$ and hence $f_k\res(\partial B\cup\partial (F^+_\eta))\to\nabla\di_F\res(\partial B\cup\partial (F^+_\eta))$ 
in $L^1(\partial B\cup\partial (F^+_\eta))$.
Since $\di_F$ is convex we have $\Delta\di_F\ge0$ a.e. in $\R^2\setminus F$, moreover
$|\nabla\di_F|=1$ in $\R^2\setminus F$; then gathering together \eqref{div-formula}, \eqref{lim_k}, \eqref{lim_k_1}  we have
\begin{equation*}
\int_{B\setminus F^+_\eta}|\Delta\di_F|\; dx=
\int_{B\setminus F^+_\eta}\Delta\di_F \;dx=
	\int_{\partial B\cup\partial (F^+_\eta)}\nabla\di_F\cdot\nu_\eta \;d\mathcal{H}^1\le \mathcal{H}^1(\partial B\cup\partial(F^+_\eta))\le C,
\end{equation*}
	with $C>0$ independent of $\eta$. 
By the arbitrariness of $\eta>0$, the thesis follows.
\end{proof}
\begin{cor}\label{cor:integration_by_parts}
	Let $U\subset\R^2$ be a bounded open set with Lipschitz boundary. Let $F\subset\R^2$ be closed and convex such that $U\cap F=\emptyset$ and let $\psi\in W^{1,1}(U)\cap L^\infty(U)\cap C^0(U)$. Then the following formula holds:
		\begin{equation*}
	\begin{split}
-	\int_{U}\psi\Delta\di_{F}\,dx
	=
	\int_{U}  \nabla\psi\cdot\nabla\di_{F}\,dx- \int_{\partial U}\psi\,\gamma\,d\mathcal{H}^1,
	\end{split}
	\end{equation*}
	where $\nu$ is the outer normal to $\partial U$ and $\gamma$ denotes the normal trace of $\nabla\di_F$ on $\partial U$. 
	\end{cor}

\begin{proof}
	We have $|\nabla \di_F|=1$ in $\R^2\setminus F$, moreover since $U\cap F=\emptyset$, by
	Lemma \ref{lem:prop_distanza_da_convesso} we deduce also $\Delta\di_F\in L^1(U)$. Therefore the thesis readily follows by applying \cite[Theorem 1.9]{A}.
\end{proof}

\begin{remark}\label{rem:traccia}
The normal trace $\gamma$ of $\nabla\di_F$ on $\partial F$ equals $1$ $\mathcal H^1$-a.e. on $\partial F$. Indeed, from Corollary \ref{cor:integration_by_parts} we have that for all $\varphi\in C^1_c(\R^2;\R^2)$ it holds
\begin{align}
-	\int_{\R^2\setminus F^+_\eta}\varphi\Delta\di_{F}\,dx
	&=
	\int_{\R^2\setminus F^+_\eta}  \nabla\varphi\cdot\nabla\di_{F}\,dx- \int_{\partial(F_\eta^+)}\varphi\,\gamma\,d\mathcal{H}^1\nonumber\\
	&=\int_{\R^2\setminus F^+_\eta}  \nabla\varphi\cdot\nabla\di_{F}\,dx- \int_{\partial(F_\eta^+)}\varphi\,d\mathcal{H}^1,\nonumber
\end{align}
where we have used that $\partial (F_\eta^+)$ being a level set of $\di_F$, it results $\nabla\di_F=\nu_\eta$ on it. 
Letting $\eta\rightarrow 0$ and using that $\Delta \di_F\in L^1(B\setminus F)$ for all balls $B$, we infer 
\begin{align}
-\int_{\R^2\setminus F}\varphi\Delta\di_{F}\,dx
&=
\int_{\R^2\setminus F}  \nabla\varphi\cdot\nabla\di_{F}\,dx- \int_{\partial F}\varphi\,d\mathcal{H}^1.\nonumber
\end{align}
By the arbitrariness of $\varphi$ and again by Corollary \ref{cor:integration_by_parts}, the claim follows.

\end{remark}

\begin{lemma}\label{lem:limite_di_traccia}
 Let $F\subset\overline\Om$ be closed and convex with non-empty interior, 
and let $\delta>0$.
	 Let $\psi\in W^{1,1}((F^+_{\delta}\setminus F)\cap\Om)\cap L^\infty((F^+_{\delta}\setminus F)\cap\Om)\cap C^0((F^+_{\delta}\setminus F)\cap\Om)$. Then 
	\begin{equation}\label{eq:lim_tracce}
		\lim_{\eps\to0^+,\;\eps<\delta} \int_{\Om\cap\partial(F^+_\eps)}\psi\,d\mathcal{H}^1=\int_{\Om\cap\partial F}\psi\,d\mathcal{H}^1\,.
	\end{equation}
\end{lemma}

\begin{proof}
	Let $\eps\in(0,\delta)$ and $T_\eps:=(F^+_\eps\setminus F)\cap \Omega$. Since  $T_\eps\cap F=\emptyset$, by Corollary \ref{cor:integration_by_parts} we get
	\begin{equation}\label{eq:lim_tracce*}
	\begin{split}
	-\int_{T_\eps}\psi\Delta\di_{F}\,dx
	=
	\int_{T_\eps}  \nabla\psi\cdot\nabla\di_{F}\,dx- \int_{\partial T_\eps}\psi\,\gamma\,d\mathcal{H}^1\,.
	\end{split}
	\end{equation}
	By Remark \ref{rem:traccia} 
we have
	\begin{equation}\label{eq:lim-tracce1}
	\begin{split}
	-\int_{T_\eps}\psi\Delta\di_{F}\,dx
	=&\int_{T_\eps}  \nabla\psi\cdot\nabla\di_{F}\,dx\\&+ \int_{\Om\cap\partial F}\psi\,d\mathcal{H}^1
	-\int_{\Om\cap\partial(F_\eps^+)}\psi\,d\mathcal{H}^1
	-\int_{((F^+_\eps)\setminus F)\cap\partial \Om}\psi\,\gamma\,d\mathcal{H}^1\,.
	\end{split}
	\end{equation}
Now
	\begin{equation}\label{eq:lim-tracce2}
	\lim_{\eps\to0^+} \Big|\int_{T_\eps}  \nabla\psi\cdot\nabla\di_{F}\,dx
\Big|	\le\lim_{\eps\to0^+} \int_{T_\eps} |\nabla \psi|\,dx=0\,,
	\end{equation}
	and
	\begin{equation}\label{eq:lim-tracce3}
	\lim_{\eps\to0^+} \Big|	\int_{(F^+_\eps\setminus F)\cap\partial\Om}\psi\,\gamma\,d\mathcal{H}^1\Big|\le 
	\lim_{\eps\to0^+}	\int_{(F^+_\eps\setminus F)\cap\partial\Om}\psi\,d\mathcal{H}^1=0\,.
	\end{equation}
	Moreover, since $\Delta\di_{F}\in L^1(T_\eps)$ by Lemma \ref{lem:prop_distanza_da_convesso}, we deduce also 
	\begin{equation}\label{eq:lim-tracce4}
	\lim_{\eps\to0^+}\Big| 	\int_{T_\eps}-\psi\,\Delta\di_{F}\,dx\Big|\le  \|\psi\|_{L^\infty}\lim_{\eps\to0^+}	\int_{T_\eps}|\Delta \di_{F}|\,dx=0\,.
	\end{equation}
	Finally gathering together \eqref{eq:lim-tracce1}-\eqref{eq:lim-tracce4} we infer \eqref{eq:lim_tracce}.
\end{proof}
\begin{remark} \label{rem:limite_di_traccia} Let $F$, $\delta$  and $\psi$ be as in Lemma \ref{lem:limite_di_traccia}. 
Let $\alpha$ be any connected component of $\Om\cap\partial F$,  and for every $0<\eps<\delta$  let $\alpha_\eps$ be the corresponding component of $\Om\cap\partial (F^+_\eps)$; namely, if $\pi_F$ is the orthogonal projection onto the convex closed set $F$, setting
$$\widehat\alpha_\eps:=\{x\in \partial (F^+_\eps):\pi_F(x)\in \alpha\},$$ 
then one has $\alpha_\eps:=\widehat\alpha_\eps\cap \Om$. 
 Arguing as in Lemma \ref{lem:limite_di_traccia}, we can show that 
\begin{equation*}
	\lim_{\eps\to0^+} 
	\int_{\alpha_\eps}\psi\,d\mathcal{H}^1=\int_{\alpha}\psi\,d\mathcal{H}^1\,.
\end{equation*}
\end{remark}

\begin{lemma}\label{lem:continuita-bordo-E}
Let $(\sigma,\psi)\in \comp$ be a minimizer as in 
Theorem \ref{teo:convexification}.
 Then there is a minimizer $(\widehat\sigma,\widehat\psi)\in \comp$ 
with the following properties:
\begin{enumerate}
	\item\label{lem1.}$\partial E(\widehat\sigma)\cap\partial\Om=\partial E(\sigma)\cap\partial\Om$;
	\item\label{lem2.}  $\widehat\psi$ is continuous and null
	 on $\Om\cap\partial E(\widehat\sigma)$.
\end{enumerate}
\end{lemma}
The second condition means essentially that $\widehat\psi$ vanishes on $\Om\cap\partial E(\widehat\sigma)$ when considering its trace from the side of $\Om\setminus E(\widehat\sigma)$.
\begin{proof} 

We know by Lemma \ref{lem:analicity} that $(\sigma,\psi)$ satisfies the following properties:
	\begin{itemize}
		\item
		Each connected component of $E(\sigma)$ is convex;
		\item $\psi$ is positive and real analytic 
		in $\Om\setminus E(\sigma)$;
		\item $\psi=\varphi$ on $\partial^D\Om\setminus\partial E(\sigma)$.
	\end{itemize}
	
In what follows we are going to modify $(\sigma,\psi)$ near each arc of $\partial E(\sigma)$ using an iterative argument in order to get a new minimizer $(\widehat\sigma,\widehat\psi)\in\comp$ that satisfies \ref{lem1.}-\ref{lem2.}. To this aim we denote by $F_1,\dots,F_k$ with $1\le k\le n$ the closed connected components of  $E(\sigma)$; we also set 
	$	\delta_0:=\min_{i\ne j}\dist(F_i,F_j)>0$. 	Moreover by the first property we deduce that $\Om\cap\partial E(\sigma)$ is the union of an at most countable family of pairwise disjoint arcs with endpoints in $\partial\Om$, i.e.,
\begin{equation*}
\Om\cap	\partial E(\sigma)=
\bigcup_{i=1}^k\bigcup_{j=1}^\infty \alpha_{i,j}\,,
\end{equation*}
where $\alpha_{i,j}$ is a connected component of $\Om\cap\partial F_i$ for $i\in\{1,
\dots,k\}$, $j\ge1$\footnote{Notice that at this stage we do not have any information about the geometry of the set $\partial\Om\cap\partial E(\sigma)$, and $\Om\cap\partial F_i$ could a priori be the union of infinitely many connected components.}.\\

		\textit{Step 1: Base case.}  Let  $\alpha$ be one of the connected components of $\Om\cap\partial F$, with $F:=F_i$ for some $i\in\{1,\dots,k\}$. 
		In this step we construct a new minimizer $(\sigma^\alpha,\psi^\alpha)\in\comp$ such that $\partial E(\sigma^\alpha)\cap\partial\Om=\partial E(\sigma)\cap\partial\Om$
		and 
		$\psi^\alpha$ is continuous and null on $\alpha'$, where
		$\alpha'\subset\Om\cap\partial E(\sigma^\alpha)$ is a suitable curve that replaces $\alpha$ and has the same endpoints as $\alpha$.\\

	For $\eps\in(0,\delta_0/2)$ we define the stripe $$\widehat T_\eps(\alpha):=\{x\in\Om\setminus F\colon\dist(x,\alpha)<\eps\}\subset F^+_\eps\setminus F\,,$$ and consider the planar curve $\alpha_\eps$ in $\overline\Om$ defined as in Remark \ref{rem:limite_di_traccia}. 
	Let $T_\eps(\alpha)$ be the connected component of $\widehat T_\eps(\alpha)$ whose boundary contains $\alpha_\eps$. Let $L_\eps$ be defined as
	$$ L_\eps:= \partial T_\eps(\alpha)\cap \partial\Om,
	$$
	so that in particular
	$\partial T_\eps(\alpha)=\alpha\cup\alpha_\eps\cup L_\eps \,.$
	Let $p,q\in \partial\Om$ be the endpoints of $\alpha$ (and then also the endpoints of  $\alpha_\eps\cup L_\eps$, which are independent of $\eps$).
	We define the curves
	$$\Gamma_\eps:=\Gamma_\eps^+\cup\Gamma_\eps^-\,,\quad
	\Gamma_\eps^+:=\mathcal{G}_{\psi\res\alpha_\eps}\cup \mathcal{G}_{\varphi\res L_\eps}\cup l^+\,,\quad \Gamma_\eps^-:=\mathcal{G}_{-\psi\res\alpha_\eps}\cup \mathcal{G}_{-\varphi\res L_\eps}\cup l^-\,,$$
	where 
	\begin{equation*}
		l^+:=(\{p\}\times[0,\varphi(p)])\cup (\{q\}\times[0,\varphi(q)])\,,\quad
			l^-:=(\{p\}\times[-\varphi(p),0])\cup(\{q\}\times[-\varphi(q),0])\,.
	\end{equation*}
	By observing that $L_\eps\subset\partial^D\Om\setminus\partial E(\sigma)$ and recalling that $\psi=\varphi$ on $\partial^D\Om\setminus\partial E(\sigma)$ we deduce that
	 $\Gamma_\eps$ is a closed non-planar curve in $\R^3$ that satisfies assumptions \ref{(1)}-\ref{(3)} of Lemma \ref{lem:plateau_solution}.
	In particular a solution $S_\eps$ to the classical Plateau problem corresponding to $\Gamma_\eps$ is a disc-type surface such that:
	\begin{enumerate}
		\item $\beta_{p,q}^\eps:=S_\eps\cap(\R^2\times\{0\})$ is a simple curve of class $C^\infty$ joining $p$ and $q$;
		\item $S_\eps$ is symmetric with respect to the horizontal plane;
		\item the surface $S^+_\eps:=S_\eps\cap\{x_3\ge0\}$ 
is the graph of a function $\psi_{p,q}^\eps\in W^{1,1}( U_{p,q}^\eps)
\cap C^0(\overline U_{p,q}^\eps)$, where $U_{p,q}^\eps\subset F\cup T_\eps(\alpha)$ is the open region enclosed between $\alpha_\eps\cup L_\eps$ and $\beta_{p,q}^{\eps}$;
		\item the curve $\beta_{p,q}^\eps$ is contained in the closed
		convex hull of $\Gamma_\eps$ and  $( F\cup T_\eps(\alpha))\setminus U_{p,q}^\eps$ is convex.
	\end{enumerate}
We would like to compare the area of $S_\eps^+$ with the area of the generalized graph of $\psi$ on $\overline {T_\eps(\alpha)}$. This is not immediate since, due to the fact that $\psi$ is just $BV$, we cannot, a priori, conclude that this generalized graph is of disc-type\footnote{This is due to the jump of $\psi$ on $\partial F$ which is, in general, not regular enough.}. Hence we proceed as follows. 
	We fix $\bar\eps\in(0,\delta_0/2)$; we claim that 
	\begin{equation}\label{eq:comparison-with-plateau}
	\mathcal{A}(\psi_{p,q}^{\bar\eps}; U_{p,q}^{\bar\eps})
	\le	
	\mathcal{A}(\psi; T_{\bar\eps}(\alpha))+
	\int_{\alpha}\psi\res{T_{\bar\eps}(\alpha)}\,d\mathcal{H}^1\,.
	\end{equation}
	Since $\psi$ is analytic in $T_{\bar\eps}(\alpha)\subset\Om\setminus E(\sigma)$, by Lemma \ref{lem:limite_di_traccia} and Remark \ref{rem:limite_di_traccia} it follows that 
	\begin{equation}\label{claim}
	\lim_{\eps\to0^+,\;\eps<\bar\eps} 
	\int_{\alpha_\eps}\psi\res{T_{\bar\eps}(\alpha)}\,d\mathcal{H}^1=\int_{\alpha}\psi\res{T_{\bar\eps}(\alpha)}\,d\mathcal{H}^1\,.
	\end{equation}
We take
	$$T_{\eps}^{\bar\eps}(\alpha):={T_{\bar\eps}(\alpha)\setminus \overline {T_{\eps}(\alpha)}}\quad\text{and}\quad Y_{\bar\eps}:=S_{\eps}\cup\mathcal{G}_{\psi\res T_\eps^{\bar\eps}(\alpha)} \cup\mathcal{G}_{-\psi\res T_\eps^{\bar\eps}(\alpha)}\,. $$
	Since $S_\eps$ is a disc-type surface and $\psi$ is analytic in $T_\eps^{\bar\eps}(\alpha)$ it turns out that $Y_{\bar\eps}$ is also a disc-type surface satisfying $\partial Y_{\bar\eps}=\Gamma_{\bar\eps}$. Therefore using that $S_{\bar\eps}$ and $S_\eps$ are solutions to the Plateau problems corresponding to $\Gamma_{\bar\eps}$ and $\Gamma_\eps$ respectively, we have
	\begin{equation*}
	\begin{split}
	\mathcal{H}^2(S_{\bar\eps})\le	\mathcal{H}^2(Y_{\bar\eps})&= 2\mathcal{H}^2(\mathcal{G}_{\psi\res T_\eps^{\bar\eps}(\alpha)} )+\mathcal{H}^2(S_\eps)\\
	&\le 2\mathcal{H}^2(\mathcal{G}_{\psi\res T_{\bar\eps}(\alpha)}) +
	2\int_{\alpha_\eps\cup L_\eps}\psi\res{T_{\bar\eps}(\alpha)}\, d\mathcal{H}^1\\
	&= 2\mathcal{H}^2(\mathcal{G}_{\psi\res T_{\bar\eps}(\alpha)}) + 2\int_{\alpha_\eps}\psi\res{T_{\bar\eps}(\alpha)}\, d\mathcal{H}^1+
	2\int_{ L_\eps}\psi\res{T_{\bar\eps}(\alpha)}\,d\mathcal{H}^1\,.
	\end{split}
	\end{equation*}
	Passing to the limit as $\eps\to0^+$, by \eqref{claim} and the fact that $\mathcal{H}^1(L_\eps)\to0$, we obtain
	\begin{equation*}
	\mathcal{H}^2(S_{\bar\eps})\le	
	2\mathcal{H}^2(\mathcal{G}_{\psi\res T_{\bar\eps}(\alpha)}) +
	2\int_{\alpha}\psi\res{T_{\bar\eps}(\alpha)} \,d\mathcal{H}^1,
	\end{equation*}
	which yields
	\begin{equation*}
	\mathcal{A}(\psi_{p,q}^{\bar\eps};U_{p,q}^{\bar\eps})=
	\mathcal{H}^2(S_{\bar\eps}^+)\le	
	\mathcal{H}^2(\mathcal{G}_{\psi\res T_{\bar\eps}(\alpha)}) +
	\int_{\alpha}\psi\res{T_{\bar\eps}(\alpha)}\,d\mathcal{H}^1=\mathcal{A}(\psi; T_{\bar\eps}(\alpha))+
	\int_{\alpha}\psi\res{T_{\bar\eps}(\alpha)}\,d\mathcal{H}^1,
	\end{equation*}
	and \eqref{eq:comparison-with-plateau} is proved.
	
We now define $E^\alpha:=(E(\sigma)\cup T_{\bar\eps}(\alpha))\setminus U^{\bar\eps}_{p,q}$ and 
	\begin{equation*}
\psi^\alpha:=\begin{cases}
0&\text{in}\ E^\alpha\\
\psi^{\bar\eps}_{p,q}& \text{in}\ U^{\bar\eps}_{p,q} \\
\psi& \text{otherwise}\,.
\end{cases}
\end{equation*}
	By \eqref{eq:comparison-with-plateau} and using that $U^{\bar\eps}_{p,q}\cup E^\alpha=E(\sigma)\cup T_{\bar\eps}(\alpha)$ we derive
\begin{equation}\label{eq:area_psi'}
\begin{split}
\mathcal{A}(\psi^\alpha;\Om)-|E^\alpha|
&= \mathcal A(\psi^{\bar\eps}_{p,q};U^{\bar\eps}_{p,q})+
\mathcal{A}(\psi; \Om\setminus(U^{\bar\eps}_{p,q}\cup E^\alpha))
\\
&= \mathcal A(\psi^{\bar\eps}_{p,q};U^{\bar\eps}_{p,q})+
\mathcal{A}(\psi; \Om\setminus(T_{\bar\eps}(\alpha)\cup E(\sigma)))
\\
&\le \mathcal A(\psi; T_{\bar\eps}(\alpha))+
\int_{\alpha}\psi\res{T_{\bar\eps}(\alpha)}\,d\mathcal{H}^1+\mathcal{A}(\psi; \Omega\setminus T_{\bar\eps}(\alpha))-|E(\sigma)|\\
&
=
\mathcal{A}(\psi; \Omega)-|E(\sigma)|\,.
\end{split}
\end{equation}
	It remains to construct $\sigma^\alpha\in \Sigma_{\rm conv}$. Without loss of generality we may assume 
$$\sigma_1([0,1]),\dots,\sigma_h([0,1])\subset F\quad\text{and}\quad \sigma_{h+1}([0,1]),\dots,\sigma_n([0,1])\not\subset F$$ 
	for some $h\le n$, notice that if $h=n$ the second family of curves is empty. Then we define $\sigma^\alpha:=(\sigma^\alpha_1,\dots,\sigma_h^\alpha,\sigma_{h+1},\dots,\sigma_n)\in\Lip([0,1];\overline\Om)^n$ where if $h>1$
	\begin{equation*}
			\sigma^\alpha_{i}([0,1])=\begin{cases}
				\overline{q_{i}p_{i+1}}&\text{for}\quad i=1,\dots,h-1\\
				\partial (F\cup T_{\bar\eps}(\alpha)\setminus U^{\bar{\eps}}_{p,q})\setminus\Bigl((\cup_{i=1}^{h}\partial^0_i\Om)\cup
				(\cup_{i=1}^{h-1}\overline{q_{i}p_{i+1}})\Bigr)&\text{for}\quad i=h,
			\end{cases} 
		\end{equation*}
	where $	\overline{q_{i}p_{i+1}}$ is the segment joining $q_i$ to $p_{i+1}$;	if instead $h=1$ we simply set 
	\begin{equation*}
	\sigma^\alpha_{1}([0,1])=	\partial (F\cup T_{\bar\eps}(\alpha)\setminus U^{\bar{\eps}}_{p,q})\setminus\partial_1^0\Om.
	\end{equation*}

Clearly the pair $(\sigma^\alpha,\psi^\alpha)$ belongs to $\comp$, and by \eqref{eq:area_psi'} it satisfies 
\begin{equation*}
\begin{split}
{\mathcal F}(\sigma^\alpha,\psi^\alpha)=
{\mathcal F}(\sigma, \psi)\,.
\end{split}
\end{equation*}
Moreover 
$
\partial E(\sigma^\alpha)\cap\partial\Om=\partial E(\sigma)\cap\partial\Om$ and $\psi^\alpha$ is continuous and null on $\alpha'$, where
\begin{equation}\label{gamma'}
\alpha':=\beta_{p,q}^{\bar\eps}\subset \Om\cap\partial E(\sigma^\alpha)\,.
\end{equation}
Summarizing, we have replaced the curve $\alpha$ with $\alpha'$, ensuring that the new function $\psi^\alpha$ is now continuous and null on $\alpha'$. 

	\textit{Step 2: Iterative case.}  In this step we construct a minimizer $(\widehat\sigma,\widehat\psi)\in\comp$ that satisfies the thesis by iterating step one at most a countable number of  times.\\
	We first consider $F=F_1$ and apply step 1 for each $\alpha_{1,j}$ with $j\ge1$.
	More precisely we define the pair $(\sigma_{1,j},\psi_{1,j})\in\comp$ as follows:
	\begin{itemize}
		\item if $j=1$ we set
		\begin{equation*}
		(\sigma_{1,1},\psi_{1,1}):=(\sigma^{\alpha_{1,1}},\psi^{\alpha_{1,1}} )\,,
		\end{equation*}
	where $(\sigma^{\alpha_{1,1}},\psi^{\alpha_{1,1}} )\in\comp$ is  a minimizer constructed as in step 1 with $\alpha=\alpha_{1,1}$;
		\item if  $j>1$ we set
		\begin{equation*}
			(\sigma_{1,j},\psi_{1,j}):=	(\sigma_{1,j-1}^{\alpha_{1,j}},\psi_{1,j-1}^{{\alpha}_{1,j}})\,,
		\end{equation*}
		where $(\sigma_{1,j-1}^{\alpha_{1,j}},\psi_{1,j-1}^{{\alpha}_{1,j}})\in\comp$ is a minimizer constructed as in step 1 with $(\sigma,\psi)=(\sigma_{1,j-1},\psi_{1,j-1})$ and $\alpha=\alpha_{1,j}$.
	\end{itemize}
Since
	$
		\mathcal{F}(\sigma_{1,j},\psi_{1,j})=\mathcal F(\sigma,\psi)
$
		for all $j\ge1$, by  Lemma \ref{lem:compactness} it follows that $(\sigma_{1,j},\psi_{1,j})$ converges to $(\sigma_1,\psi_1)\in\comp$ in the sense of Definition \ref{def:conv}. Moreover by construction we have that for every $j\ge1$ the pair $(\sigma_{1,j},\psi_{1,j})$  satisfies
		\begin{equation*}
		\partial E(\sigma_{1,j})\cap\partial\Om=	\partial E(\sigma)\cap\partial\Om\,,
		\end{equation*} 
		and $\psi_{1,j}$ is continuous and null on $\cup_{h=1}^j\alpha'_{1,h}\subset \Om\cap\partial E(\sigma_{1,j})\cap\partial F_1$,
		where $\alpha'_{1,h}$ are defined as in \eqref{gamma'}. As a consequence $(\sigma_1,\psi_1)$ satisfies 
			\begin{equation*}
			\partial E(\sigma_{1})\cap\partial\Om=	\partial E(\sigma)\cap\partial\Om\,,
		\end{equation*} 
		 and $\psi_1$ is continuous and null on $\cup_{j=1}^\infty\alpha'_{1,j}\subset \Om\cap\partial E(\sigma_{1})\cap\partial F_1$.
	Moreover 
		\begin{equation*}
	\Om\cap	\partial E(\sigma_1)=(\cup_{j=1}^\infty\alpha'_{1,j})\cup
		(\cup_{i=2}^k\cup_{j=1}^\infty \alpha_{i,j})\,,
	\end{equation*}
	Now repeating the argument above for the pair $(\sigma_1,\psi_1)$ and $i=2$  we obtain a  new minimizer $(\sigma_2,\psi_2)\in\comp$  satisfying 
		\begin{equation*}
		\partial E(\sigma_{2})\cap\partial\Om=	\partial E(\sigma)\cap\partial\Om\,,
	\end{equation*}
$\psi_2$ is continuous and null on $\cup_{j=1}^\infty(\alpha'_{1,j}\cup\alpha'_{2,j})\subset \Om \cap \partial E(\sigma_{1})\cap(\partial F_1\cup\partial F_2)$
and
	\begin{equation*}
		\Om \cap \partial E(\sigma_2)=(\cup_{i=1}^2\cup_{j=1}^\infty\alpha'_{i,j})\cup
		(\cup_{i=3}^k\cup_{j=1}^\infty \alpha_{i,j})\,.
	\end{equation*}
Iterating this process a finite number of times we finally get a minimizer $(\widehat\sigma,\widehat\psi)\in\comp$  with the required properties.
\end{proof}

We are finally in the position to conclude the proof of Theorem \ref{thm:regularity}.

\begin{proof}[Proof of Theorem \ref{thm:regularity}]
Let $(\sigma,\psi)\in\comp$ be any minimizer as in Theorem \ref{teo:convexification}.
By Lemma \ref{lem:analicity} we know that $(\sigma,\psi)$ satisfies properties \ref{1.}, \ref{2.} and
\begin{equation*}
	\psi=\varphi\quad\text{on}\quad\partial^D\Om\setminus\partial E(\sigma)\,.
\end{equation*}
Moreover by Lemma \ref{lem:continuita-bordo-E} there is a minimizer $(\widehat\sigma,\widehat\psi)\in\comp$ such that  
\begin{equation}\label{eq:prop1}
	\partial E(\widehat{\sigma})\cap\partial\Om=	\partial E(\sigma)\cap\partial\Om\,,
\end{equation}
and $\widehat\psi$ is continuous and null on $\Om\cap\partial E(\widehat\sigma)$.

It remains to show that if $\partial^D_i\Om$ is not straight for some $i=1,
\dots,n$, then $$\partial E(\sigma)\cap\partial^D_i\Om=\partial E(\widehat\sigma)\cap\partial^D_i\Om=\emptyset\,.$$ 
If instead $\partial^D_i\Om$ is straight for some $i=1,\dots,n$ we prove that property \ref{4.} holds.
Eventually we show that  there is a minimizer that satisfies property \ref{5.}.
This will be achieved in a number of steps.

\step 1 Assuming that there is $i\in\{1,\dots,n\}$ such that $\partial_i^D\Om$ is not straight, we show that $\partial_i^D\Om\cap E(\widehat\sigma)=\emptyset$. To prove this we proceed by analysing three different cases.\\

\textit{Case A}:	
Suppose, to the contrary, 
that there is a non-straight\footnote{Namely, $\arc{ab}$ is not contained in a line.} arc $\arc{ab}$ (with endpoints $a\ne b$) in $\partial_i^D\Om\cap\partial E(\widehat\sigma)$. Thus in particular $\arc{ab}\subset\cup_{j=1}^n\widehat\sigma_j([0,1])$. We may assume without loss of generality that $\arc{ab}\subset\widehat\sigma_1([0,1])$.\\
Then we consider the curves
\begin{equation}\label{def:curve}
\Gamma:=\Gamma^+\cup\Gamma^-\,,\quad
\Gamma^+:=\mathcal{G}_{\varphi\res\arc{ab}}\cup l^+\,,\quad	\Gamma^-:=\mathcal{G}_{-\varphi\res\arc{ab}}\cup l^-\,,
\end{equation}
where 
\begin{equation*}
l^+:=(\{a\}\times[0,\varphi(a)])\cup  (\{b\}\times[0,\varphi(b)])\,,\quad
l^-:=(\{a\}\times[-\varphi(a),0])\cup  (\{b\}\times[-\varphi(b),0])\,.
\end{equation*}
In this way $\Gamma$ satisfies the assumptions of Lemma \ref{lem:plateau_solution} and hence a solution $S$ to the Plateau problem spanning $\Gamma$ is a disc-type surface such that: 
\begin{enumerate}[label=\roman*.]
	\item\label{i} $\beta_{a,b}:=S\cap (\R^2\times\{0\})$ 
is a simple curve of class $C^\infty$ joining $a$ and $b$;
	\item\label{ii}$S$ is symmetric with respect to $\R^2\times\{0\}$;
	\item \label{iii} the surface $S^+:=S\cap\{x_3\ge0\}$ 
is the graph of a function $\psi_{a,b}\in W^{1,1}( U_{a,b})\cap C^0(\overline U_{a,b})$, where $U_{a,b}\subset E(\widehat\sigma_1)$  is the open region enclosed between $\arc{ab}$ and $\beta_{a,b}$;
	\item\label{iv} the curve $\beta_{a,b}$ is contained in the closed
	convex hull of $\Gamma$ and $E(\widehat\sigma_1)\setminus U_{a,b}$ is convex.
\end{enumerate}
The inclusion $U_{a,b}\subset E(\widehat\sigma_1)$ follows since $\arc{ab}\subset\widehat\sigma_1([0,1])$, $E(\widehat\sigma_1)$ is convex, and $S$ is contained in the convex envelope of $\Gamma$.
Furthermore by the minimality of $S$ one has
\begin{equation}\label{eq:contradiction}
\mathcal{A}(\psi_{a,b};U_{a,b})=\mathcal{H}^2(S^+)< \int_{\arc{ab}}\varphi\,d\mathcal{H}^1= \int_{\arc{ab}}|\widehat\psi-\varphi|\,d\mathcal{H}^1\,.
\end{equation}
Here the strict inequality follows since the vertical wall spanning $\Gamma$ given by $\{(x',x_3)\colon x'\in\arc{ab},\ x_3\in[-\varphi(x'),\varphi(x')]\}$ is a disc-type surface but, since $\arc{ab}$ is not a segment, cannot be a solution to the Plateau problem.
We now consider the pair $(\widetilde\sigma,\widetilde\psi)\in\comp$ given by 
\begin{equation}\label{def:sigma,psi}
\widetilde\sigma:=(\widetilde\sigma_1,\widehat\sigma_2,\dots,\widehat\sigma_n)\,,\qquad
	\widetilde\psi:=\begin{cases}
	0&\text{in }\widetilde E\,,\\
	\psi_{a,b}&\text{in }U_{a,b}\,,\\
\widehat	\psi&\text{otherwise}\,,
	\end{cases}
\end{equation}
where $\widetilde\sigma_1$ is such that $\widetilde\sigma_1([0,1])=(\widehat\sigma_1([0,1])\setminus\arc{ab})\cup\beta_{a,b}$ and   $\widetilde{E}:= E(\widehat\sigma)\setminus U_{a,b}=E(\widetilde\sigma)$.
Then noticing that $\widehat\psi=0$ in $U_{a,b}$, $E(\widehat\sigma)=E(\widetilde\sigma)\cup U_{a,b}$, and recalling \eqref{eq:contradiction}, we get
\begin{equation*}
\begin{split}
	{\mathcal F}(\widetilde\sigma,\widetilde\psi)&=\mathcal{A}(\widetilde\psi;\Om)-|E(\widetilde\sigma)|+\int_{\partial\Om}|\widetilde\psi-\varphi|\,d\mathcal{H}^1\\
	&=\mathcal{A}(\widehat\psi;\Om\setminus U_{a,b})+\mathcal{A}(\psi_{a,b};U_{a,b})-|E(\widetilde\sigma)|+\int_{\partial\Om}|\widetilde\psi-\varphi|\,d\mathcal{H}^1\\
	&=\mathcal{A}(\widehat\psi;\Om)+\mathcal{A}(\psi_{a,b};U_{a,b})-|E(\widehat\sigma)|+\int_{\partial\Om}|\widehat\psi-\varphi|\,d\mathcal{H}^1\\
	&<\mathcal{A}(\widehat\psi;\Om)-|E(\widehat\sigma)|+\int_{\partial\Om}|\widetilde\psi-\varphi|\,d\mathcal{H}^1+\int_{\arc{ab}}|\widehat\psi-\varphi|\,d\mathcal{H}^1\\
	&=\mathcal{A}(\widehat\psi;\Om)-|E(\widehat\sigma)|+\int_{\partial\Om}|\widehat\psi-\varphi|\,d\mathcal{H}^1=	{\mathcal F}(\widehat\sigma,\widehat\psi)\,,
\end{split}
\end{equation*}
where the penultimate equality follows from the fact that $\widetilde\psi$ is continuous and equal to $\varphi$ on $\arc{ab}$ while the traces of  $\widetilde\psi$ and $\widehat\psi$ coincide on $\partial\Om\setminus\arc{ab}$.
This contradicts the minimality of $(\widehat\sigma,\widehat\psi)$. \\

\textit{Case B}: Suppose by contradiction that the set $\partial_i^D\Om\cap\partial E(\widehat\sigma)$ contains  an isolated point $c$ or has a straight segment $\overline{cc'}$ as isolated connected component. 
Then there are two arcs $\arc{ab}\subset\partial_i^D\Om$ and $\arc{a'b'}\subset\partial E(\widehat\sigma)$ with either $a\neq a'$ or $b\neq b'$ (and with endpoints $a\ne b$ and $a'\ne b'$) such that $\overline{aa'}\cap\overline{bb'}=\emptyset$ and  $\arc{ab}\cap\arc{a'b'}=\{c\}$ (respectively $\arc{ab}\cap\arc{a'b'}=\overline{cc'}$). Notice also that, since $\partial_i^D\Om$ is not straight, the segment $\overline{cc'}$ does not coincide with $\partial_i^D\Om$ and hence the arc $\arc{ab}$ can be chosen so that it properly contains the segment $\overline{cc'}$. We consider the curves
\begin{equation}\label{eq:gamma}
\Gamma:=\Gamma^+\cup\Gamma^-\,,\quad \Gamma^+:=\mathcal{G}_{\varphi\res\arc{ab}}
\cup\mathcal{G}_{\widehat\psi\res\overline{aa'}}\cup\mathcal{G}_{\widehat\psi\res\overline{bb'}}
\,,\quad
\Gamma^-:=\mathcal{G}_{-\varphi\res\arc{ab}}
\cup\mathcal{G}_{-\widehat\psi\res\overline{aa'}}
\cup\mathcal{G}_{-\widehat\psi\res\overline{bb'}}\,.
\end{equation}
Notice that $\Gamma^\pm$ connect $a'$ to $b'$.
By applying again Lemma \ref{lem:plateau_solution} to the nonplanar curve $\Gamma$ and arguing as in case A we obtain the contradiction also in this case.\\

\textit{Case C}: More generally, assume by contradiction that both the sets $\partial_i^D\Om\cap\partial E(\widehat\sigma)$ and  $\partial_i^D\Om\setminus\partial E(\widehat\sigma)$ are nonempty. 
Then we can find a not flat  arc $\arc{ab}\subset\partial_i^D\Om$ such that the following holds\footnote{This is a consequence of the fact that $\arc{ab}\setminus\partial E(\widehat\sigma)$ is relatively open in $\arc{ab}$, so it is an at most countable union of disjoint relatively open arcs.}: there are pairs of points $\{c_j,d_j\}_{j\in \mathbb N}\subset\partial_i^D\Om\cap\partial E(\widehat\sigma)$ such that the arcs $\arc{ad_0}$, $\arc{c_0b}$, and $\{\arc{c_jd_j}\}_{j=1}^\infty$ are mutually disjoint and 
\begin{equation*}
\arc{ab}\setminus\partial E(\widehat\sigma)=\arc{ad_0}\cup(\cup_{j=1}^\infty\arc{c_jd_j})\cup\arc{c_0b}\,.
\end{equation*}
Without loss of generality, we might assume that all the points $c_j,d_j\in\widehat\sigma_1([0,1])$.
For all $j\ge1$ we denote by $V_j$ the region enclosed by $\arc{c_jd_j}$ and $\partial E(\widehat\sigma)$\footnote{These regions are simply connected since $c_j,d_j\in\widehat\sigma_1([0,1])$.}. We now argue as in case B and choose $a',b'\in \widehat\sigma_1([0,1])$. Additionally, let $V_0=V_0^a\cup V_0^b$, with $V_0^a$ (respectively $V_0^b$) be the region enclosed between $\partial E(\widehat\sigma)$ and $\overline{aa'}\cup\arc{ad_0}$ ($\partial E(\widehat\sigma)$ and $\overline{bb'}\cup\arc{c_0b}$, respectively). We finally define $\Gamma$ correspondingly, as in \eqref{eq:gamma}. 
Again by Lemma \ref{lem:plateau_solution} the solution $S$ to the Plateau problem corresponding to $\Gamma$ satisfies properties  \ref{i}-\ref{iv} with $a'$ and $b'$ in place of  $a$ and $b$ respectively.
Moreover by the minimality of $S$ for every $N\ge1$ there holds\footnote{The right-hand side is the area of the surface given by the (positive) subgraph of $\varphi$ on $\arc{ab}\setminus 
\cup_{j=1}^N\arc{c_jd_j}$ and the graph of $\widehat\psi$ on the region $\cup_{j=0}^NV_j$, which is of disc-type. 
To see this we use that the trace of $\widehat\psi$ on the 
subarcs of $\partial E(\widehat\sigma)$ between the points $c_j$ and $d_j$ 
is zero (and between $a'$ and $d_0$, and $d_0$ and $b'$).}
\begin{equation}\label{eq:contradiction_bis}
\mathcal{A}(\psi_{a',b'};U_{a',b'})=\mathcal{H}^2(S^+)\le \int_{\arc{ab}}\varphi\,d\mathcal{H}^1-\int_{\arc{ad_0}\cup\arc{c_0b}}\varphi\,d\mathcal{H}^1-\sum_{j=1}^N\int_{\arc{c_jd_j}}\varphi\,d\mathcal{H}^1+\sum_{j=0}^N \mathcal{A}(\psi,V_j)\,.
\end{equation}
In particular by taking the limit as $N\to\infty$ in \eqref{eq:contradiction_bis} we get
\begin{equation}\label{eq:contradiction_tris}
\mathcal{A}(\psi_{a',b'};U_{a',b'})=\mathcal{H}^2(S^+)\le \int_{\arc{ab}\setminus\partial E(\widehat\sigma)}\varphi\,d\mathcal{H}^1+\mathcal{A}(\widehat\psi,\cup_{j=0}^\infty V_j)\,.
\end{equation}
Let $(\widetilde\sigma,\widetilde\psi)\in\comp$ be defined as in \eqref{def:sigma,psi}, then observing that $\widehat\psi=0$ in $U_{a',b'}\setminus(\cup_{j=0}^\infty V_j)$, $E(\widehat\sigma)=E(\widetilde\sigma)\cup (U_{a',b'}\setminus\cup_{j=0}^\infty V_j)$ and using \eqref{eq:contradiction_tris}  we  deduce 
\begin{equation*}
\begin{split}
{\mathcal F}(\widetilde\sigma,\widetilde\psi)
&=\mathcal{A}(\widehat\psi;\Om\setminus U_{a',b'})+\mathcal{A}(\psi_{a',b'};U_{a',b'})-|E(\widetilde\sigma)|+\int_{\partial\Om}|\widetilde\psi-\varphi|\,d\mathcal{H}^1\\
&=\mathcal{A}(\widehat\psi;\Om\setminus(\cup_{j=0}^\infty V_j))+\mathcal{A}(\psi_{a',b'};U_{a',b'})-|E(\widehat\sigma)|+\int_{\partial\Om}|\widetilde\psi-\varphi|\,d\mathcal{H}^1\\
&\le\mathcal{A}(\widehat\psi;\Om\setminus(\cup_{j=0}^\infty V_j))-|E(\widehat\sigma)|+\int_{\partial\Om}|\widetilde\psi-\varphi|\,d\mathcal{H}^1+\int_{\arc{ab}\cap\partial E(\widehat\sigma)}\varphi\,d\mathcal{H}^1+\mathcal{A}(\widehat\psi;\cup_{j=0}^\infty V_j)\\
&=\mathcal{A}(\widehat\psi;\Om)-|E(\widehat\sigma)|+\int_{\partial\Om}|\widehat\psi-\varphi|\,d\mathcal{H}^1=	{\mathcal F}(\widehat\sigma,\widehat\psi)\,,
\end{split}
\end{equation*}
which in turn implies 
\begin{equation}\label{eq:contradiction_four}
{\mathcal F}(\widetilde\sigma,\widetilde\psi)\le
{\mathcal F}(\widehat\sigma,\widehat\psi)\,.
\end{equation}
To conclude we need to show that the inequality in \eqref{eq:contradiction_four} is strict. To this aim we choose $c\in\{c_j\}_{j=1}^\infty$. 
Consider the curves $\Gamma_1$ and $\Gamma_2$ defined as follows 

\begin{equation*}
	\Gamma_1:=\Gamma_1^+\cup\Gamma_1^-\,,\quad \Gamma_1^+:=\mathcal{G}_{\varphi\res\arc{ac}}
	\cup\mathcal{G}_{\widehat\psi\res\overline{aa'}}\cup l^+
	\,,\quad
	\Gamma_1^-:=\mathcal{G}_{-\varphi\res\arc{ac}}
	\cup\mathcal{G}_{-\widehat\psi\res\overline{aa'}}\cup l^-\,,
\end{equation*}
\begin{equation*}
	\Gamma_2:=\Gamma_2^+\cup\Gamma_2^-\,,\quad \Gamma_2^+:=\mathcal{G}_{\varphi\res\arc{cb}}
	\cup\mathcal{G}_{\widehat\psi\res\overline{bb'}}\cup l^+
	\,,\quad
	\Gamma_2^-:=\mathcal{G}_{-\varphi\res\arc{cb}}
	\cup\mathcal{G}_{-\widehat\psi\res\overline{bb'}}\cup l^-\,,
\end{equation*}
where 
\begin{equation*}
	l^+:=(\{c\}\times[0,\varphi(c)])\,,\quad
	l^-:=(\{c\}\times[-\varphi(c),0])\,.
\end{equation*}
 Let $S_1$ and $S_2$ be the solutions to the Plateau problem corresponding to $\Gamma_1$ and $\Gamma_2$ respectively, so that properties \ref{i}-\ref{iv} are satisfied with $c$ in place of $b'$ and $a'$ respectively.
By the minimality of $S$ we have
\begin{equation}\label{eq:ab}
\mathcal{A}(\psi_{a',b'},U_{a',b'})<\mathcal{A}(\psi_{a',c},U_{a',c})+\mathcal{A}(\psi_{c,b'},U_{c,b'})\,.
\end{equation}\label{eq:ac}

On the other hand by arguing as above\footnote{With the arc $\arc{ac}$ ($\arc{cb}$, respectively) in place of $\arc{ab}$.} we conclude 
\begin{equation}
\mathcal{A}(\psi_{a',c},U_{a',c})\le\int_{\arc{ac}\cup\partial E(\widehat\sigma)}\varphi\,d\mathcal{H}^1+\mathcal{A}(\widehat\psi,\cup_{j\in I_1} V_j\cup V^a_0)\,,
\end{equation}
and 
\begin{equation}\label{eq:cb}
\mathcal{A}(\psi_{c,b'},U_{c,b'})\le\int_{\arc{cb}\cup\partial E(\widehat\sigma)}\varphi\,d\mathcal{H}^1+\mathcal{A}(\widehat\psi,\cup_{j\in I_2} V_i\cup V^b_0)\,,
\end{equation}
where $I_1:=\{j\colon \arc{c_jd_j}\subset\arc{ac}\}$ and $I_2:=\{j\colon \arc{c_jd_j}\subset\arc{cb}\}$. Gathering together \eqref{eq:ab}-\eqref{eq:cb} we derive
\begin{equation*}
\mathcal{A}(\psi_{a',b'},U_{a',b'})<
\int_{\arc{ab}\cup\partial E(\widehat\sigma)}\varphi\,d\mathcal{H}^1+\mathcal{A}(\widehat\psi,\cup_{j=0}^\infty V_j)\,,
\end{equation*}
which in turn implies 
\begin{equation*}
{\mathcal F}(\widetilde\sigma,\widetilde\psi)<
{\mathcal F}(\widehat\sigma,\widehat\psi)\,,
\end{equation*}
and thus the contradiction.

\step 2 Assuming there is $i\in\{1,\dots,n\}$ such that $\partial_i^D\Om$ is a straight segment, and we show that either $\partial E(\widehat\sigma)\cap\partial_i^D\Om=\emptyset$ or $\partial E(\widehat\sigma)\cap\partial_i^D\Om=\partial_i^D\Om$.\\
Suppose by contradiction that  $\partial E(\widehat\sigma)\cap\partial_i^D\Om\ne\emptyset$ and also $\partial_i^D\Om\setminus \partial E(\widehat\sigma)\neq \emptyset$. Without loss of generality we can restrict to the case $\partial E(\widehat\sigma)\cap\partial_i^D\Om=\partial F\cap\partial_i^D\Om$ with $F$ any connected component of $E(\widehat\sigma)$. Since $F$ is convex and $\partial_i^D\Om$ is a segment $\partial F\cap\partial_i^D\Om$  has to be connected, i.e., it is either a single point  $a$ or a segment $\overline{aa'}\neq \partial_i^D\Om$.\\
In both cases we then consider a (small enough) ball $B$ centred at $a$ such that $B\cap E(\widehat\sigma)=B\cap F$ (in the second case we also require that the radius of $B$ is smaller than $\overline{aa'}$). 

If $\partial F\cap\partial_i^D\Om=\{a\}$ we let $\{p,q\}:=\partial B\cap\partial F$ and $\{b,c\}:=\partial B\cap\partial_i^D\Om$ (with $b,p$ and $c,q$ lying on the same side with respect to $a$).  Then 
we define the curves
\begin{equation*}
\Gamma:=\Gamma^+\cup\Gamma^-\,,\quad	\Gamma^+:=\mathcal{G}_{\varphi\res\overline{bc}}\cup \mathcal{G}_{\psi\res\arc{bp}}	\cup \mathcal{G}_{\psi\res\arc{cq}}\,,\quad \Gamma^-:=\mathcal{G}_{-\varphi\res\overline{bc}}\cup \mathcal{G}_{-\psi\res\arc{bp}}
	\cup \mathcal{G}_{-\psi\res\arc{cq}}\,,
\end{equation*}
where $\arc{bp}$, $\arc{cq}$ denote the arcs in $\partial B$ joining $b$ to $p$ and $c$ to $q$ respectively.

If $\partial F\cap\partial_i^D\Om=\overline{aa'}$ we let $\{p,q\}:=\partial B\cap\partial F$ and $\{b,c\}:=\partial B\cap\partial_i^D\Om$ where we identify $q$ and $c$. Then we consider the curves 
\begin{equation*}
	\Gamma:=\Gamma^+\cup\Gamma^-\,,\quad	\Gamma^+:=\mathcal{G}_{\varphi\res\overline{bc}}\cup \mathcal{G}_{\psi\res\arc{bp}}	\cup (\{c\}\times[0,\varphi(c)])\,,\quad \Gamma^-:=\mathcal{G}_{-\varphi\res\overline{bc}}\cup \mathcal{G}_{-\psi\res\arc{bp}}
	\cup (\{c\}\times[-\varphi(c),0])\,.
\end{equation*}
By applying again Lemma \ref{lem:plateau_solution} to $\Gamma$ and arguing as above we get the contradiction.

\step 3 We show that there is a minimizer $(\widetilde\sigma,\widetilde\psi)$ that satisfies property \ref{5.}.  \\
We first notice that $\widehat\psi$ is continuous and null on $\partial E(\widehat\sigma)\setminus\partial^D\Om$.
Moreover by steps 1 and 2 it follows that $\partial E(\widehat\sigma)\cap\Om$ is the union of a finite number of pairwise disjoint Lipschitz curves each of them joining each $p_i$ for $i=1,\dots,n$ to each of  the $q_j$ for some $j=1,\dots,n$. 
To conclude it is enough to replace each curve, without increasing the energy, with a smooth one having the same endpoints.
More precisely, let $\gamma$ be any of such curves. Reasoning as in the proof of  Lemma \ref{lem:continuita-bordo-E} step 1, we can replace $(\widehat\sigma,\widehat\psi)$ with a new minimizer $(\sigma^\gamma,\psi^\gamma)\in\comp$ such that $\partial E(\sigma^\gamma)\cap\partial\Om=\partial E(\sigma)\cap\partial\Om$
and 
$\psi^\gamma=0$ on $\gamma'$, where
$\gamma'\subset\partial E(\sigma^\gamma)\cap\Om$ is a suitable smooth curve that replaces $\gamma$ and has the same endpoints of $\gamma$.
In particular $\psi^\gamma$ is continuous and null on $\partial E(\sigma^\gamma)\setminus\partial^D\Om$.
Eventually iterating this procedure for each curve in $\partial E(\widehat\sigma)\setminus\partial\Om$ we can construct a new minimizer
$(\widetilde\sigma,\widetilde\psi)$  with the required properties.
\end{proof}

\subsection{The example of the catenoid containing a segment}\label{subsec_example}

Consider the setting depicted in Figure \ref{figura1mezzo}. Here 
$\Omega=R_{2l}=(0,2l)\times (-1,1)$, 
$n=1$, $\partial^D\Om=(\{0,2l\}\times (-1,1))\cup ((0,2l)\times \{-1\})$ and $\partial^0\Om=(0,2l)\times \{1\}$, $p=(0,1)$, $q=(2l,1)$.
The map $\varphi$ is $\varphi(w_1,w_2)=\sqrt{1-w_2^2}$, and thus 
vanishes on $[0,2l]\times \{-1\}$; for this reason this case is not covered by our analysis. However we can find a solution as in Theorem \ref{teo_main_intro} also in this case, by an approximation procedure.

Precisely, for $\eps>0$ and consider an approximating sequence $(\varphi)_\eps$ of  continuous Dirichlet data, with $\mathcal G_{\varphi_\eps}$ Lipschitz, which tends to $\varphi$ uniformly and satisfy $\varphi_\eps=0$ on $\partial^0\Om$, $\varphi_\eps>0$ on $\partial^D\Om$. Let $(\sigma_\eps,\psi_\eps)$ be a solution as in Theorem \ref{teo_main_intro} corresponding to the boundary datum $\varphi_\eps$; as $\mathcal F(\sigma_\eps,\psi_\eps)$ is equibounded\footnote{We can indeed always bound it from above by $|\Om|+\int_{\partial_D\Om}|\varphi_\eps|d\mathcal H^1$.}, arguing as in the proof of  Lemma \ref{lem:compactness}, we can see that, up to 
a subsequence, $((\sigma_\eps,\psi_\eps))$ tends to some $(\sigma,\psi)\in \mathcal W_{\textrm{conv}}$, which minimizes the functional $\mathcal F$ with Dirichlet condition $\varphi$.
In this case however we cannot guarantee that $\sigma$ does not touch $\partial^D\Om$, even if this is not a straight segment. 
This is essentially due to the presence of the 
portion $[0,2l]\times \{-1\}$ of $\partial \Om$ 
where $\varphi$ is zero, which does not allow to apply the arguments used in the proof of Theorem \ref{thm:regularity}.

In particular, it can be seen that if $l$ is large enough, the solution $(\sigma,\psi)$ splits and becomes degenerate, 
being $\psi\equiv0$ and the functional $\mathcal F$ pays only the area of two vertical half discs of radius $1$. Under a certain threshold instead the solution satisfies the regularity properties stated in Theorem \ref{thm:regularity}, and in particular $\psi=\varphi$ on $\partial^D\Omega$, and $\sigma$ is the graph of a smooth convex function passing through $p$ and $q$. We refer to \cite{vortex} for details and comprehensive proofs of these facts.

\section{Comparison with the parametric Plateau problem: The case $n=1,2$}
\label{sec:comparison_with_the_parametric_Plateau_problem:the_case_n=1,2}
In this section we compare the solutions provided by Theorems \ref{thm:existence} and \ref{thm:regularity} with the solutions to the classical Plateau problem in parametric form.
Specifically, motivated by the example of the catenoid, we will restrict our analysis to the classical 
disc-type and annulus-type Plateau problem.
These configurations correspond to the cases $n=1$ and $n=2$ respectively, i.e.,
the Dirichlet boundary $\partial^D\Om$ is either an open arc  or the union of two open arcs of $\partial\Om$ with disjoint closure.
 Due to the highly involved geometric arguments, 
we do not discuss the case $n>2$, which requires further investigation.

Thus, in this section we assume $n=1,2$. We first discuss the case $n=1$ which is a consequence of  Lemma \ref{lem:plateau_solution}, and then the case $n=2$.

\subsection{The case $n=1$}\label{sec:the_case_n=1}

Let $n=1$. Let $p_1,q_1\in\partial\Om$, $\partial^D\Om=\partial^D_1\Om$, $\varphi$ be as in Section \ref{subsec:setting_of_the_problem} and 
consider the space curve $\gamma_1:=\mathcal{G}_{\varphi\res\partial_1^D\Om}$ joining $p_1$ to $q_1$.
We define the curve 
$$\Gamma:=\gamma_1\cup {\rm Sym}(\gamma_1),$$
where ${\rm Sym}(\gamma_1):=\mathcal{G}_{-\varphi\res\partial_1^D\Om}$,
and 
consider the classical Plateau problem in parametric form 
spanning $\Gamma$. More precisely we look for a solution to
\begin{equation}\label{plateau_1}
m_1(\Gamma):=\inf_{\Phi\in\mathcal{P}_1(\Gamma)}\int_{B_1}|\partial_{w_1}\Phi\wedge\partial_{w_2}\Phi|dw,
\end{equation}
where 
\begin{equation}\label{def:adm_Phi}
	\begin{split}
	\mathcal{P}_1(\Gamma):=\Big\{\Phi\in H^1( B_1;\R^3)\cap C^0(\overline B_1;\R^3)&\text{ such that }\Phi\res\partial B_1\colon\partial B_1\to\Gamma \\&\text{ is a weakly monotonic parametrization of $\Gamma$}
\Big\}.
	\end{split}
\end{equation}
Then the following holds:
\begin{theorem}[\textbf{The disc-type Plateau problem ($n=1$)}]\label{plateau:n1}
 Let $\Phi\in\mathcal{P}_1(\Gamma)$ be a solution to \eqref{plateau_1} and let 
$$S^+:=\Phi(\overline B_1)\cap \{x_3\ge0\}\qquad\text{and}\qquad S^-:=\Phi(\overline B_1)\cap \{x_3\le0\}.$$
Then there exists a minimizer $(\sigma,\psi)\in \comp$ of $\mathcal F$ in 
$\mathcal W$ 
satisfying properties \ref{1.}-\ref{5.} of  Theorem \ref{thm:regularity} and such that
\begin{equation}\label{eq:graph}
	S^{\pm}=\mathcal G_{\pm\psi\res (\overline{\Om\setminus E(\sigma)})}.
\end{equation}
Conversely let $(\sigma,\psi)\in \comp$ be a minimizer of $\mathcal F$ in $\mathcal W$ satisfying properties \ref{1.}-\ref{5.} of  Theorem \ref{thm:regularity}. Then the disc-type surface 
$$S:=\mathcal G_{\psi\res (\overline{\Om\setminus E(\sigma)})}\cup\mathcal G_{-\psi\res (\overline{\Om\setminus E(\sigma)})}$$
is a solution to the classical Plateau problem associated to $\Gamma$, i.e., there is $\Phi\in\mathcal{P}_1(\Gamma)$
 solution to \eqref{plateau_1} such that $\Phi(\overline{B_1})=S$.
\end{theorem}

\subsection{The case $n=2$}\label{sec:6.2}

Let $n=2$. Let $\Om$, $p_1,q_1,p_2,q_2\in\partial\Om$, $\partial^D\Om$, $\partial_1^D\Om$, $\partial_2^D\Om$, $\varphi$ be as in Section \ref{subsec:setting_of_the_problem} and 
consider the space curve $\gamma_i:=\mathcal{G}_{\varphi\res\partial_i^D\Om}$ joining $p_i$ to $q_i$ for $i=1,2$.
 We define the curves
\begin{equation*}
	\Gamma_1:=\gamma_1\cup{\rm Sym}(\gamma_1),\qquad 	\Gamma_2:=\gamma_2\cup{\rm Sym}(\gamma_2),
\end{equation*}
where ${\rm Sym}(\gamma_i):=\mathcal{G}_{-\varphi\res\partial_i^D\Om}$ for $i=1,2$.
We consider the 
classical Plateau problem in parametric form spanning the curve 
$$
\Gamma:=\Gamma_1\cup\Gamma_2.
$$
Precisely we set $\openannulus\subset\R^2$ to be an open annulus 
enclosed between two concentric circles $C_1$ and $C_2$,
 and we look for a solution to 
\begin{equation}\label{catenoid_plateau}
m_2(\Gamma):=\inf_{\Phi\in\mathcal{P}_2(\Gamma)}\int_{\openannulus}|\partial_{w_1}\Phi\wedge\partial_{w_2}\Phi|dw,
\end{equation}
where
\begin{equation*}
	\begin{split}
\mathcal{P}_2(\Gamma):=\Big\{\Phi\in H^1(\openannulus;\R^3)\cap C^0(\overline\Sigma_{\rm ann};\R^3) &\text{ such that }\Phi(\partial \openannulus)=\Gamma\text{ and }\Phi\res C_j:C_j\rightarrow \Gamma_j\\&\text{ is a weakly monotonic parametrization of $\Gamma_j$ for $j=1,2$}\Big\}.
	\end{split}
\end{equation*}

Here the crucial assumption that we require is that the curves $\Gamma_j$ 
have the orientation inherited by the orientation\footnote{Once we fix 
an orientation of $\partial \Omega$, the orientation of the graph 
$\mathcal G_\varphi$ 
of $\varphi$ is inherited, since 
$\mathcal G_\varphi$ 
is standardly defined as the push-forward of the current of integration on $\partial_D\Om$ by the map $x\mapsto(x,\varphi(x))$. } of the graph of $\varphi$ on $\partial_j^D\Om$.\\
Due to the specific geometry of  $\Gamma$ we can appeal to Theorem \ref{myexistence} below (which is a consequence of \cite[Theorem 1 and Theorem 5]{MY}) to deduce the existence of a minimizer. This might not be true for a more general $\Gamma$. To this purpose for $j=1,2$ we consider
the minimization problem defined in \eqref{plateau_1} for the curve $\Gamma_j$, namely
 \begin{align}\label{minimizationSigma}
 m_1(\Gamma_j)=\inf_{\Phi\in\mathcal{P}_1(\Gamma_j)}\int_{B_1}|\partial_{w_1}\Phi\wedge\partial_{w_2}\Phi|dw,
 \end{align}
with $\mathcal{P}_1(\Gamma_j)$ defined as in \eqref{def:adm_Phi}.

\begin{remark}
	By standard arguments one sees that $m_2(\Gamma)\le m_1(\Gamma_1)+m_1(\Gamma_2)$. Indeed, two disc-type surfaces can be joined by a very thin tube (with arbitrarily small area) in order to change the topology of the two discs into an annulus-type surface. 
\end{remark}

\begin{definition}\label{def:MY}
Let $\Phi\in\mathcal{P}_2(\Gamma)$ be a solution to \eqref{catenoid_plateau}. 
We say that $\Phi$ is a $\mathcal{MY}$ solution to \eqref{catenoid_plateau} if 
$\Phi$ is harmonic, conformal, and it is  an embedding. In particular, in such a case, $m_2(\Gamma)=\mathcal{H}^2(\immclosedann)$.
\end{definition}
\begin{theorem}[\textbf{Meeks and Yau}]\label{myexistence}
	Suppose $m_2(\Gamma)<m_1(\Gamma_1)+m_1(\Gamma_2)$. Then there exists a $\mathcal{MY}$ solution $\Phi\in\mathcal{P}_2(\Gamma)$ to \eqref{catenoid_plateau}. Furthermore, every minimizer of \eqref{catenoid_plateau} is a $\mathcal{MY}$ solution.
\end{theorem}

\begin{proof}
See \cite{MY}.
\end{proof}
This result allows to prove the following:
\begin{theorem}[\textbf{The annulus-type Plateau problem ($n=2$)}]
	\label{thm:comparison-with-classical-plateau}
	The following hold:
	\begin{enumerate}[label=$(\roman*)$]
		\item \label{1plateau}	Suppose $m_2(\Gamma)<m_1(\Gamma_1)+m_1(\Gamma_2)$.
		Let $\Phi\in\mathcal{P}_2(\Gamma)$ 
be a  $\mathcal{MY}$ solution to \eqref{catenoid_plateau} and let 
		\begin{equation*}
S:=\immclosedann, \qquad
			S^+:=S\cap\{x_3\ge0\},\qquad
			S^-:=S\cap\{x_3\le0\}.
		\end{equation*}
		Then there exists a minimizer $(\sigma,\psi)\in \comp$ 
of $\mathcal F$ in $\mathcal W$ satisfying properties \ref{1.}-\ref{5.} of  Theorem \ref{thm:regularity} and such that
		\begin{equation}\label{param_to_nonparam}
			S^{\pm}=\mathcal G_{\pm\psi\res (\overline{\Om\setminus E(\sigma)})}.
		\end{equation}
	\item\label{2plateau} Suppose  $m_2(\Gamma)= m_1(\Gamma_1)+m_1(\Gamma_2)$. For $j=1,2$ let $\Phi_j\in\mathcal{P}_1(\Gamma_j)$ be a solution to \eqref{minimizationSigma} and let $S_j:=\Phi_j(\overline B_1)$. Let also
	\begin{equation*}
		S^+:=(S_1\cup S_2)\cap\{x_3\ge0\}\qquad\text{and}\qquad
		S^-:=(S_1\cup S_2)\cap\{x_3\le0\}.
	\end{equation*}
	Then $S_1\cap S_2=\emptyset$ and there exists a minimizer $(\sigma,\psi)\in \comp$ of $\mathcal F$ in $\mathcal W$ satisfying properties \ref{1.}-\ref{5.} of  Theorem \ref{thm:regularity} and such that  \eqref{param_to_nonparam} holds.
	\item \label{3plateau}	
Conversely, let $(\sigma,\psi)\in \comp$ be a minimizer of $\mathcal F$ in $\mathcal W$ satisfying properties \ref{1.}-\ref{5.} of  Theorem \ref{thm:regularity}. Then the surface 
	$$S:=\mathcal G_{\psi\res (\overline{\Om\setminus E(\sigma)})}\cup\mathcal G_{-\psi\res (\overline{\Om\setminus E(\sigma)})}$$
	is either an annulus-type surface or the union of two disjoint disc-type surfaces, and is a solution to the classical Plateau problem associated to $\Gamma$. More precisely, either there is a $\mathcal{MY}$ solution $\Phi\in\mathcal{P}_2(\Gamma)$ to \eqref{catenoid_plateau}  with 
$S=\immclosedann$, or there are $\Phi_j\in\mathcal{P}_1(\Gamma_j)$ solutions to \eqref{minimizationSigma} for $j=1,2$, such that $S=\Phi_1(\overline B_1)\cup \Phi_2(\overline B_1)$ and $\Phi_1(\overline B_1)\cap \Phi_2(\overline B_1)=\emptyset$.
	\end{enumerate}\end{theorem}

\subsection{Toward the proofs of Theorems \ref{plateau:n1} 
and \ref{thm:comparison-with-classical-plateau}: preliminary lemmas}

In order to prove Theorems \ref{plateau:n1} and 
\ref{thm:comparison-with-classical-plateau}, we collect some technical lemmas.

\begin{lemma}\label{lem_tec1}
	Let $n=2$, and $(\sigma,\psi)\in\comp$ be a minimizer 
of $\mathcal F$ in 
$\admclassconv$
satisfying properties \ref{1.}-\ref{5.} of Theorem \ref{thm:regularity}. 
	\begin{itemize}
\item[(a)] 
Suppose that $\overline{\Om\setminus E(\sigma)}$ is simply connected. 
Then there exists an injective map $\Phi\in W^{1,1}(\openannulus;\R^3)\cap C^0(\overline{\openannulus};\R^3)$ such that 
$$\immclosedann=\mathcal G_{\psi\res (\overline{\Om\setminus E(\sigma)})}\cup\mathcal G_{-\psi\res (\overline{\Om\setminus E(\sigma)})},$$
and $\Phi\res C_j\colon C_j\to \Gamma_j$ is a weakly monotonic parametrization of $\Gamma_j$ for $j=1,2$.
\item[(b)] Suppose that $\Om\setminus E(\sigma)$ 
consists of two connected components, whose closures  $F_1$ and $F_2$ are disjoint, 
with 
$F_j \supseteq \partial^D_j\Om$ for $j=1,2$. 
Then there exist two injective maps $\Phi_1,\Phi_2\in W^{1,1}(B_1;\R^3)\cap C^0(\overline{B_1};\R^3)$ such that 
$$\Phi_j(\overline{B_1})=\mathcal G_{\psi\res {F_j}}
\cup\mathcal G_{-\psi\res{F_j}},\qquad j=1,2,$$
and $\Phi_j\res \partial B_1\colon\partial B_1\to \Gamma_j$ is a weakly monotonic parametrization of $\Gamma_j$ for $j=1,2$.
	\end{itemize}
	\end{lemma}
\begin{proof}
(a). Since $\overline{\Om\setminus E(\sigma)}$ is simply connected, the 
maps 
\begin{equation}\label{eq:Psi_tilde}
\widetilde \Psi^\pm\in W^{1,1}({\Om\setminus E(\sigma)};\R^3)
\cap C^0(\overline{\Om\setminus E(\sigma)};\R^3), 
\qquad \widetilde \Psi^\pm(p):=(p,\pm\psi(p)),
\end{equation}
 are disc-type parametrizations of  $\mathcal G_{\pm\psi\res (\overline{\Om\setminus E(\sigma)})}$.
	
	Now, using a homeomorphism of class $H^1$ between $\overline{\Om\setminus E(\sigma)}$ and a disc, we can parametrize\footnote{For instance,
we can consider a (flat) disc-type Plateau  solution 
spanning $\partial(\Om\setminus E(\sigma))$. 
Then we can employ a Lipschitz homeomorphism between the disc and the half-annulus.}
$\overline{\Om\setminus E(\sigma)}$ with a half-annulus, obtained as the region enclosed between two concentric half-circles 
with endpoints $A_1,A_2,A_3,A_4$ (in the order) on the same diameter, and the two segments $\overline{A_1A_2}$ and $\overline{A_3A_4}$. Then we construct a parametrization $\Psi^+$ of $\mathcal G_{\psi\res (\overline{\Om\setminus E(\sigma)})}$ from the half-annulus, such that $\Psi^+(A_1)=(q_1,0)$, $\Psi^+(A_2)=(p_2,0)$, $\Psi^+(A_3)=(q_2,0)$, $\Psi^+(A_4)=(p_1,0)$, and sending weakly monotonically the two half-circles 
into $\gamma_1$ and $\gamma_2$, and the two segments into $\sigma_1$ and 
$\sigma_2$, respectively.
Similarly, we construct  a parametrization $\Psi^-$ of $\mathcal G_{-\psi\res (\overline{\Om\setminus E(\sigma)})}$ from 
another copy of a half-annulus, just by setting $\Psi^-:=\textrm{Sym}(\Psi^+)$ (the symmetric of $\Psi^+$ with respect to the plane containing $\Omega$). \\
	Eventually, glueing the two half-annuli 
along the two segments, we obtain a parametrization $\Phi$ 
of $\mathcal G_{\psi\res (\overline{\Om\setminus E(\sigma)})}
\cup \mathcal G_{-\psi\res (\overline{\Om\setminus E(\sigma)})}$ 
defined on  $\overline\Sigma_{\rm ann}$.
	By the continuity of $\psi$ on $\partial^D\Omega$  we have that $ \Phi$ parametrizes $\Gamma_i$ on $C_i$, $i=1,2$.
	
(b). 	It is sufficient to argue as in case (a),  
by replacing $\Omega\setminus E(\sigma)$ in turn with $F_1$ and $F_2$ and 
$\openannulus$ with $B_1$ to find $\Phi_1$ and $\Phi_2$, respectively. 
\end{proof}

\begin{lemma}\label{lem_tec2}
Let $n=2$, and  $(\sigma,\psi)\in\comp$ be a minimizer of 
$\mathcal F$ in 
$\mathcal W$ 
satisfying properties \ref{1.}-\ref{5.} of Theorem \ref{thm:regularity}.
	\begin{itemize}
\item[(a)] 
Suppose that $\overline{\Om\setminus E(\sigma)}$ is simply connected and 
\begin{equation}\label{eq:lem_tec2}
\mathcal H^2(\mathcal G_{\psi\res (\overline{\Om\setminus E(\sigma)})}\cup\mathcal G_{-\psi\res (\overline{\Om\setminus E(\sigma)})})\leq m_2(\Gamma).
\end{equation}
Let $\Phi$ be the parametrization given by Lemma \ref{lem_tec1} (a).
Then there exists a reparametrization of the annulus $\openannulus$ 
such that, using it to reparametrize $\Phi$, the corresponding map (still denoted by $\Phi$)
belongs to $\mathcal P_2(\Gamma)$ and solves  \eqref{catenoid_plateau}. 
\item[(b)] Suppose that $\Om\setminus E(\sigma)$ 
consists of two connected components whose $F_1$ and $F_2$ are disjoint, and    
$F_j \supset \partial_j^D \Om$ for $j=1,2$, and
$$\mathcal H^2(\mathcal G_{\psi\res {F_j}}
\cup\mathcal G_{-\psi\res {F_j}})\leq m_1(\Gamma_j),
\qquad j=1,2.$$
 Let $\Phi_1,\Phi_2$  be the maps given by Lemma \ref{lem_tec1} (b). Then, for $j=1,2$, there is a reparametrization of $\Phi_j$ belonging to $\mathcal P_1(\Gamma_j)$ and solving \eqref{minimizationSigma}.
	\end{itemize}
\end{lemma}
\begin{proof}
(a). 
Fix a point $\widetilde p\in \Om\setminus E(\sigma)$ and set $\widetilde \Psi^+_k:=\widetilde\Psi^+\res H_k$, 
where $\widetilde \Psi$ is defined in \eqref{eq:Psi_tilde} and $H_k$ is the 
connected component of 
	$$\widetilde H_k:=\{p\in \overline{\Om\setminus E(\sigma)}:\textrm{dist}(p,\partial(\overline{\Om\setminus E(\sigma)}) )\geq 1/k\}
$$
containing $\widetilde p$.

	For $k \in \mathbb N$ large enough  $H_k$ is  
simply connected  with rectifiable boundary. 
In particular $\widetilde \Psi^+_k$ parametrizes a disc-type surface, and 
using the regularity of $ \psi$ in ${\Om\setminus E(\sigma)}$,  
it follows that $\widetilde \Psi^+_k$ is Lipschitz continuous. 
Furthermore, 
$\widetilde \Psi^+_k\res \partial H_k$ 
parametrizes a Jordan curve,
and these curves
converge, in the sense of Fr\'echet (see \cite[Theorem 4, Section 4.3]{DHS}) 
as $k\rightarrow +\infty$, to the curve having image
 $\widetilde \Psi^+(\partial ({\Om\setminus E(\sigma)}) ) )=:\lambda$.  
Notice that
	\begin{align}\label{lambda}
	\lambda=\sigma_1\cup\sigma_2\cup\gamma_1\cup \gamma_2.
	\end{align}
Call $\lambda_k$ the image of the curve  given by 
$\widetilde \Psi^+_k\res \partial H_k$.
	Let $\mathcal{P}_1(\lambda_k)$, $\mathcal{P}_1(\lambda)$, $m_1(\lambda_k)$, $m_1(\lambda)$ be defined as in \eqref{def:adm_Phi} and \eqref{plateau_1} with $\lambda_k$ and $\lambda$ in place of $\Gamma$ respectively.
	Up to reparametrizing $\overline B_1$ (see footnote 15),
 $\widetilde\Psi^+_k$ belongs to $\mathcal{P}_1(\lambda_k)$, therefore 
	$$\mathcal H^2(	\mathcal G_{\psi\res H_{k}})=\int_{H_k}|\partial_{w_1}\widetilde \Psi_k^+\wedge \partial_{w_2}\widetilde \Psi_k^+ |dw\geq m_1(\lambda_k)\qquad\forall k\geq1.$$
We claim that   equality holds in the previous expression, namely
	\begin{align}\label{claim_speranza}
	\mathcal H^2(	\mathcal G_{\psi\res H_k})=m_1(\lambda_k)\qquad \forall k\geq1.
	\end{align}
	Indeed, assume by contradiction that $\mathcal H^2(	\mathcal G_{\psi\res H_{k_0}})>m_1(\lambda_{k_0})$ for some $k_0\geq1$,
and pick $\delta>0$ with
	\begin{equation}\label{eq:absurd}
\mathcal H^2(	\mathcal G_{\psi\res H_{k_0}})\geq \delta+
m_1(\lambda_{k_0}).
	\end{equation}
Take $\Phi_{k_0}\in \mathcal{P}_1(\lambda_{k_0})$
	a solution to $m_1(\lambda_{k_0})$.
	For $k> k_0$, as $H_{k_0}\subset H_k$, by a 
glueing argument\footnote{This is done, for instance, by glueing an external annulus to a disc, and using $\Phi_{k_0}$ from the disc, and a reparametrization of $\mathcal G_{\psi\res(H_k\setminus H_{k_0})}$ from the annulus.}, we can find $\Phi_k\in\mathcal{P}_1(\lambda_k)$ such that $\Phi_k(\overline{B_1})=\Phi_{k_0}(\overline{B_1})\cup \mathcal G_{\psi\res(H_k\setminus H_{k_0})}$. Thus by \eqref{eq:absurd} we have
$$
\begin{aligned}
\mathcal H^2(	\mathcal G_{\psi\res H_{k}})
\geq & 
\delta + m_1(\lambda_{k_0})
+
\mathcal H^2(	\mathcal G_{\psi\res (H_k\setminus H_{k_0})})
\\
= & 
\delta + \mathcal H^2(\Phi_{k_0}(\overline B_1))
+
\mathcal H^2(	\mathcal G_{\psi\res (H_k\setminus H_{k_0})})
\geq
\delta+ m_1(\lambda_k)
\qquad \forall k> k_0.
\end{aligned}
$$
	Letting $k\rightarrow +\infty$, since $\lambda_k\rightarrow\lambda$ in the sense of Fr\'echet, we have $m_1(\lambda_k)\rightarrow m_1(\lambda)$ \cite[Theorem 4, Section 4.3]{DHS}.
In particular, from the previous inequality we infer
	\begin{equation*}
	\mathcal{F}(\sigma,\psi)=\mathcal{H}^2(\mathcal G_{\psi\res(\overline{\Om\setminus E(\sigma)})})\ge\delta+m_1(\lambda).
	\end{equation*}
	Hence we conclude $$\mathcal{H}^2(\mathcal G_{\psi\res(\overline{\Om\setminus E(\sigma)})}\cup \mathcal G_{-\psi\res(\overline{\Om\setminus E(\sigma)})})\geq 2\delta+2m_1(\lambda)\geq  2\delta+m_2(\Gamma),$$
	which contradicts \eqref{eq:lem_tec2}.
	In the last inequality we have used that $2m_1(\lambda)\geq m_2(\Gamma)$; this follows from the fact that a disc-type parametrization of a minimizer for $m_1(\lambda)$ can be reparametrized on a half-annulus (as in the proof of Lemma \ref{lem_tec1}), and glued with another reparametrization of it on the other half-annulus, so to obtain a parametrization of an annulus-type surface spanning $\Gamma$ which is admissible for \eqref{catenoid_plateau}.	
	 Hence claim \eqref{claim_speranza} follows.
	 
	Now, since $\psi$ is Lipschitz continuous on $\overline H_k$, for 
all $k\in \mathbb N$ there exists a parametrization $\Psi_k\in H^1(B_1;\R^3)\cap C^0(\overline {B_1};\R^3)$ with $\Psi_{k}(\partial B_1)=\lambda_{k}$ monotonically which solves the classical disc-type 
Plateau problem {spanning} $\lambda_k$ and such that 
	$$\Psi_k(B_1)=\mathcal G_{\psi\res H_k}.$$
	Letting $k\rightarrow +\infty$
and using that the Dirichlet energy of $\Psi_k$ equals the area of $\mathcal G_{\psi\res H_k}$, we 
conclude that $(\Psi_k)$ tends to a map $\Psi\in H^1(B_1;\R^3)\cap C^0(\overline {B_1};\R^3)$ 
with $\Psi(\partial B_1)=\lambda$ {weakly} monotonically, and that is a solution of the classical disc-type Plateau problem with $$\Psi(\overline B_1)=\mathcal{G}_{\psi\res(\overline{\Om\setminus E(\sigma)})}.$$
	Arguing as in the proof of Lemma \ref{lem_tec1} we finally get a parametrization $\Phi:
\overline\Sigma_{\rm ann}\rightarrow \R^3$ which belongs to $\mathcal P_2(\Gamma)$ and parametrizes $\mathcal{G}_{\psi\res(\overline{\Om\setminus E(\sigma)})}\cup \mathcal{G}_{-\psi\res(\overline{\Om\setminus E(\sigma)})}$.
	This concludes the proof of (a). 

(b). 	It is sufficient to argue as in case (a),  
by replacing $\Omega\setminus E(\sigma)$ in turn with $F_1$ and $F_2$ and 
$\openannulus$ with $B_1$ to find $\Phi_1$ and $\Phi_2$, respectively. 
\end{proof}

We can now start the proof of Theorems \ref{plateau:n1} and 
\ref{thm:comparison-with-classical-plateau}. 

\subsection{Proof of Theorem \ref{plateau:n1}}

\begin{proof}[Proof of Theorem \ref{plateau:n1}] Let $\Phi\in\mathcal{P}_1(\Gamma)$ be a solution to \eqref{plateau_1}. 
	The curve $\Gamma$ satisfies the assumptions of Lemma \ref{lem:plateau_solution}, hence the minimal disc-type surface $S:=\Phi(\overline{B_1})$ satisfies the following properties:
	\begin{itemize}
		\item $\beta_{p_1,q_1}:=S\cap (\R^2\times\{0\})\subset \overline \Om$ 
		is a simple curve of class $C^\infty$ 
joining $p_1$ and $q_1$ and such that $\beta_{p_1,q_1}\cap\partial\Om=\{p_1,q_1\}$;
		\item $S$ is symmetric with respect to $\R^2\times\{0\}$;
		\item the surface $S^+=S\cap\{x_3\ge0\}$ is the graph of a 
		function $\widetilde\psi\in W^{1,1}(U_{p_1,q_1})\cap C^0(\overline U_{p_1,q_1})$, where $U_{p_1,q_1}\subset \Om$ is the open region enclosed between $\partial_1^D\Om$  and $\beta_{p_1,q_1}$. Moreover $\widetilde\psi$ is analytic in $U_{p_1,q_1}$;
		\item the curve $\beta_{p_1,q_1}$ is contained in the closed
		convex hull of $\Gamma$, and $\Om\setminus U_{p_1,q_1}$ is convex.
	\end{itemize}
	
	Let  $(\sigma,\psi)\in\comp$ be given by
	$$\sigma:=\sigma_1\qquad\text{and}\qquad
	\psi:=\begin{cases}
	0& \text{in } \Om\setminus U_{p_1,q_1}\\
	\widetilde\psi & \text{in } U_{p_1,q_1},
	\end{cases}$$
	where $\sigma_1([0,1])=\beta_{p_1,q_1}$.
	Clearly \eqref{eq:graph} holds. Moreover $\mathcal H^2(S)=2\mathcal F(\sigma,\psi)=m_1(\Gamma)$. It remains to show that this is a minimizer of $\mathcal F$.
	Let $(\sigma',\psi')\in\comp$ be a minimizer  of $\mathcal F$  that satisfies properties \ref{1.}-\ref{5.} of  Theorem  \ref{thm:regularity} and  consider the disc-type surface with boundary $\Gamma$ given by $S':=\mathcal{G}_{\psi'\res(\overline{\Om\setminus E(\sigma')})}
	\cup \mathcal{G}_{-\psi'\res(\overline{\Om\setminus E(\sigma')})}$. 
	Since $(\sigma,\psi)$ is admissible for $\mathcal F$, we deduce
	$$\mathcal{H}^2(S')=2\mathcal F(\sigma',\psi')\leq m_1(\Gamma).$$
	Then we are in the hypotheses of Lemma \ref{lem_tec2} and so  there is a parametrization
	$\Phi'\in \mathcal P_1(\Gamma) $ with $\Phi'(\overline B_1)=S'$.
	By minimality of  $(\sigma',\psi')$ and of $S$ we have
	\begin{equation}
	\mathcal H^2(S)\le	\mathcal{H}^2(S')=2\mathcal{F}(\sigma',\psi')\le 2\mathcal{F}(\sigma,\psi)=\mathcal H^2(S).
	\end{equation}
	Hence $(\sigma,\psi)$ is a minimizer of $\mathcal F$ in $\mathcal W$ and $\Phi'$ is  a solution to \eqref{plateau_1}.
	
	Conversely, let $(\sigma,\psi)\in \comp$ be a solution that 
satisfies properties \ref{1.}-\ref{5.} of  Theorem \ref{thm:regularity}.
 Let $\widetilde\Phi$ be a solution to \eqref{plateau_1}; then we can find $(\widetilde\sigma,\widetilde\psi)\in \mathcal W$ whose 
doubled graph $\widetilde S=\mathcal G_{\widetilde\psi\res (\overline{\Om\setminus E(\widetilde\sigma)})}\cup\mathcal G_{-\widetilde\psi\res (\overline{\Om\setminus E(\widetilde\sigma)})}$ satisfies 
	$$ \mathcal H^2(S)=2\mathcal F(\sigma,\psi)\leq2\mathcal F(\widetilde\sigma,\widetilde\psi)=\mathcal H^2(\widetilde S)=m_1(\Gamma).$$ Arguing as before  we  find a map $\Phi\in \mathcal P_1(\Gamma)$ parametrizing $S$. 
	We conclude  that $\Phi$ is a solution to \eqref{plateau_1}, and the theorem is proved.
\end{proof}

\subsection{Proof of Theorem \ref{thm:comparison-with-classical-plateau}}
The proof of 
Theorem \ref{thm:comparison-with-classical-plateau} 
is much more involved, so we divide it in a number of steps. We start with a result (which can be seen as the counterpart of
Lemma \ref{lem:plateau_solution} for the Plateau problem defined in \eqref{catenoid_plateau}) that will be crucial to prove \ref{1plateau}. 
In what follows we denote by $\pi\colon\R^3\to\R^2\times\{0\}$ the orthogonal projection. 
\begin{theorem}\label{crucial_teo}
Suppose $m_2(\Gamma)<m_1(\Gamma_1)+m_1(\Gamma_2)$ and let
$\Phi\in\mathcal{P}_2(\Gamma)$ be a $\mathcal{MY}$ solution to \eqref{catenoid_plateau}.
 Then the minimal surface $\immclosedann$ satisfies the following properties: 
	\begin{enumerate}[label=$(\arabic*)$]
		\item \label{crucial1} The set $\pi(\immclosedann)$ 
is simply connected in $\overline\Omega$; $\Om\cap\partial \pi(\immclosedann)$
		consists of two disjoint embedded curves   $\beta_1$ and $\beta_2$ of class $C^\infty$ joining $q_1$ to $p_2$, and $q_2$ to $p_1$, respectively. Moreover, the closed region $E_i$ 
enclosed between $\partial_i^0\Om$ and $\beta_i$, $i=1,2$, is convex;
		\item\label{crucial2} $\immclosedann$ is symmetric with respect to the plane $\R^2\times\{0\}$;
		\item\label{crucial3}  $\immclosedann\cap(\R^2\times\{0\})=\beta_1\cup\beta_2$; 
		\item\label{crucial4} $S^+:=\immclosedann\cap \{x_3\ge0\}$ is Cartesian. Precisely,  it is 
the graph of a function $\widetilde\psi\in W^{1,1}({\rm int}(\pi(\immclosedann)))\cap C^0({\pi(\immclosedann)})$. 
	\end{enumerate}
\end{theorem}
The proof of Theorem \ref{crucial_teo} is a consequence of Lemmas \ref{lemma_step1}, \ref{lemma_step4}, \ref{lemma_step2}, \ref{lemma_step3}, \ref{lem:construction-of-parametrization}, and \ref{lemma_step5} below.
\begin{lemma}\label{lemma_step1}
	Suppose $m_2(\Gamma)<m_1(\Gamma_1)+m_1(\Gamma_2)$ and let
	$\Phi\in\mathcal{P}_2(\Gamma)$ be a $\mathcal{MY}$ solution to \eqref{catenoid_plateau}.
	Then $\pi(\immclosedann)$ is a simply connected region in $\overline\Omega$ containing $\partial_1^D\Om\cup\partial_2^D\Om$.
\end{lemma}
\begin{proof}
	We recall that $\Phi:\overline\Sigma_{\rm ann}\rightarrow \R^3$ is an embedding. The fact that  $\pi(\immclosedann)$ is a subset of $\overline\Omega$ and contains $\partial_1^D\Om\cup\partial_2^D\Om$ follows from the fact that the interior of $\immclosedann$ is contained in the convex hull of $\Gamma$.
	So it remains to show that $\pi(\immclosedann)$ is simply connected.\\
	Suppose by contradiction that  $\pi(\immclosedann)$ is not simply connected.
Let $H$  be a hole of it, namely a region in $\Om$ surrounded by a loop contained in $\pi(\immclosedann)$ and such that $H\cap 
\pi(\immclosedann)=\emptyset$; choose a point $P\in H$.  
	We will search for a contradiction by exploiting that  $\openannulus$ is an annulus and using that the map $\Phi$ is analytic and harmonic.
	
	Let $\theta$ be the angular coordinate of a cylindrical coordinate system $(\rho,\theta,z)$ in $\R^3$ centred at $P$ and with  $z$-axis the vertical line $\pi^{-1}(P)$. For $\theta \in [0,2\pi)$ we consider the half-plane orthogonal to $\R^2\times\{0\}$ defined by $$\Pi_\theta:=\{(\rho,\theta,z)\colon\rho>0,z\in\R\}.$$ 
		\begin{figure}
		\begin{center}
			\def\svgwidth{0.40\textwidth}
			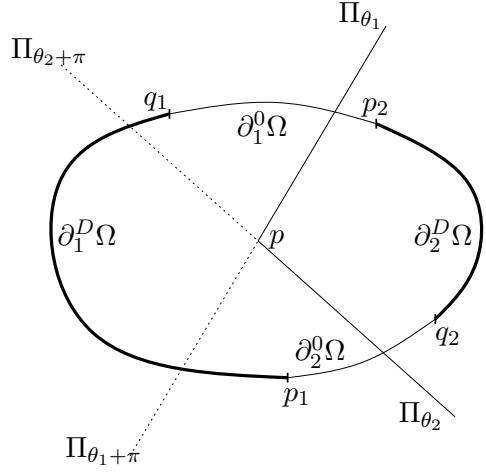
			\caption{The horizontal section of  
two planes $\Pi_{\theta_1}$ and $\Pi_{\theta_2}$ intersecting $\partial^0_1\Om$ and $\partial^0_2\Om$, respectively.
			}\label{fig-costruzione}
		\end{center}
	\end{figure}
 Now we fix two values $\theta_1$ and $\theta_2$ so that $\Pi_{\theta_1}$ and $\Pi_{\theta_2}$ intersect (the interior of) $\partial_1^0\Om$ and $\partial_2^0\Om$ respectively. The half-planes\footnote{The angles are considered $(\textrm{mod}\; 2\pi)$.} $\Pi_{\theta_1+\pi}$ and $\Pi_{\theta_2+\pi}$  might intersect $\partial^D\Om$ (see Figure \ref{fig-costruzione}). However, since the points $p_1$, $q_1$, $p_2$, $q_2$, are in clockwise order on $\partial \Om$, and $\Omega$ is convex, it is not difficult to conclude the following assertion:
	  \medskip
	
The half-planes $\Pi_{\theta_1+\pi}$ and $\Pi_{\theta_2+\pi}$ cannot intersect the two components $\partial_1^D\Om$ and $\partial_2^D\Om$ of $\partial^D\Om$ at the same time. 
	\medskip
	
In other words: If, for instance, $\Pi_{\theta_1+\pi}$ intersects $\partial_1^D\Om$, then $\Pi_{\theta_2+\pi}$ does not intersect $\partial_2^D\Om$. Let us prove the assertion in the form of the last statement, being the other 
cases similar. This is trivial, since, if $\Pi_{\theta_1}$ intersects $\partial_1^0\Om$ and $\Pi_{\theta_1+\pi}$ intersects $\partial_1^D\Om$ (as in Figure \ref{fig-costruzione}), we have that $\Pi_{\widehat\theta}$ intersects $\partial_1^D\Om\cup\partial_1^0\Om$ for all $\widehat\theta\in [\theta_1,\theta_1+\pi]$. As either $\theta_2$ or $\theta_2+\pi$ belongs to $[\theta_1,\theta_1+\pi]$, we have that $\Pi_{\theta_2}\cup\Pi_{\theta_2+\pi}$ intersects $\partial_1^D\Om\cup\partial_1^0\Om$. Since by hypothesis $\Pi_{\theta_2}$ intersects $\partial_2^0\Om$, it follows that $\Pi_{\theta_2+\pi}$ does not intersect $\partial_2^D\Om$, and the statement follows.

Moreover, since $\Pi_{\theta_1}$ intersects $\partial^0_1\Om$ and $\Pi_{\theta_2}$ intersects $\partial^0_2\Om$, it is straightforward that:
	\medskip
	
	If $\Pi_{\theta_1+\pi}$ intersects $\partial_1^0\Om$ then also  $\Pi_{\theta_2+\pi}$ intersects $\partial_1^0\Om$. 
	\medskip

	We are now ready to conclude the proof of the lemma. We have to discuss the  following cases:
	\begin{itemize}
		\item[(1)] $\Pi_{\theta_1+\pi}$ intersects $\partial^0\Om$;
		\item[(2)] $\Pi_{\theta_1+\pi}$ intersects $\partial_1^D\Om$;
		\item[(3)] $\Pi_{\theta_1+\pi}$ intersects $\partial_2^D\Om$.
	\end{itemize} 
	By hypothesis on $P$, for all $\theta\in[0,2\pi)$ the intersection between $\immclosedann$ and $\Pi_\theta$ consists of a 
family of smooth simple curves, either closed or with endpoints on $\Gamma$. Correspondingly, $\Phi^{-1}(\immclosedann\cap \Pi_\theta)$ is a family of closed curves in 
$\overline\Sigma_{\rm ann}$, possibly with endpoints on $C_1\cup C_2$.\\
	 In particular, since $\Pi_{\theta_1} \cap \partial^0_1\Om \neq \emptyset$, the set\footnote{Since $\Pi_{\theta_1} \cap \partial^D\Om = \emptyset$ these curves must be closed in $\openannulus$.} $\Phi^{-1}
(\immclosedann\cap \Pi_{\theta_1})$ is a family of closed curves in $\openannulus$.
	 
	In case (1) also $\Phi^{-1}(\immclosedann\cap \Pi_{\theta_1+\pi})$ consists of closed curves in $\openannulus$.
 Take two loops $\alpha$ and $\alpha'$ in $ \Phi^{-1}(
\immclosedann\cap \Pi_{\theta_1})$ and in $ \Phi^{-1}(
\immclosedann\cap \Pi_{\theta_1+\pi})$ respectively.
  Let $d_1$ be the signed distance function from the plane $\overline\Pi_{\theta_1}\cup\Pi_{\theta_1+\pi}$, positive on $\partial_2^D\Om$. Since $d_1\circ\Phi$ changes its sign when one crosses 
transversally $\alpha$ and $\alpha'$, we easily see that both $\alpha$ and $\alpha'$ cannot be homotopically trivial  in $\openannulus$ (by harmoniticy of $d_1\circ\Phi$, if for instance $\alpha $ is homotopically trivial in $\openannulus$, $d_1\circ\Phi=0$ in the region enclosed by $\alpha$, i.e. the image of $\Phi$ is locally flat, contradicting the analyticity of $\Phi$). Hence, since $\Phi$ is an embedding, they run exactly one time around $C_1$; as a consequence, they must be homotopically equivalent to each other in $\openannulus$. On the other hand, they do not intersect each other 
($\Phi$ is an embedding), so they bound an annulus-type region in $\openannulus$, and by harmonicity $d_1\circ\Phi$ is constantly null in this region. 
This would imply again that the image by $\Phi$ of this annulus is contained in $\overline \Pi_{\theta_1}\cup\Pi_{\theta_1+\pi}$, a contradiction.

In case (2), by recalling our assertion, we 
deduce that $\Pi_{\theta_2+\pi}$ might intersect either $\partial^0\Om$ or $\partial^D_1\Om$.
Further  we can exclude that $\Pi_{\theta_2+\pi}$ intersects $\partial^0\Om$ (otherwise, we repeat the argument for case (1) switching the role of $\theta_1$ and $\theta_2$). 
Therefore  the only remaining possibility is that $\Pi_{\theta_2+\pi}$ intersects $\partial_1^D\Om$ (see Figure \ref{fig-costruzione}).
Let $d_2$ be the signed distance function from $\Pi_{\theta_2}\cup\Pi_{\theta_2+\pi}$ positive on $\partial_2^D\Om$. In particular, $d_i\circ\Phi$, $i=1,2$, is positive on the  circle $C_2$ of 
$\overline\Sigma_{\rm ann}$. By hypothesis on $d_i$, $i=1,2$, we see that $d_1$ is positive on $\Pi_{\theta_2}$, and $d_2$ is positive on $\Pi_{\theta_1}$. 

As in case (1), let $\alpha\in \Phi^{-1}(\immclosedann\cap \Pi_{\theta_1})$ and
$\beta\in \Phi^{-1}(\immclosedann\cap \Pi_{\theta_2})$ 
 be two loops. We know that $\alpha$ and $\beta$ are closed in $\openannulus$.
   Again, we conclude that $\alpha$ and $\beta$ are homotopically equivalent in $\openannulus$, and both run one time around $C_2$. Assume without loss of generality that $\alpha$ encloses $\beta$, which in turn encloses $C_2$. Since $d_2\circ\Phi$ is positive on both $\alpha$ and $C_2$, $d_2\circ\Phi$ must be positive in the region enclosed between them, contradicting the fact that it 
vanishes on $\beta$.

If instead we are in case (3) we can argue analogously to case (2) and get a contradiction. In all cases (1), (2), and (3), we reach a contradiction which derives by assuming that $\pi(\immclosedann)$ is not simply connected. The proof is achieved.
\end{proof}

We next  proceed to characterize the geometry of $\Om\cap\partial 
\pi(\immclosedann)$.

\begin{lemma}\label{lemma_step4}
	Suppose  $m_2(\Gamma)<m_1(\Gamma_1)+m_1(\Gamma_2)$ and let
	$\Phi\in\mathcal{P}_2(\Gamma)$ be a $\mathcal{MY}$ solution to \eqref{catenoid_plateau}.
	Then $\Om\cap\partial \pi(\immclosedann)$
	consists of two disjoint Lipschitz
embedded curves $\beta_1$ and $\beta_2$ joining  $q_1$ to $p_2$, and $q_2$ to $p_1$, respectively. 
Moreover, the closed regions $E_i$ 
enclosed between $\partial^0_1\Om$ and $\beta_i$
are convex for $i=1,2$.
	
\end{lemma} 

\begin{proof}
	By Lemma \ref{lemma_step1}, $\pi(\immclosedann)$ is 
simply connected in $\overline\Omega$, and 
contains $\partial^D\Om$. Therefore $\overline\Omega\setminus  \pi(\immclosedann)$ consists of two simply connected components, one containing $\partial_1^0\Omega$ and the other containing $\partial_2^0\Om$. 
	Let $E_1$ and $E_2$ be the closures of these two components\footnote{The sets $E_1$ and $E_2$ have nonempty interior, since $\Phi(\openannulus)$ is contained in the interior of the convex hull of $\Phi(\partial\openannulus)$, hence  contained in the cylinder $\Omega\times\R$.},
	so that in particular the boundary of $E_i$ is a 
simple Jordan curve 
of the form $ \beta_i\cup\partial_i^0\Om$ for some 
embedded curve $\beta_i\subset\overline\Om$ joining the endpoints of $\partial_i^0\Om$.
	We will prove that $E_i$ is convex for $i=1,2$.
	This will also imply that $\beta_i$ are Lipschitz.
	
	Take $i=1$, and assume by contradiction that $E_1$ is not convex. 
Thus we can find a line $l$ in $\R^2$ and three different points $A_1$, $A_2$, $A_3$ on $l$, with $A_2\in\overline{A_1A_3}$, so that $A_2$ is contained in $\Om\setminus E_1$, and $A_1$ and $A_3$ belong to the interior of $E_1$.

	Consider the region $\pi(\immclosedann)\setminus l$, which consists in several (open) connected components. There is one of 
these  connected components, say $U$,
 which does not intersect $\partial^D\Om$ and 
whose boundary contains $A_2$. In addition, 
$\overline U \cap \partial^D\Om = \emptyset$. 
	Indeed, $\partial U$ is the union of a segment $L$ (containing $A_2$) and a curve $\gamma$ (contained in $\beta_1\subseteq\partial (\pi(\immclosedann)$) joining its endpoints.
	Hence, $\overline U\setminus U=\gamma\cup L$, and $L$ cannot intersect $\partial^D\Om$ by the hypothesis on $A_1$, $A_2$, and $A_3$.
	
	Let $\Pi_l\subset\R^3$ be the 
plane containing $l$ and 
orthogonal to the plane containing $\Om$; 
As usual, $\Pi_l\cap 
\immclosedann$ is a family of closed curves, possibly with endpoints on $\Gamma\cap \Pi_l$.
	Now, pick a point $P$ on $\partial U\setminus L$, and let $Q$ be a point on $\immclosedann$ so that $\pi(Q)=P$.
Let $d_l:\R^3\rightarrow \R$ be the signed distance from $\Pi_l$, 	
with $d_l(Q)=d_l(P)>0$. We claim that, if $D$ is the connected component of $\{w\in
\overline\Sigma_{\rm ann}:d_l\circ\Phi(w)>0\}$ containing  the point $\Phi^{-1}(Q)$, then $D\cap \partial\openannulus=\emptyset$.
	This would contradict the harmonicity of $d_l\circ\Phi$, since $d_l\circ\Phi$ would be zero 
on $D$, but $d_l(Q)>0$. 
	
	Assume by contradiction that the converse holds. Then there is  an arc $\alpha:[0,1]\rightarrow D\cup\partial\openannulus$  joining $\Phi^{-1}(Q)$ to $\partial\openannulus$ . 
The image of the 
map $\pi\circ\Phi\circ \alpha$ is an arc in $\overline \Omega$ 
joining $P$ to $\partial^D\Om$ and such that $d_l\geq0$ on it. 
Clearly this arc is a subset of $\pi(\immclosedann)$. Since $\pi\circ\Phi\circ \alpha(0)=P$, it follows that the image of $\pi\circ\Phi\circ \alpha$ is contained in $\overline U$. Now $\overline U$ does not intersect $\partial^D\Om$, contradicting that $\pi\circ\Phi\circ \alpha(1)\in \partial^D\Om$. This concludes the proof.
\end{proof}

In the next step we show that there exists a set $E\subset\R^3$ of finite perimeter such that 
$$\partial E=\partial^*E=\Phi(\openannulus)\cup\overline\Delta_1\cup\overline\Delta_2,$$ 
where
\begin{equation}\label{delta}
\Delta_i:=\{P=(P',P_3)\in \R^3:P'=(P_1,P_2)\in \partial_i^D\Om, \;P_3\in (-\varphi(P'),\varphi(P'))\},
\qquad 
i=1,2.
\end{equation}
 In particular $\overline\Delta_1\cup\overline\Delta_2\subset(\partial\Om)\times\R$ and
$(\Om\times\R)\cap\partial E=\Phi(\openannulus)$.\\

We first fix some notation. We let $\jump{E}\in\mathcal{D}_3(\R^3)$ be the 3-current given by integration over  $E$ with $E\subset\R^3$ being a set of finite perimeter. 
To every $\mathcal{MY}$ solution 
$\Phi\in\mathcal{P}_2(\Gamma)$  
to \eqref{catenoid_plateau} we associate the push-forward 2-current  $\Phi_\sharp\jump{\openannulus}\in\mathcal D_2(\R^3)$ given by integration over the (suitably oriented) surface $\Phi(\openannulus)$ 
\cite[Section 7.4.2]{Krantz-Parks}.
Finally if $\mathcal{T}\in\mathcal{D}_k(U)$ with $U\subset\R^3$ open and $k=2,3$, we denote by $|\mathcal{T}|$ the mass of $\mathcal{T}$ in $U$ [see \cite[p. 358]{Federer}].

\begin{lemma}[{\bf Region enclosed by $\Phi(\openannulus)$}]\label{lemma_step2}
	Suppose $m_2(\Gamma)<m_1(\Gamma_1)+m_1(\Gamma_2)$ and let $\Phi\in\mathcal{P}_2(\Gamma)$ be a $\mathcal{MY}$ solution to \eqref{catenoid_plateau}.
	Then there is a closed
 finite perimeter set $E\subset\overline\Om\times\R$ such that $\partial E=\Phi(\openannulus)$ in $\Om\times\R$.
\end{lemma}

\begin{proof}
As $\Phi_\sharp\jump{\openannulus}$ is a boundaryless integral $2$-current in $\Omega\times\R$, there exists (see, e.g., \cite[Theorem 7.9.1]{Krantz-Parks}) an integral $3$-current $\mathcal E\in \mathcal D_3(\Omega\times\R)$ with $\partial\mathcal E=\Phi_\sharp\jump{\openannulus}$, and we might also assume that the support of $\mathcal E$ is compact in $\overline\Omega\times\R$. We claim that, up to switching the orientation of $\Phi_\sharp\jump{\openannulus}$, $\mathcal E$ has multiplicity in $\{0,1\}$, and hence is the integration $\jump{E}$ over a bounded measurable set $E$. 
This is a finite perimeter set if we show that the integration over $(\Om\times\R)\cap\partial^*E$ coincides with $\Phi_\sharp\jump{\openannulus}$.

By Federer decomposition theorem \cite[Section 4.2.25, p. 420]{Federer} (see also \cite[Section 4.5.9]{Federer} and \cite[Theorem 7.5.5]{Krantz-Parks}) there is a sequence
$(E_k)_{k\in \mathbb N}$ of finite perimeter sets such that
\begin{equation}\label{6.9a}
	\mathcal E=\sum_{k=1}^{+\infty}\sigma_k\jump{E_k},\qquad \sigma_k\in \{-1,1\},
\end{equation}
 moreover 
\begin{align}\label{6.9}
|\mathcal E|=\sum_{k=1}^{+\infty}|E_k|\quad\text{and} \quad|\partial\mathcal E|=\mathcal H^2(\Phi(\openannulus))=\sum_{k=1}^{+\infty}\mathcal H^2(\partial^*E_k).
\end{align}
We start by observing that
\begin{equation}\label{inclusione}
\partial^*E_k\subseteq \Phi(\openannulus)\qquad \forall k\in \mathbb N.	
\end{equation}
Indeed, fixing $k\in \mathbb N$, by the second equation in \eqref{6.9}, we have that $\partial^*E_k$ is contained in the support of $\partial\mathcal E$, which in turn is $\Phi(\openannulus)$.
As a consequence, if $P=(P_1,P_2,P_3)\in(\Om\times\R)\cap \overline{\partial^*E_k}$, then $P\in \Phi(\openannulus)$. Around $P$ we can find suitable coordinates and a cube $U=(P_1-\eps,P_1+\eps)\times (P_2-\eps,P_2+\eps)\times(P_3-\eps,P_3+\eps)$ such that $\Phi(\openannulus)\cap U$ is the graph $\mathcal G_h$ of a smooth function $h:(P_1-\eps,P_1+\eps)\times (P_2-\eps,P_2+\eps)\rightarrow (P_3-\eps,P_3+\eps)$. Moreover, $\Phi_\sharp\jump{\openannulus}=\jump{\mathcal G_h}$ in $U$. We conclude\footnote{This is a consequence of the constancy lemma and the fact that $\partial\mathcal E-\Phi_\sharp\jump{\openannulus}=0$ in $U$.} that $\mathcal E\res U=\jump{SG_h\cap U}+m\jump{U}$, with $SG_h$ the subgraph of $h$, and $m\in \mathbb Z$.\\
We claim that 
$$\forall k\quad\text{either}\quad E_k\cap U=SG_h\cap U \quad\text{or} \quad E_k\cap U=U\setminus SG_h.$$
Indeed, assume for instance that $|E_k\cap SG_h\cap U|>0$ and $|(SG_h\setminus E_k)\cap U|>0$; by the constancy lemma it follows that $\partial\jump{E_k}$ is nonzero in the simply connected open set $SG_h$, contradicting \eqref{inclusione}. 

As a consequence of the preceding claim, we have that $U\cap \partial^*E_k=U\cap \Phi(\openannulus)$. Since this argument holds for any choice of $P\in (\Omega\times\R)\cap \overline{\partial^*E_k}$, we have proved that $(\Omega\times\R)\cap \overline{\partial^*E_k}$ is relatively open 
(and relatively closed at the same time) in $\Phi(\openannulus)$, which in turn being a connected open set, implies 
$$\immclosedann=\overline{\partial^*E_k}\qquad \forall k\in \mathbb N.$$

Denote by $\mathcal I^{\pm}:=\{k\in \mathbb N:\sigma_k=\pm1\}$, where $\sigma_k$ appears in \eqref{6.9a}.
Going back to the local behaviour around $P\in \Phi(\openannulus)$, if $U$ is a neighbourhood as above, we see that for all $k\in \mathcal I^+$  either $E_k\cap U=SG_h$ or $E_k=U\setminus SG_h$ (namely, all the $E_k$'s coincide in $U$), since otherwise, there will be cancellations in the series $\sum_{k\in \mathcal I^+}\partial \jump{E_k}$, in contradiction with the second formula in \eqref{6.9}. Assume without loss of generality that for all $k\in \mathcal I^+$ we have $E_k\cap U=SG_h$; thus, arguing as before, for all $k\in \mathcal I^-$ we must have $E_k\cap U=U\setminus SG_h$.
 
 We obtain that $\mathcal E\res U=m\jump{SG_h}-n\jump{U\setminus SG_h}$ 
for some nonnegative integers $n,m$. Since  $(\partial \mathcal E)\res U=(m+n)\jump{\mathcal G_h}$ 
and also $(\partial\mathcal E)\res U=\Phi_\sharp\jump{\openannulus}=\jump{\mathcal G_h}$ in $U$, we conclude $m+n=1$. Hence either $m=1$ 
and $n=0$, or $m=0$ and $n=1$. On the other hand, we know that 
 $\mathcal E\res U=\sum_{k\in \mathcal I^+}\jump{E_k\cap U}-\sum_{k\in \mathcal I^-}\jump{E_k\cap U}$, from which it follows that $\mathcal I^+$ has cardinality $m$ and $\mathcal I^-$ has cardinality $n$. 
 Namely, one of the sets $\mathcal I^\pm$ is empty, and the other contains only one index.
 
We conclude that the sum in \eqref{6.9a} involves only one index, that is, there is only one compact set $E$ in $\overline\Omega\times\R$ such that (up to switching the orientation) $$\mathcal E=\jump{E}.$$
 This concludes the proof.
\end{proof}

\begin{remark}\label{rem:steiner-sym}
From the fact that $ (\overline \Omega\times\R)\cap \partial E=\immclosedann\cup\overline\Delta_1\cup\overline\Delta_2$,  we 
easily see that $\pi( E)=\pi(\immclosedann)$ which, by Lemma \ref{lemma_step1}, is simply connected.
\end{remark}

We denote by ${\rm sym}_{\rm st}(E)$ the set (symmetric with respect to the horizontal plane $\R^2\times\{0\}$) obtained applying to $E$ the Steiner symmetrization  with respect to  $\R^2\times\{0\}$.  \\
Clearly ${\rm sym}_{\rm st}(E)\cap(\partial_i^D\Om\times\R)=\overline{\Delta_i}$ with $\Delta_i$ defined as in \eqref{delta}.
We define the surfaces  
\begin{equation}\label{surf}
	S:=	\partial({\rm sym}_{\rm st} (E))\setminus(\Delta_1\cup\Delta_2),\quad
		S^+:=	S\cap\{x_3\ge0\},\quad	S^-:=	S\cap\{x_3\le0\}.
\end{equation}
Since $P({\rm sym}_{\rm st}(E))\le P(E)$ (here $P(\cdot)$ is the perimeter \cite{AFP}) we have $\mathcal{H}^2(S)\le\mathcal{H}^2(
\immclosedann)$.

\begin{lemma}[{\bf Graphicality 
of $\partial({\rm sym}_{\rm st} (E))$} ]\label{lemma_step3} 
		Suppose $m_2(\Gamma)<m_1(\Gamma_1)+m_1(\Gamma_2)$ and let $\Phi\in\mathcal{P}_2(\Gamma)$ be a $\mathcal{MY}$ solution to \eqref{catenoid_plateau}.
Let $E$ be the finite perimeter set given by Lemma \ref{lemma_step2} and $S^\pm$ be as in \eqref{surf}.
Then there is $\widetilde \psi\in BV({\rm int}(\pi(E))\cap C^0(\pi(E))$ such that 
$
S^\pm=
	\mathcal {G}_{\pm\widetilde\psi}$.
In particular $S^\pm\cap(\R^2\times\{0\})=\overline{\Om\cap\partial(\pi(E))}$.
\end{lemma}

The proof of Lemma \ref{lemma_step3} essentially follows from the fact that 
$\immclosedann$ is a  minimal surface in $\overline\Omega\times \R$. 

\begin{proof}
Since $E$ has finite perimeter,
there exists a function $\widetilde \psi\in BV(\pi(E))$  such that 
$
	S^\pm=\mathcal G_{\pm\widetilde \psi}$ \cite{CCF}. So,
we only need to show that 
$\widetilde\psi$ is continuous.
Take a point $\pointinOm$ in the interior of $\pi(E)$; if $\pointinOm
=\pi(\Phi(w))$ for some $w$, then $w\in \openannulus$, since 
$\pi(\Phi(C_i)) \subset \partial \Om$ for $i=1,2$.
If at none of the points of $\pi^{-1}(\pointinOm)\cap 
\immclosedann$ the tangent plane to $\immclosedann$ is vertical, 
then $\widetilde\psi$ is $C^\infty$ in a neighbourhood of $\pointinOm$, 
since it is the linear combination of smooth functions 
(see the discussion after formula \eqref{6.18} below, where details are given).
Therefore we only have to check continuity of $\widetilde\psi$ 
at those points $\pointinOm$ for which there is $\pointinspace\in 
\pi^{-1}(\pointinOm)\cap\immclosedann$ such that
$\immclosedann$ has a vertical tangent plane $\Pi$ at $\pointinspace$. 
 
 Consider a system of Cartesian coordinates  
centred at $ \pointinspace$, with the $(x,y)$-plane 
coinciding with $\Pi$, the $x$-axis coinciding with the line $\pi^{-1}( \pointinOm)$, 
and let $z=z(x,y)$ (defined at least in a neighbourhood of $0$) be the analytic function whose graph coincides with  $\immclosedann$. 
This map, restricted to the $x$-axis, is analytic and
 it vanishes at $x=0$; hence it is 
either  constantly zero or it 
has a discrete set of zeroes (in the neighbourhood where it exists).
We now exclude the former case: If $z(\cdot,0)$ is constantly zero, 
it means that around $\pointinspace$ there is a vertical open
segment included in $\pi^{-1}(\pointinOm)$, 
which is 
contained in $\immclosedann$. Let $Q$ be an extremal point of this segment, and 
let $\Pi_Q$ be the tangent plane to $\immclosedann$ at $Q$. 
This plane must 
contain as tangent vector the above segment, hence $\Pi_Q$ is vertical and 
contains $\pi^{-1}(\pointinOm)$. 
Choosing again a suitable Cartesian coordinate system  centred at $Q$ we can express locally 
the surface $\immclosedann$ as the graph of an analytic function defined in a 
neighbourhood of $Q$ in $\Pi_Q$, and so 
the restriction of this map to 
$\pi^{-1}(\pointinOm)$ is analytic in a neighbourhood of $Q$,  
hence it must be constantly zero since it is zero in a 
left (or right) neighbourhood of $Q$. 
 What we found is that we can properly extend the segment $\overline{\pointinspace Q}$ 
on the $Q$ side to a segment $\overline{PR}$ contained in 
$\immclosedann$. 
By iterating  this argument we 
conclude that the whole line $\pi^{-1}(\pointinOm)$ 
is contained in $\immclosedann$, which
is impossible since ${\Phi}(\overline\Sigma_{\rm ann})$ is bounded.

Hence  the zeroes of the function $z(\cdot,0)$ are isolated, so the 
next assertion follows:
\medskip

\textit{Assertion A: Let $\pointinspace\in \pi^{-1}(\pointinOm)\cap
\immclosedann$. Then in a neighbourhood of $\pointinspace$ 
the only intersection between $\immclosedann$ 
and $\pi^{-1}(\pointinOm)$ is $\pointinspace$ itself.} 
\medskip

We can now conclude the proof of 
the continuity of the function $\widetilde\psi$. Let $\pointinOm$ be in the 
interior of $\pi(E)$, and write $\pi^{-1}(\pointinOm)\cap 
\immclosedann=\{Q_1,Q_2,\dots,Q_{m}\}\subset\Omega\times\R$. 
It follows that 
\begin{align}\label{verticalslice}
2\widetilde\psi(\pointinOm)=\mathcal H^1(\pi^{-1}(\pointinOm)\cap E)=\sum_{j=1}^m\sigma_j (Q_j)_3,
\end{align}
where $(Q_j)_3$ is the vertical coordinate of $Q_j$ and $\sigma_j\in\{-1,0,1\}$ is defined as
\begin{align}
\sigma_j=\begin{cases}
-1&\text{if }\overline{Q_{j-1}Q_j}\subset\overline{ \R^3\setminus E}\text{ and }\overline{Q_{j}Q_{j+1}}\subset  E,\\
1 &\text{if }\overline{Q_{j-1}Q_j}\subset E\text{ and }\overline{Q_{j}Q_{j+1}}\subset \overline{\R^3\setminus E},\\
0&\text{otherwise,}
\end{cases}
\qquad \qquad j=1,\dots,m.
\end{align}
Let $\pointinOm_k 
\in \textrm{int}(\pi(E))$ be such that the sequence
$(\pointinOm_k )$ converges to $\pointinOm$, and write $\pi^{-1}
(\pointinOm_k )\cap \immclosedann=\{Q^k_1,Q^k_2,\dots,Q^k_{m_k}\}\subset\Omega\times\R$. With a similar 
notation as above, we have
\begin{equation}\label{verticalslice2}
2\widetilde\psi(\pointinOm_k)=\mathcal H^1(\pi^{-1}(\pointinOm_k)
\cap E)=\sum_{j=1}^{m_k}\sigma^k_j (Q^k_j)_3.
\end{equation}
Now, if $\pointinOm$ is such that at
 every point $Q_j$ the tangent plane to $\immclosedann$ 
is not vertical, then  $\immclosedann$ is a smooth Cartesian surface in a neighbourhood of 
$Q_j$, and so it is clear that, for $k$ large enough, 
\begin{align}\label{6.18}
m=m_k,\qquad 
Q_j^k\rightarrow Q_j,\qquad \sigma^k_j\rightarrow\sigma_j\qquad\text{for all }j=1,\dots,m,
\end{align}
and the continuity of \eqref{verticalslice} follows.
Therefore it remains to check continuity in the case that the tangent plane to some $Q_j$ is vertical.

Let $\widetilde Q$ 
be one of these points, with associated sign $\widetilde\sigma$. 
By assertion A there is $\delta>0$ 
so that $\widetilde Q$ is the unique intersection between $\pi^{-1}(\pointinOm)$ and 
$\immclosedann$ with vertical coordinate in $[\widetilde Q_3-\delta,\widetilde Q_3+\delta]$. 
This means that the segments  $\pi^{-1}(\pointinOm)\cap\{\widetilde Q_3-\delta<x_3<\widetilde Q_3\}$ 
and $\pi^{-1}(\pointinOm)\cap\{\widetilde Q_3<x_3<\widetilde Q_3+\delta\}$ are either
subsets of $\textrm{int}(E)$ or subsets of 
$\R^3\setminus E$. In particular,
 there is a neighbourhood $U\subset\Omega$ of $\pointinOm$ such that the discs $U\times \{x_3=\widetilde Q_3-\delta\}$ and $U\times \{x_3=\widetilde Q_3+\delta\}$ are 
subsets of  $\textrm{int}(E)$ or of $\R^3\setminus E$. Suppose 
without loss of generality that both these discs are inside $\R^3\setminus E$ (the other cases being
 similar),  
so that $\widetilde \sigma=0$. 
We infer that, for $k$ large enough so that $\pointinOm_k\in U$, 
there is a finite subfamily $\{Q^k_j:j\in J\}$ of $\{Q^k_1,Q^k_2,\dots,Q^k_{m_k}\}$ contained in $\{ \widetilde Q_3<x_3<\widetilde Q_3+\delta\}$ and which satisfies the following: The sum in \eqref{verticalslice2} restricted to such subfamily reads as:
$$ \sum_{j\in J}\sigma^k_j (Q^k_j)_3=(Q^k_{j_l})_3-(Q^k_{j_{l-1}})_3+\dots+(Q^k_{j_2})_3-(Q^k_{j_1})_3,$$
where 
$J=\{j_1,j_2,\dots,j_l\}$ and $(Q^k_{j_l})_3>(Q^k_{j_{l-1}})_3>\dots>(Q^k_{j_2})_3>(Q^k_{j_1})_3$ (in the case that $j_l=1$ necessarily $\sigma_{j_1}^k=0$ and the sum is zero). We have to show that this sum tends to $\widetilde\sigma \widetilde Q_3=0$ as $k\rightarrow
+\infty$, which is true, since each $Q^k_j$ tends to $\widetilde Q$. 
Repeating this argument for each point $\widetilde Q$ appearing in \eqref{verticalslice} 
with a vertical tangent plane to $\immclosedann$, we conclude the proof of continuity of $\widetilde\psi$ in the interior of $\pi(E)$.\\
\smallskip

Let now $\pointinOm\in\partial(\pi(E))$. If $\pointinOm\in \partial(\pi(E))\cap \Omega$ 
then every point in $\pi^{-1}(\pointinOm)\cap \immclosedann$ has vertical tangent plane 
and we can argue as in the previous case.
It remains to show continuity of $\widetilde \psi$ 
on $\partial\pi(E)\cap \partial\Omega$. In this case we 
exploit the fact that the interior of 
$\immclosedann$ is contained in $\Om\times \R$. We sketch  the proof without details since it is very similar to the previous arguments. Let 
$\pointinOm\in \partial_1^D\Om$, 
thus $\pi^{-1}(\pointinOm)\cap\Gamma_1$ consists of two points 
$Q_1$ and $Q_2$. Let $(\pointinOm_k)$ 
be a sequence of points in $\overline{\pi(E)}$ converging to $P$. 
For $\pointinOm_k \in \partial^D_1\Om$ it follows $\pi^{-1}(\pointinOm_k)
\cap\Gamma_1=\{Q^k_1,Q^k_2\}$ and the continuity of $\widetilde\psi$ 
follows from the continuity of $\varphi$ on $\partial^D_1\Om$, 
whereas if $\pointinOm_k$ 
is in the interior of $\pi(E)$ there holds $\pi^{-1}(\pointinOm_k)
\cap\Gamma_1=\{Q^k_1,Q^k_2,\dots, Q^k_{m_k}\}$. 
Using the continuity of $\Phi$ up to $C_1$, it is easily seen that all such points must converge, 
as $k\rightarrow +\infty$, either to $Q_1$ or to $Q_2$. 
Hence we can repeat an argument similar to the one used before.
\end{proof}  

\begin{lemma}\label{lem:construction-of-parametrization}
Suppose $m_2(\Gamma)<m_1(\Gamma_1)+m_1(\Gamma_2)$ and let $\Phi\in\mathcal{P}_2(\Gamma)$ be a $\mathcal{MY}$ solution to \eqref{catenoid_plateau}. Let $E$  be the finite perimeter set given in Lemma \ref{lemma_step2} and let
		$S$ be defined as in \eqref{surf}.
		Then there is an injective map $\widetilde\Phi\in H^1
(\openannulus;\R^3)\cap C^0(\overline \Sigma_{\rm ann};\R^3)$ 
which maps $\partial \openannulus$ 
weakly monotonically to $\Gamma$ and  such that $\widetilde\Phi (\overline \Sigma_{\rm ann}) =S$, and also
		 \begin{align}\label{eq:area-symmetrization}
		\mathcal H^2(S)=	\int_{\Sigma_{{\rm ann}}}|\partial_{w_1}\widetilde \Phi\wedge \partial_{w_2}\widetilde \Phi |dw=
			\int_{\Sigma_{{\rm ann}}}|\partial_{w_1}\Phi\wedge\partial_{w_2}\Phi|dw=m_2(\Gamma).
		\end{align}
In particular, $\widetilde \Phi$ is a solution of \eqref{catenoid_plateau}.
\end{lemma}

\begin{proof}
	By Lemma \ref{lemma_step3} there is $\widetilde \psi\in BV({\rm int}(\pi(E))\cap C^0(\pi(E))$ such that $S^\pm=
		\mathcal {G}_{\pm\widetilde\psi}$.
As a consequence,  for  $p\in \partial^D\Om$
	we have $\widetilde\psi(p)=\varphi(p)$ and for $p\in \partial (\pi(E))\cap \Omega$ 
we have $\widetilde\psi(p)=0$. \\
	By Lemma \ref{lemma_step1} $\pi(E)$ is simply connected, and so 
the maps $\widetilde \Psi^\pm:\pi(E)\rightarrow \R^3$ given by
	$\widetilde \Psi^\pm(p):=(p,\pm\widetilde\psi(p))$ are disc-type parametrizations of  $S^\pm$.
	Moreover $S^+$ and $S^-$ glue to each other along  
$\partial({\rm sym}_{\rm st}(E))\cap(\R^2\times\{0\})=
\beta_1\cup\beta_2$, where $\beta_1$ and $\beta_2$ are the curves given by Lemma \ref{lemma_step4} .

	Let $(\sigma,\psi)\in\comp$ be a minimizer of $\mathcal F$ which satisfies 
properties \ref{1.}-\ref{5.} of Theorem \ref{thm:regularity}. Setting $\widetilde\sigma:=(\beta_1,\beta_2)$ and by extending $\widetilde\psi$ to zero in $\overline\Om\setminus \pi(E)$, without relabelling it, by minimality we get $$2\mathcal F(\sigma,\psi)\leq 2\mathcal F(\widetilde\sigma,\widetilde\psi)=\mathcal H^2(S),$$ whence, by Remark \ref{rem:steiner-sym}
	 \begin{align}\label{last_eq}
	2\mathcal F(\sigma,\psi)\leq \mathcal{H}^2(S)\le\mathcal{H}^2(\immclosedann)
	=\int_{\Sigma_{{\rm ann}}}|\partial_{w_1}\Phi\wedge\partial_{w_2}\Phi|dw=m_2(\Gamma).
	\end{align}
We are in the hypotheses of Lemma \ref{lem_tec2}, and therefore there exists a map $\widetilde \Phi\in P_2(\Gamma)$ parametrizing $\mathcal G_{\psi\res(\overline{\Om\setminus E(\sigma)})}\cup \mathcal G_{-\psi\res(\overline{\Om\setminus E(\sigma)})}$ which is a minimizer 
of \eqref{catenoid_plateau}. In particular, $2\mathcal F(\sigma,\psi)=m_2(\Gamma)$, and all the inequalities in \eqref{last_eq} are equalities. We 
deduce also that 
$(\widetilde \sigma,\widetilde\psi)$ is a minimizer of $\mathcal F$ in $\comp$, so that by Theorem \ref{thm:regularity} $\widetilde \psi$ is analytic in $\textrm{int}(\pi(E))$. As a consequence it belongs to $W^{1,1}(\pi(E);\R^3)$.

We now conclude the proof of the lemma by invoking again Lemma \ref{lem_tec2}.
\end{proof}

\begin{lemma}\label{lemma_step5}
Suppose $m_2(\Gamma)<m_1(\Gamma_1)+m_1(\Gamma_2)$ and let $\Phi\in\mathcal{P}_2(\Gamma)$ be a $\mathcal{MY}$ solution to \eqref{catenoid_plateau}. Let $E$ be the finite perimeter set given in Lemma \ref{lemma_step2} and let $S$ be defined as in \eqref{surf}. Then 
$\immclosedann=S$
 and in particular $E={\rm sym}_{\rm st}(E)$.
\end{lemma}
\begin{proof} By Lemma \ref{lem:construction-of-parametrization} we have that $\mathcal H^2(S)=m_2(\Gamma)$ from which it follows that  $P(\textrm{sym}_{\rm st}(E))= P(E)$. Then  we can apply \cite[Theorem 1.1]{CCF}
to deduce the existence of two functions
$f,g:\pi(E)\rightarrow \R$ of bounded variation, such that $\partial^* E=\mathcal G_f\cup\mathcal G_{g}$ (up to $\mathcal H^2$-negligible sets). We will show that $f=\widetilde\psi$ and $g=-\widetilde\psi$. To this aim, 
thanks again to \cite[Theorem 1.1]{CCF}, we know that for a.e. $p\in \pi(E)$, the two unit normal vectors $\nu^f=(\nu^f_1,\nu^f_2,\nu^f_3)$ and $\nu_g=(\nu^g_1,\nu^g_2,\nu^g_3)$ to $\mathcal G_f$ and $\mathcal G_g$ at the points $(p,f(p))$ and $(p,g(p))$, respectively, satisfy
	\begin{align}\label{normals}
		(\nu^f_1,\nu^f_2,\nu^f_3)=(\nu^g_1,\nu^g_2,-\nu^g_3).
	\end{align} 
	To conclude the proof it is then sufficient to show that $f=-g$ a.e. on $\pi(E)$: indeed this would readily imply $E={\rm sym}_{\rm st}(E)$ and hence $f=\widetilde\psi$.
	
	Let $p\in \textrm{int}(\pi(E))$; if 
	\begin{equation}
	\pi^{-1}(p)\cap S=\{P_1,P_2,\dots,P_k\},
	\end{equation}
	then for a.e. $p\in \textrm{int}(\pi(E))$ it is $k\leq2$. We now show that, for all $p\in \textrm{int}(\pi(E))$, if $k>1$, none of the points  $\{P_1,P_2,\dots,P_k\}$ has vertical tangent plane. 
Assume by contradiction that $P_1$ has vertical tangent plane $\Pi_1$. In this case $\Pi_1\cap S$ consists, in a neighbourhood $U$ of $P_1$, of at least $2$ curves crossing {transversally} at $P_1$. These curves, by assertion A in the proof of Lemma \ref{lemma_step3}, intersect $\pi^{-1}(p)$ only at $P_1$. Moreover, in a neighbourhood $V$ of $P_2$, with $U\cap V=\varnothing$, $\Pi_1\cap S$ consists of (at least) one (or more)
 curve passing through $P_2$. This curve is locally Cartesian if $\pi^{-1}(p)$ crosses $S$ transversally in $P_2$, otherwise it can be locally the union of two curves ending at $P_2$, with vertical tangent plane, which lie on the same side of $\Pi_1$ with respect to $\pi^{-1}(p)$. In both cases, we deduce that there is a point $q\in \Pi_1\cap (\Om\times\{0\})$ for which $\pi^{-1}(q)$ intersects transversally $S$ in at least three points. As a consequence, for all $q'$ in a neighbourhood of $q$ in 
$\Om$, the line $\pi^{-1}(q')$ intersects $S$ at more than two points, which is a contradiction.
	We have proved the following:
	\medskip
	
	\textit{Assertion: for all $p\in\textrm{int}(E)$ the line $\pi^{-1}(p)$ either intersects $S$ transversally at two points $P_1,P_2$, or it intersects $S$ at only one point $P_1$.}  
	\medskip 
	
	We now see that the latter case cannot happen. 
Indeed, first one checks
 that in this case the intersection cannot be transversal\footnote{This is a consequence of the fact that the line $\pi^{-1}(p)$ must lie outside the set $E$, with the only exception of the point $P_1$. Indeed, otherwise, there must be some other point in $\pi^{-1}(p)\cap S$, $E$ being compact in $\R^3$.}, and that $\pi^{-1}(p)$ must be tangent to $S$ at $P_1$. Let $\Pi_1$ be the 
vertical tangent plane to $S$ at $P_1$. 
Let $\Pi_1^\perp$ be the vertical plane orthogonal to $\Pi_1$ 
passing 
through $P_1$. In a neighbourhood of $P_1$, the unique curve in $S\cap \Pi_1^\perp$ 
must be the union of two curves joining at $P_1$, and these curves must belong to
 the same half-plane of $\Pi_1^\perp$ with boundary $\pi^{-1}(p)$.
As a consequence, if $p'\in \Omega\cap \Pi_1^\perp$ is in
that half-plane, then $\pi^{-1}(p')$ consists of at least two points; if $p'$ lies in the opposite half-plane, then $\pi^{-1}(p')$ is empty. This means that necessarily $p\in \partial \pi(E)$. Namely, the previous assertion can be strengthened to:
	\medskip
	
\textit{For all $p\in\textrm{int}(E)$ the line $\pi^{-1}(p)$ intersects $S$ transversally at exactly two points $P_1,P_2$.}  
\medskip 

The consequence of this is that $f$ and $g$ belong to $W^{1,1}(\textrm{int}(\pi(E)))$ and are also smooth in $\textrm{int}(\pi(E))$. 
Indeed, let $p\in \textrm{int}(\pi(E))$, so $f(p)\neq g(p)$, and 
	\begin{equation}\label{twopoints}
		\pi^{-1}(p)\cap S=\{(p,f(p)),(p,g(p))\}.
	\end{equation}
	Since $S$ is locally the graph of smooth functions around $(p,f(p))$ and $(p,g(p))$, 
these functions coincide with $f$ and $g$, respectively.
		We can now conclude the proof 
of the lemma: let us choose a simple curve $\alpha:[0,1]\rightarrow \pi(E)$ with $\alpha(0)\in \partial^D\Om$ and $\alpha(1)=p$ such that \eqref{normals} holds for $\mathcal H^1$ a.e. $p\in \alpha([0,1])$. Since $f\circ \alpha$ and $g\circ \alpha$ 
are differentiable  in $[0,1]$, condition 
\eqref{normals} uniquely determines the tangent planes to $\mathcal G_f$ and $\mathcal G_g$, and hence it implies that the derivatives of $f\circ \alpha$ and $g\circ\alpha$ satisfy
	\begin{align}
	(f\circ \alpha)'(t)+(g\circ \alpha)'(t)=0,\qquad \text{ for a.e. }t\in[0,1].
	\end{align}
	By continuity of $f$ and $g$ one infers $f\circ\alpha+g\circ\alpha=c$ a.e. on $[0,1]$ (actually everywhere since $f+g$ is continuous), for some constant $c\in\R$. To show that $c=0$ it is sufficient to observe that $f\circ\alpha(0)=\varphi(\alpha(0))$ and $g\circ\alpha(0)=-\varphi(\alpha(0))$. Hence $f(p)=-g(p)$, and the thesis of Lemma \ref{lemma_step5} is achieved. 
\end{proof}

We are now in a position to conclude the proof of  Theorem \ref{crucial_teo}.

\begin{proof}[Proof of Theorem \ref{crucial_teo}]
	Property \ref{crucial1} follows by Lemma \ref{lemma_step1} and Lemma \ref{lemma_step4}. Properties \ref{crucial2}--\ref{crucial4} follow by Lemma \ref{lemma_step3} and Lemma \ref{lemma_step5}. To see that $\beta_i$ are 
$C^\infty$ it is sufficient to observe that, in view of
 the Cartesianity of $S^+$ and $S^-$, their union coincides 
with the set $S\cap\{x_3=0\}$ which, by standard arguments, 
 is the image of the zero-set of $\Phi_3$, which is smooth.
\end{proof}
\begin{theorem}\label{teorema-finale}
	
	There holds
	\begin{align}
		2\min_{(s,\zeta)\in \comp}\mathcal F(s,\zeta)= m_2(\Gamma).
	\end{align}
\end{theorem}
\begin{proof}
	\step 1 
$2\min_{(s,\zeta)\in \comp}\mathcal F(s,\zeta)\le m_2(\Gamma)$.

	Suppose $m_2(\Gamma)<m_1(\Gamma_1)+m_1(\Gamma_2)$.
Let $\Phi\in\mathcal{P}_2(\Gamma)$ be a $\mathcal{MY}$ solution to \eqref{catenoid_plateau} and let $S:=\immclosedann$. By Theorem \ref{crucial_teo} the following properties hold:
\begin{itemize}
	\item $S\cap(\R^2\times\{0\})=\beta_1\cup\beta_2$ with $\beta_1$ and $\beta_2$ disjoint
embedded curves of class $C^\infty$ joining $q_1$ to $p_2$ and $q_2$ to $p_1$, respectively;
	\item $S$ is symmetric with respect to $\R^2\times\{0\}$;
	\item for $i=1,2$ the closed 
region $E_i$ enclosed between $\partial^0_i\Om$ and $\beta_i$ is convex;
	\item $S^+=S\cap\{x_3\ge0\}$ is the graph of  $\widetilde\psi\in W^{1,1}(U)\cap C^0(\overline U)$, where $U=\Om\setminus( E_1\cup  E_2)$  is the open 
region enclosed between $\partial^D\Om$ and $\beta_1\cup\beta_2$.
\end{itemize}
Let $(\sigma,\psi)\in\comp$ be given by 
\begin{equation*}
	\sigma:=(\sigma_1,\sigma_2)\quad\text{and}\quad\psi:=\begin{cases}
		0&\text{in } \Om\setminus U,
\\
		\widetilde{\psi}&\text{in } U,
	\end{cases}
\end{equation*}
where $\sigma_i([0,1])=\beta_i$ for $i=1,2$.
Then clearly $S^+=\mathcal{G}_{\psi\res(\Om\setminus E(\sigma))}$ and
\begin{equation*}
\min_{(s,\zeta)\in\comp}\mathcal{F}(s,\zeta)\le	\mathcal{F}(\sigma,\psi)=\mathcal{H}^2(S^+)=\frac12m_2(\Gamma).
\end{equation*}

Suppose now $m_2(\Gamma)=m_1(\Gamma_1)+m_1(\Gamma_2)$. Let $\Phi_j\in \mathcal{P}_1(\Gamma_j)$ be a solution to \eqref{plateau_1} and $S_j:=\Phi_j(\overline B_1)$ ($j=1,2$).   For $j=1,2$, let $D_j$ be the closed convex hull of $\Gamma_j$: clearly $D_1\cap D_2=\emptyset$. By Lemma \ref{lem:plateau_solution} each $S_j$ satisfies the following properties:
	\begin{itemize}
		\item $S_j\cap(\R^2\times\{0\})=\beta_j\subset D_j$ is a simple {smooth} curve joining $p_j$ to $q_j$;
		\item $S_j$ is symmetric with respect to $\R^2\times\{0\}$;
		\item $S_j^+:=S\cap\{x_3\ge0\}$ is the graph of  a function $\widetilde \psi_j\in W^{1,1}(U_j)\cap C^0(\overline U_j)$, where $U_j\subset D_j$ is the open region enclosed between $\partial^D\Om_j$ and $\beta_j$;
		\item $\beta_j$ is contained in $D_j$ and $F_j\setminus U_j$ is convex.
	\end{itemize}
	Let $(\sigma,\psi)\in\comp$ be given by 
	\begin{equation*}
		\sigma:=(\sigma_1,\sigma_2)\quad\text{and}\quad\psi:=\begin{cases}
			0&\text{in } \Om\setminus \{U_1\cup U_2\},\\
			\widetilde{\psi}_j&\text{in } U_j\, \text{ for }j=1,2,
		\end{cases}
	\end{equation*}
	where $\sigma_1([0,1]):=\overline{p_1q_2}$ and $\sigma_2([0,1]):=\beta_2\cup\,\overline{q_2p_1}\,\cup\beta_1$.
	Then $S^+:=S^+_1\cup S^+_2=\mathcal{G}_{\psi\res(\Om\setminus E(\sigma))}$ and 
	\begin{equation*}
		\min_{(s,\zeta)\in\comp}\mathcal{F}(s,\zeta)\le	\mathcal{F}(\sigma,\psi)=\mathcal{H}^2(S^+)=\frac12(m_1(\Gamma_1)+m_1(\Gamma_2))=\frac12m_2(\Gamma),
	\end{equation*}
		and the proof of step 1 is concluded.\\

	\step 2 
$2\min_{(s,\zeta)\in \comp}\mathcal F(s,\zeta)\ge m_2(\Gamma)$.

	Let $(\sigma,\psi)\in\comp$ be a minimizer satisfying
properties 
\ref{1.}-\ref{5.} of 
Theorem \ref{thm:regularity}.
	\\
	If $E(\sigma_1)\cup E(\sigma_2)=\emptyset$, by Step 1 we can apply Lemma \ref{lem_tec2} and find an injective  parametrization 
$\Phi\in \mathcal P_2(\Gamma)$ such that 
 $\Phi_i(\partial \openannulus)=\Gamma$ monotonically
	$\Phi(\overline\Sigma_{\rm ann})=\mathcal G_\psi\cup \mathcal G_{-\psi}$, and $$2\mathcal{F}(\sigma,\psi)=\int_{\Sigma_{\rm ann}}|\partial_{w_1}\Phi\wedge\partial_{w_2}\Phi|dw\geq m_2(\Gamma).$$
		If instead $E(\sigma_1)\cup E(\sigma_2)\ne\emptyset$, similarly we find injective  parametrizations 
$\Phi_1\in \mathcal P_1(\Gamma_1)$ and $\Phi_2\in \mathcal P_1(\Gamma_2) $  
such that $\Phi_j(\partial B_1)=\Gamma_j$ monotonically
for $j=1,2$, $\Phi_1(\overline B_1)\cup\Phi_2(\overline B_1)=\mathcal G_\psi\cup \mathcal G_{-\psi}$, and 
	$$2\mathcal{F}(\sigma,\psi)=\int_{B_{1}}|\partial_{w_1}\Phi_1\wedge\partial_{w_2}\Phi_1|dw+\int_{B_{1}}|\partial_{w_1}\Phi_2\wedge\partial_{w_2}\Phi_2|dw\geq m_1(\Gamma_1)+ m_1(\Gamma_2)\ge m_2(\Gamma).$$
	This concludes the proof.
\end{proof}

Now the proof of Theorem  \ref{thm:comparison-with-classical-plateau} is easily achieved.

\begin{proof}[Proof of Theorem \ref{thm:comparison-with-classical-plateau}] 

\ref{1plateau}. Let $\Phi\in\mathcal{P}_2(\Gamma)$, $S$, $S^+$, $S^-$ be as in the statement. 
By arguing as in the proof of  Theorem \ref{teorema-finale} we can find $(\sigma,\psi)\in\comp$ such that $S^\pm=\mathcal{G}_{\pm\psi\res (\Om\setminus E(\sigma))}$. 
Then by Theorem \ref{teorema-finale} we have

\begin{equation}\label{mesorotta}
\mathcal{F}(\sigma,\psi)=\frac12 m_2(\Gamma)= \min_{(s,\zeta)\in\comp}\mathcal{F}(s,\zeta)
\end{equation}
Hence $(\sigma,\psi)$ is a minimizer for $\mathcal F$ in $\mathcal W$; moreover by the properties of  $S$ it also satisfies properties \ref{1.}-\ref{5.} of Theorem \ref{thm:regularity}.

\ref{2plateau}.
Let $\Phi_j\in \mathcal{P}_1(\Gamma_j)$, $S_j$ for $j=1,2$, $S^+$, $S^-$ be as in the statement.  
Again arguing as in the proof of  Theorem \ref{teorema-finale}, we can find $(\sigma,\psi)\in\comp$ such that $S^\pm=\mathcal{G}_{\pm\psi\res(\Om\setminus E(\sigma))}$ and \eqref{mesorotta} holds, so that $(\sigma,\psi)$ is a minimizer of $\mathcal F$ in $\mathcal W$ satisfying properties \ref{1.}-\ref{5.} of Theorem \ref{thm:regularity}.

\ref{3plateau}. Let $(\sigma,\psi)\in\comp$ be a minimizer of $\mathcal F$ in $\mathcal W$ satisfying properties \ref{1.}-\ref{5.} of  Theorem \ref{thm:regularity}. Let also
	$$S:=\mathcal G_{\psi\res (\Om\setminus E(\sigma))}\cup\mathcal G_{-\psi\res (\Om\setminus E(\sigma))}.$$
	
	Suppose  $E(\sigma_1)\cap E(\sigma_2)=\emptyset$. Then 
there is $\Phi\in \mathcal{P}_2(\Gamma)$ which is a $\mathcal{MY}$ solution to \eqref{catenoid_plateau} such that $\immclosedann=S$:
indeed, to see this, it is sufficient to apply	Lemma \ref{lem_tec2}, since by Theorem \ref{teorema-finale} we have
		\begin{equation}\label{cd0}
		2\mathcal F(\sigma, \psi)= m_2(\Gamma).
	\end{equation}
\medskip
	
Suppose now $E(\sigma_1)\cap E(\sigma_2)\ne\emptyset$; then with a similar argument we can construct $\Phi_j\in\mathcal P_1(\Gamma_j)$ for $j=1,2$ solutions to \eqref{plateau:n1} such that $\Phi_1(\overline B_1)\cup \Phi_2(\overline B_1)=S$.
The proof is achieved.
\end{proof}

\subsection*{Acknowledgements}
We acknowledge the financial support of the GNAMPA of INdAM (Italian institute of high mathematics).
The work of R. Marziani was supported by the Deutsche Forschungsgemeinschaft (DFG, German Research Foundation) under the Germany Excellence Strategy EXC 2044-390685587, Mathematics M\"unster: Dynamics--Geometry--Structure.


\begin{thebibliography}{}
	
	\bibitem{AcDa:94} {\sc E. Acerbi and G. Dal Maso},
	\emph{New lower semicontinuity results for polyconvex integrals},
	Calc. Var. Partial Differential Equations \textbf{2} (1994), 329--371.
	%
	\bibitem{AFP}	{\sc L.~Ambrosio, N.~Fusco and D.~Pallara,}
	\emph{``Functions of Bounded Variation and Free Discontinuity Problems"},
	\newblock{Oxford Mathematical Monographs, The Clarendon Press Oxford University Press, New York}, 2000.
	%
	\bibitem{A}{\sc G.~Anzellotti},
	\emph{Pairings between measures and bounded functions and compensated compactness},
	\newblock{Ann. Mat. Pura Appl.}
	{\bf 135} (1983), 293--318.
	%
	\bibitem{AF}
	{\sc D.~Azagra and J.~Ferrera},
	\emph{Every convex set is the set of minimizers of some $C^\infty$-smooth convex function},
	\newblock Proc. Amer. Math. Soc. {\bf 130} (2002), 3687--3692.
	%
	
	\bibitem{BeElPaSc:19} {\sc G. Bellettini, A. Elshorbagy,
M. Paolini and R. Scala}, 
	\emph{On the relaxed area of the graph of discontinuous maps from the
		plane to the plane taking three values with no symmetry assumptions},
	Ann. Mat. Pura Appl. \textbf{199} (2019), 445--477.
	%
	\bibitem{vortex}{\sc G.~Bellettini, A.~Elshorbagy and R.~Scala}
	\newblock{The $L^1$-  relaxed area of the graph of the vortex map}, submitted. Preprint arXiv 2107.07236, https://arxiv.org/abs/2107.07236 (2021).
	%
	
	
	%
	
	\bibitem{BePa:10} {\sc G. Bellettini and M. Paolini},
	\emph{On the area of the graph of a singular map from the plane to the plane taking three values},
	Adv. Calc. Var. \textbf{3} (2010), 371--386.
	%
	\bibitem{Can-Sin}{\sc P.~Cannarsa and C.~Sinestrari}
	\newblock{``Semiconcave Functions, Hamilton-Jacobi Equations, and Optimal Control"}, Progress in Nonlinear Differential Equations and Their Applications, Vol. 58, Birkh\"auser, Boston-Basel-Berlin, 2004.
	\bibitem{CCF}
{\sc	M. Chleb\'ik, A. Cianchi and N. Fusco,}
	\emph{The perimeter inequality under Steiner
	symmetrization: cases of equality,}
	Ann. of Math. \textbf{162} (2005), 525--555.
	
	
	
	
	
	
	\bibitem{DHS}
	{\sc U.~Dierkes, S.~Hildebrandt and F.~Sauvigny},
	\newblock{``Minimal Surfaces''},
	\newblock{Grundlehren der mathematischen Wissenschaften},
	Vol. 339, Springer-Verlag, Berlin-Heidelberg, 2010.
	
	\bibitem{Federer}
	{\sc H.~Federer}, \newblock{``Geometric Measure Theory"}, 
	Die Grundlehren der mathematischen Wissenschaften, Vol. 153, Springer-Verlag, New York Inc., New York, (1969).
	
	
	\bibitem{Finn} {\sc R. Finn},
	\emph{Remarks relevant to minimal surfaces and to surfaces of constant mean curvature},
	J. Anal. Math. \textbf{14} (1965), 139--160.
		%
	
	
	\bibitem{GiMoSu:98} {\sc M. Giaquinta, G. Modica and J. Sou\u{c}ek}, 
	``Cartesian Currents in the Calculus of Variations I. Cartesian Currents'',
	Ergebnisse der Mathematik und ihrer Grenzgebiete, Vol. 37,
	Springer-Verlag, Berlin-Heidelberg, 1998.
	%
	\bibitem{GiMoSu:98_2} {\sc M. Giaquinta, G. Modica and J. Sou\u{c}ek}, 
	``Cartesian Currents in the Calculus of Variations II. Variational Integrals'',
	Ergebnisse der Mathematik und ihrer Grenzgebiete, Vol. 38,
	Springer-Verlag, Berlin-Heidelberg, 1998.
	
	%
	\bibitem{Giusti:84}
	{\sc E.~Giusti},
	\newblock{``Minimal Surfaces and Functions of Bounded Variation''},
	\newblock Birkh\"auser, Boston, (1984).
	%
	\bibitem{JS}
	{\sc H. Jenkins and J. Serrin},
	\emph{The Dirichlet problem for the minimal surface equation in higher dimension},
	J. Reine
	Ang. Math. \textbf{229} (1968), 170--187.
	%
	\bibitem{Krantz-Parks}
		{\sc G.~Krantz and R.~Parks},
		``Geometric Integration Theory'',
		Cornerstones, Birkhäuser Boston, Inc., Boston, MA, (2008).
	
	
	\bibitem{MY}
	{\sc W.~H.~Meeks and S.~T.~Yau},
	\emph{The classical Plateau problem and the topology of three-dimensional manifolds}, Topology {\bf 21} (1982), 409--440.
	%
	
	\bibitem{Nitsche:89} 
	{\sc J. C. C. Nitsche}, ``Lectures on Minimal Surfaces'', Vol. I,
	Cambridge University Press, Cambridge, (1989).
	
	\bibitem{Scala:19} {\sc R. Scala}, \emph{Optimal estimates for the 
		triple junction function and other
		surprising aspects of the area functional}, { Ann. Sc. Norm. Super. Pisa Cl. Sci.} 
	{\bf XX} (2020), 491--564.
	\bibitem{W}
	{\sc M.~D.~Wills},
	\emph{Hausdorff distance and convex sets},
	\newblock J. Convex Anal. {\bf 14} (2007), 109--117.
	\end{thebibliography}
	\end{document}